\documentclass[11pt,a4paper,twoside]{article}

\usepackage{fullpage}
\usepackage[utf8]{inputenc}
\usepackage[T1]{fontenc}
\usepackage{float}

\usepackage{amsmath, amsthm, amssymb, amsfonts}
\usepackage{thmtools, thm-restate}
\usepackage{bm}
\usepackage{mathtools}
\usepackage{graphicx}
\usepackage{subcaption}
\usepackage{tikz}
\usetikzlibrary{calc}
\usepackage{color}
\usepackage{caption}
\usepackage{xifthen}
\usepackage{soul}
\mathtoolsset{showonlyrefs} 

\DeclareFontFamily{U}{matha}{\hyphenchar\font45}
\DeclareFontShape{U}{matha}{m}{n}{
      <5> <6> <7> <8> <9> <10> gen * matha
      <10.95> matha10 <12> <14.4> <17.28> <20.74> <24.88> matha12
      }{}
\DeclareSymbolFont{matha}{U}{matha}{m}{n}

\DeclareFontFamily{U}{mathx}{\hyphenchar\font45}
\DeclareFontShape{U}{mathx}{m}{n}{
      <5> <6> <7> <8> <9> <10>
      <10.95> <12> <14.4> <17.28> <20.74> <24.88>
      mathx10
      }{}
\DeclareSymbolFont{mathx}{U}{mathx}{m}{n}

\DeclareMathDelimiter{\vvvert}{0}{matha}{"7E}{mathx}{"17}

\usepackage{stmaryrd}

\usepackage{dsfont}

\usepackage{calc}

\usepackage{algorithm,algorithmicx,algpseudocode}

\usepackage[textwidth=2cm]{todonotes}
\presetkeys{todonotes}{inline}{}

\usepackage{hyperref}
\hypersetup{
 pdfauthor={Wolfgang Dahmen, Felix Gruber and Olga Mula},
 pdftitle={An Adaptive Nested Source Term Iteration for Radiative Transfer Equations},
 pdfsubject={},
 pdfkeywords={DPG transport solver, iteration in function space,
 fast application of scattering operator, Hilbert--Schmidt decomposition, matrix compression, a posteriori bounds, kinetic problems, linear Boltzmann, radiative transfer}
}

\newtheorem{theorem}{Theorem}[section]
\newtheorem{lemma}[theorem]{Lemma}

\newtheorem{corollary}[theorem]{Corollary}

\newtheorem{prop}[theorem]{Proposition}

\theoremstyle{definition}

\newtheorem{remark}[theorem]{Remark}

\numberwithin{equation}{section}

\DeclareFontFamily{U}{jkpmia}{}
\DeclareFontShape{U}{jkpmia}{m}{it}{<->s*jkpmia}{}
\DeclareFontShape{U}{jkpmia}{bx}{it}{<->s*jkpbmia}{}
\DeclareMathAlphabet{\mathfrak}{U}{jkpmia}{m}{it}
\SetMathAlphabet{\mathfrak}{bold}{U}{jkpmia}{bx}{it}

\newcommand{\R}{\ensuremath{\mathbb{R}}}

\newcommand{\N}{\ensuremath{\mathbb{N}}}

\newcommand{\cA}{\ensuremath{\mathcal{A}}}
\newcommand{\cC}{\ensuremath{\mathcal{C}}}
\newcommand{\cG}{\ensuremath{\mathcal{G}}}
\newcommand{\cJ}{\ensuremath{\mathcal{J}}}
\newcommand{\cP}{\ensuremath{\mathcal{P}}}

\newcommand{\cQ}{\ensuremath{\mathcal{Q}}}

\newcommand{\cK}{\ensuremath{\mathcal{K}}}
\newcommand{\cI}{\ensuremath{\mathcal{I}}}
\newcommand{\cL}{\ensuremath{\mathcal{L}}}

\newcommand{\be}{\begin{equation}}
\newcommand{\ee}{\end{equation}}

\newcommand{\ord}{\ensuremath{\mathcal{O}}}

\newcommand{\e}{\varepsilon}

\newcommand{\Tr}{\mathcal{T}} 
\newcommand{\op}{\mathcal{B}} 

\newcommand{\hatop}[1]{\op_{#1}}
\newcommand{\hatTr}[1]{\Tr_{#1}}

\newcommand{\s}{\ensuremath{{\pvec{s}}}}
\newcommand{\n}{\ensuremath{\mathbf{n}}}

\def\bsig{\bar\sigma}

\newcommand{\kk}{G}

\newcommand{\id}{\ensuremath{\mathrm{id}}}
\newcommand{\tripnorm}[1]{\vvvert #1\vvvert}
\newcommand{\bound}[2][]{%
\ensuremath{\ifthenelse{\isempty{#1}}{b(#2)}{b_{#1}(#2)}}%
}
\newcommand{\ub}[1][]{\bound[#1]{u}}
\newcommand{\supp}{\operatorname{supp}}
\newcommand{\diam}{\operatorname{diam}}
\newcommand{\dist}{\operatorname{dist}}

\newcommand\restr[2]{{
  \left.\kern-\nulldelimiterspace 
  #1 
  \vphantom{\big|} 
  \right|_{#2} 
  }}

\newcommand{\ds}{\ensuremath{\mathrm d\s}}
\newcommand{\dt}{\ensuremath{\mathrm dt}}
\newcommand{\dx}{\ensuremath{\mathrm dx}}

\newcommand{\dGamma}{\ensuremath{\mathrm d\Gamma}}

\newcommand{\TT}{\mathbb{T}}              
\newcommand{\cell}{C}                            
\newcommand{\WW}{\mathbb{W}}          
\newcommand{\bkk}{\ensuremath{\bm{\mathrm{G}}}} 
\newcommand{\bg}{\mathbf{g}}               

\newcommand{\PP}{\mathbb{P}}

\newcommand{\uc}{\underline{c}}

\newcommand{\D}{\textsc{D}}           
\newcommand{\pD}{\mathfrak{P}}   

\renewcommand{\S}{\textsc{S}}        
\newcommand{\pS}{\mathfrak{S}}     

\newcommand{\V}{V}                      
\newcommand{\U}{U}                      

\newcommand{\uh}{{\underline{h}}} 
\newcommand{\uT}{{\underline{T}}} 

\newcommand{\skeluh}{\partial\pD_\uh}    
\newcommand{\skelh}{\partial\pD_h}         

\newcommand{\solve}{\textsc{N-ASTI}}
\newcommand{\cost}{\textbf{cost}}
\newcommand{\ASTI}{\textsc{ASTI}}

\newcommand{\dw}[1]{{\color{black}{#1}}}
\newcommand{\om}[1]{{\color{black}{#1}}}

\begin{document}

\title{An Adaptive Nested Source Term Iteration for Radiative Transfer Equations}
\author{Wolfgang Dahmen\thanks{%
  This research was supported by
  the  NSF Grant    DMS 1720297,
  and by the SmartState and Williams-Hedberg Foundation.
  },
  Felix Gruber and Olga Mula\thanks{Corresponding author: mula@ceremade.dauphine.fr}
}
\date{}
\maketitle

\begin{abstract}
We propose a new approach to the numerical solution of radiative transfer equations with certified
a posteriori error bounds \om{for the  $L_2$ norm}. A key role is played by stable \emph{Petrov--Galerkin} type variational formulations of parametric transport equations
and corresponding radiative transfer equations. This allows us to formulate an iteration in a suitable, infinite dimensional function
space that is guaranteed to converge with a fixed error reduction per step. The numerical scheme is then based on approximately
realizing this iteration within dynamically updated accuracy tolerances that still ensure convergence to the exact solution.
To advance this iteration two operations need to be performed within suitably tightened accuracy tolerances. First, the global scattering operator needs to be approximately applied to the current iterate within a tolerance comparable to the current
accuracy level. Second, parameter dependent linear transport equations need to be solved, again at the required accuracy of the iteration. To ensure that the stage dependent error tolerances are met, one has to employ \emph{rigorous a posteriori error bounds} which, in our case, rest on a \emph{Discontinuous Petrov--Galerkin} (DPG) scheme. These a posteriori bounds are not only crucial for guaranteeing the convergence of the perturbed iteration but are also used to generate adapted parameter dependent spatial meshes. This turns out to significantly reduce overall computational complexity. Since the global operator is only applied, we avoid the need to solve linear systems with densely populated matrices. Moreover, the approximate application of the global scatterer is accelerated through low-rank approximation and matrix compression techniques.
The theoretical findings are illustrated and complemented by numerical experiments with
 non-trivial scattering kernels.
\end{abstract}


\section{Introduction}
When dealing with problems giving rise to very complex  discretizations, one often tacitly \emph{assumes} that the
numerical output represents the corresponding continuous object reasonably well, without being, however,  able to actually quantify
output quality in any rigorous sense. Often  interest shifts then  towards accurately solving the (fixed) discrete problem
which by itself may indeed pose enormous challenges. Instead, the central objective of this article is  to put forward a
new algorithmic paradigm warranting \emph{error controlled computation}. By this we mean the deviation of the numerical result from
the \emph{exact  continuous solution} is certifiably quantified and set to meet a given target accuracy with respect to a
\emph{problem relevant norm}. It goes without saying that the ability to quantify the accuracy of forward simulations is a necessary
prerequisite of Uncertainty Quantification in general.  In this article we develop such methods for a regime of kinetic models, described below, for which to the best of our knowledge
error controlled schemes have so far not been available yet.

\subsection{Problem Formulation}
\label{ssec:1.1}
We consider certain \emph{kinetic models} describing the propagation of particles in a collisional medium modeling, e.\,g., heat transfer phenomena, neutron transport or medical imaging processes. We confine the subsequent discussion to simple \emph{monoenergetic radiative transfer models} which nevertheless exhibit the main obstructions to the design of efficient numerical methods for this problem class. Let $\D\subset \R^d$ be a bounded convex domain with piecewise $C^1$ boundary $\partial\D$, where $d\ge 1$.
Hence, for almost all $x\in \partial\D$ the outward normal $\n=\n(x)$ is well defined.
Furthermore, let $\S\subset \R^d$ denote the unit $(d-1)$-sphere
representing the directions in which particles propagate. Since we focus on the monoenergetic case, the particles have all the same kinetic energy (which we assume to be equal to 1) but note that more general compact sets describing the admissible transport velocity field are possible
and the subsequent developments generalize to a correspondingly wider scope of setups. In what follows, for $\s\in \S$
\begin{equation}
\label{Gammas}
\Gamma_-(\s) \coloneqq \{ x \in \partial \D \mid \s\cdot \n(x) < 0 \} \subset \partial\D,
\end{equation}
denotes the ``inflow-boundary'' for the given direction $\s$
while
\begin{equation}
\Gamma_- \coloneqq \{ (x,\s) \in \partial \D \times \S \mid \s\cdot \n(x) < 0 \} \subset \partial\D \times \S,
\end{equation}
denotes the inflow portion of the corresponding space-direction cylinder. The corresponding outflow boundary portions
$\Gamma_+(\s)$, $\Gamma_+$ are defined analogously.

Given non-negative data $f\colon \D \times \S \to \R_+$, $g\colon \Gamma_- \to \R_+$,
a cross section function $\sigma \colon \D\times \S \to \R_+$, and a collision kernel $K\colon \D \times \S \times \S\to \R_+$,
we want to find a function $u\colon \D \times \S \to \R_+$, satisfying
\begin{equation}
\begin{alignedat}{2}
\s\cdot\nabla u(x,\s) + \sigma(x,\s) u(x,\s)
- \int_\S K(x,\s',\s)u(x,\s')\,\ds' &= f(x,\s),\qquad&& \forall (x,\s) \in \D\times \S, \\
u &= g, \qquad &&\text{on $\Gamma_-$}. 
\end{alignedat}
\label{eq::radTrans}
\end{equation}
In the following, it will be useful to view the angular direction as a parameter and introduce
the abbreviations
\begin{align*}
(\Tr_{\s} u)(x) &\coloneqq \s\cdot\nabla u(x,\s) + \sigma(x,\s) u(x,\s), &
(\cK_{\s}u)(x) &\coloneqq \int_\S K(x,\s',\s)u(x,\s')\,\ds' ,
\end{align*}
for the pure transport and collision operator respectively. Splitting the transport part into
\begin{equation}
\label{Advection}
\Tr_{\s} = \cA_{\s} + \sigma \id, \quad \cA_{\s} v \coloneqq \s\cdot \nabla v,
\end{equation}
\eqref{eq::radTrans} can be written, for homogeneous boundary data $g \equiv 0$,
 as the operator equation
\begin{equation}
\label{opeq}
(\op u)(\cdot,{\s}) \coloneqq \Tr_{\s} u - \cK_{\s} u
= \cA_{\s} u + \sigma u -\cK_{\s} u = f(\cdot,\s).
\end{equation}
There is extensive literature addressing the solvability of \eqref{opeq} depending on the interrelation of the pair $(\sigma,K)$ usually known as the \emph{optical parameters}, see e.\,g.\ \cite{DL1993b,ES2014,Bal2009,Kharroubi1997}. One may roughly distinguish two ends of the
problem scope, namely the case of \emph{dominating scattering} near the diffusive limit (see e.\,g.\ \cite{GK2010}),
and the case of \emph{dominating transport}.
Here we restrict the subsequent considerations to the latter regime that is governed by at least weakly dominating
transport and possibly anisotropic scattering. The precise conditions on corresponding pairs of optical parameters
are discussed in a later section.

Note that when the kernel $K$ vanishes the \emph{pure transport problems}
\begin{equation}
\label{puretransport0}
\Tr_\s u =f,\quad u|_{\Gamma_-(\s)}= g,\,\,\s\in \S,
\end{equation}
may be viewed as a \emph{parametric family} of PDEs giving rise to the corresponding family of \emph{fiber solutions} $u_\s$, $\s\in \S$.
Alternatively---and this is necessary for the full problem \eqref{opeq}---we can view solutions $u(x,\s)$ as functions of the spatial variable $x\in \D\subset \R^d$
and the parametric variable $\s\in\S\subset \R^{d-1}$. It will therefore be important to identify a function space $U$ consisting of functions
over $\D\times \S$ for which \eqref{opeq} is well-posed in a sense to be made precise in Section \ref{ssec:stabvar}.

\subsection{Common Approaches and Main Obstructions}
There are at least two major groups of numerical strategies for approximately solving \eqref{opeq}, namely the \emph{method of moments} and
the \emph{discrete ordinates method} (DOM), see e.\,g.\ \cite{MY2008} and \cite{Kanschat2009,ACP2011,GS2011b,RGK2012} respectively. The method of moments
builds on (low order) polynomial projections in the
parameter domain and can be viewed as a model reduction. It seems to be rather difficult though to quantify the incurred model bias
and develop rigorous error bounds for the deviation of the approximate solution from the exact one. Also, the accuracy of polynomial
expansions suffers severely from low regularity.
DOM hinges on transport solves for sufficiently many direction parameters. These can serve as quadrature nodes for the approximate
application of the integral operator in combination with Jacobi type iterations to approximately solve the very large densely populated
linear systems. However, the convergence of this iteration in the discrete setting typically degrades with increasing dominance of the
scatterer \cite{Kanschat2009}.

The common approach is to \emph{first} discretize the (continuous) problem and then address the two---at first unrelated---issues:
a) how to solve the (fixed) discrete problem efficiently;
b) how to assess the accuracy attained by the solution of the discrete
   problem.

Modern strategies to face the complexity issues posed by a) concern the development of preconditioners or multigrid strategies
or employ
sparse tensor methods based on sparse grid or hyperbolic cross approximations. The former issue is impeded by the the fact
that on a fixed discrete level it is hard to respect \emph{intrinsic problem metrics} which play a central role in the current approach.
Moreover, the distinct lack of sufficiently strong stability notions accounts, in particular, for increasing recent efforts to incorporate additional \emph{structure preserving properties}
into discrete concepts. Simple examples are nonnegativity or mass conservation.

The viability and performance of sparse tensor methods, in turn, requires
  suitable \emph{a priori regularity assumptions} such as the validity of a certain order
of \emph{mixed smoothness}, see e.\,g.\ \cite{JP1983,Asadzadeh1988,Asadzadeh1998,GS2011b},
which are then also invoked to address b).

In general, variational formulations for parametric transport problems like \eqref{puretransport0} or \eqref{opeq}
are far less common than for elliptic problems. For instance \cite{GS2011b} considers least squares formulations minimizing residuals
in $L_2(\D\times\S)$. Corresponding trial spaces require anisotropic regularity of the solution depending explicitly
and sensitively on the transport direction. This may cause stability problems when the solution exibits shear discontinuities.
Alternatively, \cite{ES2012} proposes a mixed Galerkin formulation based on splitting the solution into symmetric and asymmetric parts.
This still fails to tightly relate errors to residuals which is a key pre-requisit for rigorous a posteriori error estimates.

We summarize now some of the intrinsic obstructions to an efficient and accuracy controlled
numerical solution of such problems.
\begin{enumerate}
\item \label{it:highDim}
The solution $u$ of \eqref{opeq} is a function of $2d-1$ variables (or even more in non-stationary cases and realistic models involving
energy levels). Hence, the problem is \emph{high-dimensional} and standard schemes become possibly prohibitively inefficient.
\vspace*{-1.5mm}
\item
\label{it:Kdense}
A nontrivial scattering kernel $K$ would give rise to densely populated
very large system matrices when using standard discretizations
based on localization only.
\vspace*{-1.5mm}
\item \label{it:lowRegularity}
These obstructions are aggravated by the fact that solutions exhibit in general only a low degree of regularity, in particular, when dealing with
highly concentrated and non-smooth boundary data. Standard \emph{a priori}
error estimates involving classical isotropic Sobolev regularity scales, often derived under unrealistic assumptions, are therefore not very useful for controlling accuracy.
\end{enumerate}
The primary objective
of this paper is  to address the above issues and  develop accuracy controlled schemes and corresponding stability notions.
We confine the discussion to \emph{stationary} problems but remark that the concepts carry over to time-dependent problems.
In fact, unsteady problems become conceptually easier as it will become clear later (aside from having to deal with even more variables).

The numerical results in Section \ref{sec:numres} indicate
that the proposed stability concept, closely intertwining the continuous and discrete setting, produces meaningful physical results  without explicitly imposing additional structure preserving measures.
\vspace*{-2mm}

\subsection{Conceptual Roadmap}\label{ssec:1.2}

The approach proposed in this paper is based on the following steps:
\begin{itemize}
\item[\textbf{(I)}] Identify a pair of Hilbert spaces $U$, $V$ over $\D\times \S$ for which \eqref{eq::radTrans} permits a \emph{stable variational formulation} (see
Section \ref{ssec:stabvar} for the precise meaning) where the (infinite-dimensional) trial space is to accommodate the solution of \eqref{eq::radTrans}.
Stability means that this variational formulation identifies the operator $\op$ in \eqref{opeq} as an \emph{isomorphism} from $U$ onto
the dual $V'$ of the (infinite-dimensional) test space $V$.
\item[\textbf{(II)}] Contrive an ``ideal outer iteration''
\begin{equation}
\label{form}
u_{n+1} = u_n + {\cal P}(f - \op u_n),\quad n=0,1,2, \ldots,
\end{equation}
that converges in $U$ to the unique solution $u$ of \eqref{eq::radTrans}.
\item[\textbf{(III)}] Realize each iteration step approximately  within dynamically updated  error tolerances that are
judiciously chosen so as to guarantee  convergence of the perturbed iteration to the exact
(infinite-dimensional) solution $u$ of \eqref{eq::radTrans}.
\end{itemize}
 Steps (I) and (II) require  analytic preparations which the numerical method is based upon while numerical aspects only enter in Step (III). The contributions of this paper culminate  in Theorem \ref{th:terminates}, which we informally state here as follows.
 \\[1.5mm]
 \textbf{Main Contribution:} {\it We contrive and theoretically justify a numerical algorithm that realizes Step (III) of the Roadmap and prove that
 for any target accuracy $\e >0$ it generates an approximate solution $u_\e$ of  \eqref{opeq} that deviates from the
 exact solution in $L_2(\D\times\S)$ by at most $\e$. Since the algorithm progresses from coarse to successively finer accuracy levels
 termination at any stage comes with a current error certificate.}\\[-2mm]

 This program relies on two
 points that guide the subsequent discussions. At \emph{no stage} is there ever formulated beforehand any fixed discrete problem
 but discretizations are formed \emph{adaptively} at each stage of the (perturbed) outer iteration \eqref{form}. For this to work it is crucial
 that the accuracy of a current approximate solution can be rigorously quantified.
 The perhaps closest relative to the above roadmap are adaptive wavelet methods along the lines of \cite{CDD2002}.
 However, these schemes rely essentially on \emph{symmetric variational formulations} of \emph{Galerkin} type
 and \emph{preconditioning} on the infinite dimensional level results from finding a Riesz basis for the energy space.
 In the present context suitable variational formulations turn out to be intrinsically \emph{unsymmetric}. In fact,
 obtaining suitable \emph{a posteriori error bounds},
 will be based on \emph{unsymmetric stable variational formulations} of \emph{Petrov--Galerkin} type for \eqref{eq::radTrans} and corresponding pure transport problems
\eqref{puretransport0}, see also \cite{DHSW2012}.
A central tool is the Banach--Nečas--Babuška Theorem that is briefly recalled in Section \ref{ssec:stabvar}.
\vspace*{-2mm}
\subsection{Layout}
In the remainder of this section we describe the organization and layout of the paper following the steps (I)--(III).
\begin{itemize}
\item[\textbf{ad (I)}]
Since, depending on the optimal parameters, solutions to \eqref{opeq} may exhibit discontinuities we opt to
choose $U\coloneqq L_2(\D\times \S)= L_2(\D)\otimes L_2(\S)$ as \emph{trial space}. For a variational formulation to be stable the
(infinite-dimensional) test space $V$ must then be different from $U$. As shown in Sections \ref{ssec:2.1}--\ref{ssec:varB}, for the regime of problems considered below a proper test space warranting stability is determined by the \emph{graph norm} of the pure transport operator
$\Tr$. Moreover, as a preparation for Step (II), we derive in Section \ref{ssec:contr} bounds for $\|\Tr^{-1}\cK\|_{\cL(U,U)}$
in terms of the optical parameters.
\item[\textbf{ad (II)}] With (I) at hand we identify in Section \ref{sec:strategy} (infinite-dimensional) \emph{preconditioners} $\cP \in \cL(V',U)$ that
warrant convergence of \eqref{form} in $U$ and render Step (III) practically viable.
In particular, we  identify two problem regimes of \emph{dominating transport} and \emph{dominating scattering}, depending on whether
$\Tr^{-1}\cK$ is a contraction in $\cL(U,U)$ or not, see Sections \ref{ssec:ideal} and \ref{ssec:domscat}.
\item[\textbf{ad (III)}] The remainder of the paper is devoted to Step (III). In Section \ref{sec:periter} we identify core routines needed
for the approximate realization of \eqref{form} as well as error tolerances  these routines need to meet in order to guarantee
convergence of the perturbed outer iteration to the exact solution. Again we have to distinguish first the two regimes of dominating
transport or scattering in Sections \ref{ssec:asti} and \ref{ssec:geone}, respectively, in order to formulate then the main algorithm in
Section \ref{ssec:conv} that covers both regimes.
\end{itemize}

We stress that one never has to invert a dense system involving a discretization of the
\emph{global operator} $\cK$. Instead an error-controlled application of $\cK$ is needed.
 While most numerical studies treat either local problems or simple kernels like
constants we make a point on including non-trivial scatterers. In Section \ref{sec:4} we present a scheme
based on \emph{Alpert wavelet representation} of $\cK$ and \emph{low-rank} approximations, see Section \ref{ssec:Alpert}.

As shown in Section \ref{sec:strategy}, the application of the preconditioner $\cP$ in \eqref{form} is ultimately reduced
to the error-controlled approximate inversion of the ``lifted'' pure transport operator $\Tr$ (acting on functions on $\D\times \S$, see \eqref{lifT}), discussed in Section \ref{sec:5}. This makes essential use of recent results from
\cite{BDS2017,DS2019} where rigorous sharp a posteriori error bounds for linear transport equations are derived for
Discontinuous Petrov--Galerkin (DPG) schemes.

\begin{remark} When progressing with the (perturbed) outer iteration, target accuracies   decrease step by step so that one starts
initially with very coarse DPG discretizations. The only linear systems to be solved in the course of such a \emph{nested iteration}
are the symmetric positive definite sparse DPG systems
for the spatial problems which are always kept as small as possible depending on the current target tolerances. The size of systems
that need to be inverted is always significantly smaller than the number of overall generated degrees of freedom.
\end{remark}

Finally, we present in Section \ref{sec:numres} some first numerical experiments as a proof of concept. They demonstrate,
in particular, the crucial role of adaptivity in the transport solver. In fact, the number of degrees of freedom shown in Figure
\ref{fig:transport-solutions} already for two spatial dimensions indicate that realizing the required error tolerances with
uniform spatial grids would be infeasible.

When the specific value of a constant does not matter we frequently employ the notation $a \lesssim b$ to express that $a$ is bounded by a fixed constant multiple of $b$
independent of all parameters $a$ and $b$ may depend on, that are not explicitly mentioned.

\section{Step (I)---Variational formulations and well-posedness}
\label{sec:preliminaries}
\subsection{Stability}\label{ssec:stabvar}
Our approach relies on appropriate \emph{variational formulations} of \eqref{opeq} which allow us to interpret \eqref{opeq}
as an operator equation
\begin{equation}
\label{opgen}
\op u = f,
\end{equation}
where $\op$ is induced by this variational formulation as a linear mapping from an infinite dimensional trial space $U$ to the dual $V'$ of some
(infinite-dimensional) test space $V$ (see Section \ref{ssec:2.1} (I)). Here the spaces $U$, $V$ host functions of both the spatial variables $x$ and the parametric variables $\s$.

Denoting by $\cL(X,Y)$ the space of all bounded linear operators from $X$ to $Y$,
the objective is then to establish well-posedness of \eqref{opgen} which means bounded invertibility of $\op$ or, more precisely,
boundedness of the \emph{condition number}
\begin{equation*}
\kappa_{U,V'}(\op)\coloneqq \Vert \op\Vert_{\cL(U,V')}\Vert \op^{-1}\Vert_{\cL(V',U)}.
\end{equation*}
Specifying the precise \emph{mapping properties} is therefore the central objective of this section.
The choice of the (Hilbert-)spaces $U$, $V$ tells us under which assumptions on the data,
a unique weak solution exists and in which norm the accuracy of approximate solutions is measured.

A well-known tool to be used in this context is the following result by Banach--Nečas--Babuška
which we recall for the convenience of the reader.

\begin{theorem}
\label{Th:BNB}
Assume that $q(\cdot,\cdot)\colon X \times Y\to \R$ is a bilinear form on the Hilbert spaces $X$, $Y$ \textup{(}with norms $\|\cdot\|_X$, $\|\cdot\|_Y$\textup{)}.
The validity of the following properties:
\begin{enumerate}
\item
$q(\cdot,\cdot)$ is continuous, i.\,e., there exists a $\bar C<\infty$ such that
\begin{equation}
\label{cont}
|q(w,z)|\le \bar C\|w\|_X \|z\|_Y,\quad w\in X, \, z\in Y;
\end{equation}
\item
there exists a $\uc >0$ such that
\begin{equation}
\label{infsup}
\inf_{w\in X}\sup_{z\in Y}\frac{q(w,z)}{\|w\|_X \|z\|_Y}\ge \uc ;
\end{equation}
\item \label{it:BNBinjectivity}
for each $z\in Y\setminus\{0\}$ there exists a $w\in X$ such that $q(w,z)\neq 0$;
\end{enumerate}
is equivalent to the solvability of the problem: given $f\in Y'$ find $u\in X$ such that
\begin{equation}
\label{opeqgen}
q(u,v)= \langle v,f\rangle,\quad v\in Y.
\end{equation}
Moreover, one has the stability relation
\begin{equation}
\label{stab}
\|u\|_X \le \uc^{-1}\|f\|_{Y'}.
\end{equation}
\end{theorem}
Note that condition \ref{it:BNBinjectivity} can be replaced by a second inf-sup condition \eqref{infsup} with the roles of $X$ and $Y$ interchanged.

Denoting by $\cal Q$ the operator from $X$ to $Y$ induced by $q(\cdot,\cdot)$, the above theorem says in particular that
\begin{equation}
\label{Qcond}
\kappa_{X,Y'}({\cal Q})\le \frac{\bar C}{\uc}.
\end{equation}

\subsection{Variational Formulations of the Pure Transport problem \texorpdfstring{(\ref{puretransport})}{the operator equation}}\label{ssec:2.1}

As indicated under \textbf{ad (I)} in Section \ref{ssec:1.2}, a crucial role is played by
 a suitable weak formulation for the pure transport equation 
\begin{equation}
\label{puretransport}
\s\cdot\nabla u(x,\s) + \sigma(x,\s) u(x,\s)
 = f(x,\s),\quad \mbox{for almost all } \, (x,\s) \in \D\times \S,
\end{equation}
defined on the phase space $\D\times\S$, where, in the following,
\begin{equation}
\label{sigmamax}
\sigma \ge 0, \quad\|\sigma\|_{L_{\infty}(\D\times \S)} < \infty.
\end{equation}
We consider first corresponding \emph{fiber} problems obtained by freezing the transport direction
$\s\in \S$. In favor of possibly low regularity requirements on the solution, we follow \cite{DHSW2012}. Formally applying integration by parts,
yields the variational problem
\begin{equation}
\label{second}
a(u,v;\s)
\coloneqq \int_{\D} u(\sigma(\cdot,\s) v - \s\cdot\nabla v)\,\dx
= -\int_{\partial\D}\n\cdot\s uv \,\dx + \int_\D f v \,\dx,
\end{equation}
for test functions $v$ from a suitable space yet to be determined. In fact,   the left hand side is now well-defined for
$u\in L_2(\D)$ and $v\in H(\s;\D)$,
where
\begin{equation}
\label{HsD}
H(\s;\D)\coloneqq \{v\in L_2(\D) \mid \s\cdot\nabla v\in L_2(\D)\}
\end{equation}
is a Hilbert space endowed with the norm
\begin{equation*}
\|v\|^2_{H(\s;\D)}\coloneqq \|v\|_{L_2(\D)}^2 + \|\s\cdot\nabla v\|^2_{L_2(\D)}.
\end{equation*}
 However, for $u\in L_2(\D)$ the trace on $\partial\D$ is not well-defined.
 Introducing the closed subspaces
\begin{equation}
\label{Hs0}
H_{0,\Gamma_{\pm}(\s)}(\s;\D)\coloneqq \operatorname{clos}_{\|\cdot\|_{H(\s;\D)} }\{ v\in C^1(\bar \D) \mid \restr{v}{\Gamma_{\pm}(\s)} =0\},
\end{equation}
and restricting the test functions to $H_{0,\Gamma_{+}(\s)}(\s;\D)$, the boundary integral on the right hand side of \eqref{second}
extends only over $\Gamma_-(\s)$. Thus, prescribing inflow boundary data $g\in L_2(\Gamma_-(\s), \n\cdot\s)$,
the weighted $L_2$ space on $\Gamma_-(\s)$ with weight $|\n\cdot\s|$, a weak formulation of \eqref{puretransport}
is to seek for
\begin{align}
\label{Us}
U(\s) = U &\coloneqq L_2(\D), & V(\s) &\coloneqq H_{0,\Gamma_+(\s)}(\s;\D),
\end{align}
$u=u(\s)\in U(\s) $ such that
\begin{align}
\label{seconda}
a(u,v;\s)& \coloneqq \int_{\D} u(\sigma(\cdot,\s) v- \s\cdot \nabla v)\,\dx \nonumber\\
&= \int_{\Gamma_-(\s)}\n\cdot\s gv \,\dx + \langle v,f\rangle \eqqcolon \langle v,F\rangle,\quad
v\in V(\s).
\end{align}
Here $\langle v,f\rangle = \langle v,f\rangle_{V,V'}$ stands for the dual pairing between $V$ and $V'$.
In particular, Dirichlet boundary conditions become natural boundary conditions which is an advantage
when the domain of the inflow boundary portion varies with $\s$ because they need not be incorporated in $U$.
In this setting, at least formally, the trial space $U$ is independent of $\s$ while the test space $V =V(\s)$ depends
essentially on $\s$.

The operator $\Tr_\s$ induced by $a(u,v;\s)$ through
\begin{equation}
\label{T2}
(\Tr_{\s} w)(v)= a (w,v;\s),\quad w\in L_2(\D),\, v\in H_{0,\Gamma_+(\s)}(\s;\D),
\end{equation}
defines a bounded linear operator from $L_2(\D)$ to $(H_{0,\Gamma_+(\s)}(\s;\D))'$. Accordingly, we have for
its (exact) adjoint
\begin{equation*}
\Tr_\s^*\in \cL(H_{0,\Gamma_+(\s)}(\s;\D),L_2(\D)),\quad
\langle w, \Tr_\s^* v\rangle = a(w,v;\s),\quad
w\in L_2(\D),\, v\in H_{0,\Gamma_+(\s)}(\s;\D).
\end{equation*}

Before addressing the invertibility of the operator $\Tr_\s$ we consider the ``lifted'' versions viewed as functions of $x$ and $\s$,
see \cite{DHSW2012}.
The role of $H(\s;\D)$ (see \eqref{HsD}) is now played by
 the space
\begin{equation}
\label{H}
H(\D\times \S)\coloneqq \{v\in L_2(\D\times \S) \mid \s\cdot\nabla v \in L_2(\D\times \S)\}.
\end{equation}
\dw{The space $H(\D\times \S)$} becomes a Hilbert space under the norm
\begin{equation}
\label{Hnorm2}
\|v\|_{H(\D\times \S)}^2 \coloneqq \int_{\D\times \S}\big( |\s\cdot\nabla v(x,\s)|^2 + |v(x,\s)|^2\big)\,\dx\,\ds,
\end{equation}
Likewise, the counterparts to the spaces \eqref{Hs0} are given by the closed subspaces
\begin{equation}
\label{Hs}
H_{0,\pm}(\D\times \S)\coloneqq \operatorname{clos}_{\|\cdot\|_{H(\D\times \S)}}\{v\in C^1(\overline{\D\times \S}) \mid v|_{\Gamma_\pm}=0\}.
\end{equation}
The ``lifted'' bilinear form
\begin{equation}
\label{lifta}
a (w,v)\coloneqq \int_\S a (w,v;\s)\,\ds
\end{equation}
allows us to define, in analogy to the above fiber versions, $\Tr$ by
\begin{equation}
\label{lifT}
\langle \Tr w,v\rangle = a (w,v),\quad w\in U ,\,\ v\in V ,
\end{equation}
where
\begin{align}
\label{U}
U &\coloneqq L_2(\D\times \S), & V &\coloneqq H_{0,+}(\D\times \S).
\end{align}
Thus, the variational problem: find $u\in U $ such that for any $f\in V'$
\begin{equation}
\label{varT}
a (u,v) = \langle v,f\rangle,\quad v\in V ,
\end{equation}
is equivalent to the operator equation
\begin{equation}
\label{opT}
\Tr u=f,
\end{equation}
where $\Tr$ is viewed as a mapping from $U $ into $V'$.

The invertibility of the fiber operators $\Tr_\s$ and the lifted version
$\Tr$ will be seen to be an immediate consequence of the following
norm-equivalences, see \eqref{Hnorm2}.

\begin{theorem}
\label{thm:NE}
Under the assumption \eqref{sigmamax} one has
\begin{equation}
\label{NE}
\begin{aligned}
\|\Tr^*_\s v\|_{L_2(\D)} &\sim \|v\|_{H(\s;\D)}, \quad &
v\in V(\s) &= H_{0,\Gamma_+(\s)}(\s;\D),\, \s\in \S;\\[2mm]
\|\Tr^* v\|_{L_2(\D\times \S)} &\sim \|v\|_{H(\D\times \S)}, \quad &
v\in V &= H_{0,+}(\D\times \S).
\end{aligned}
\end{equation}
as well as
\begin{equation}
\label{NE2}
\begin{aligned}
\|\Tr_\s v\|_{L_2(\D)} &\sim \|v\|_{H(\s;\D)}, \quad &
v\in V(\s) &= H_{0,\Gamma_-(\s)}(\s;\D),\, \s\in \S;\\[2mm]
\|\Tr v\|_{L_2(\D\times \S)} &\sim \|v\|_{H(\D\times \S)}, \quad &
v\in V &= H_{0,-}(\D\times \S).
\end{aligned}
\end{equation}
where the constants in the first line are independent of $\s\in \S$ and depend only on $\sigma_{\min},\ \sigma_{\max}$ and $\hat \ell = \diam(\D)$.
\end{theorem}

In principle, these results have been already shown in \cite{DHSW2012}.  We return to a proof in the next section in order  to exhibit
the dependence of involved constants from the optical parameters which will be needed for the numerical scheme.

As a consequence of Theorem \ref{thm:NE} we obtain the following results.

 \begin{corollary}
 \label{thm:invert}
 Assume that \eqref{sigmamax} holds.
 Then there exist constants $0< \uc, \bar C < \infty$ such that for $U(\s)$,
 $U$, $V(\s)$, $V$ defined by \eqref{Us}, \eqref{U}, respectively,
 \begin{align}
 \label{eq:transportbounds}
 \|\Tr_\s\|_{\cL(U(\s),V(\s))}, \|\Tr\|_{\cL(U,V)} &\le \bar C,
 &
 \|\Tr^{-1}_\s\|_{\cL(V(\s)',U(\s) )}, \|\Tr^{-1}\|_{\cL(V',U)} &\le \uc^{-1}.
 \end{align}
 Hence, the variational problems \eqref{seconda}, \eqref{varT}, respectively, have unique solutions that depend continuously on the data.
 \end{corollary}
\begin{proof}
First note that Theorem \ref{thm:NE} implies that
\begin{align}
\label{trip}
\tripnorm{v}_{\Tr^*_\s} &\coloneqq \|\Tr_\s^* v\|_{L_2(\D)}, &
\tripnorm{v}_{\Tr^* } &\coloneqq \|\Tr ^* v\|_{L_2(\D\times \S)},
\end{align}
are equivalent norms on $V(\s) = H_{0,\Gamma_+(\s)}(\s;\D)$, $V = H_{0,+}(\D\times \S)$, respectively. Endowing
$V(\s)$, $V$ with these norms, observe that
\begin{align}
\label{sup2}
\sup_{w\in U}\frac{a(w,v)}{\|w\|_{U}} &= \sup_{w\in U}\frac{\langle w,\Tr^* v\rangle}{\|w\|_{U}} = \|\Tr^*v\|_{U'}= \|\Tr^*v\|_{U}= \tripnorm{v}_{\Tr^*}.
\end{align}
Since by \eqref{NE2}, $\Tr$ is injective, and hence $\Tr^*$ is surjective, we obtain
 \begin{align}
\label{infsup2}
\sup_{v\in V}\frac{a(w,v)}{\tripnorm{v}_{\Tr^*}}
&= \sup_{v\in V}\frac{\langle \Tr w,v\rangle}{\tripnorm{v}_{\Tr^*}}
= \sup_{v\in V}\frac{\langle w, \Tr^* v\rangle}{\tripnorm{v}_{\Tr^*}}
\ge \frac{\|w\|_{L_2(\D\times\S)}^2}{\|w\|_{L_2(\D\times\S)}}
= \|w\|_{L_2(\D\times\S)} = \|w\|_{U},
\end{align}
which says that $\uc = \bar C =1$ and hence, by Theorem \ref{Th:BNB},
\begin{equation}
\label{condT}
\kappa_{U,V'}(\Tr)=1
\end{equation}
for $U$, $V$ as in \eqref{U}. The treatment of the fiber operators $\Tr_{\s}$ is completely analogous.
Hence,
 with the choice \eqref{trip} of norms \eqref{seconda} and \eqref{varT} are perfectly conditioned,
 i.\,e., the operators $\Tr_\s$, $\Tr$ are even \emph{isometries} between the respective pairs of spaces.
This completes the proof.
\end{proof}

\begin{remark}
\label{rem:s}
Later, both the fact that the fiber operators $\Tr_\s$ as well as the lifted versions $\Tr$ have bounded condition numbers will be
used in the envisaged numerical scheme.
\end{remark}

It will be useful to clearly distinguish the two above variational formulations\\[2mm]
\noindent
\textbf{Variational formulation (F1):} \textit{determined by the
combination of the bilinear form $a(\cdot,\cdot)$ from \eqref{seconda} with the pair of spaces $U$, $V$ it is supposed to
act on, namely
\begin{equation}
\tag{F1}
\label{F1}
\begin{aligned}
a(u,v) &= \int_{\D\times \S} u(x,\s)(\sigma(x,\s) v(x,\s)- \s\cdot \nabla v(x,\s))\,\dx \,\ds, \\[2mm]
U &= L_2(\D\times \S),\quad V = H_{0,+}(\D\times \S),
\end{aligned}
\end{equation}
}
\vspace*{2mm}
\noindent
\textbf{Variational formulation (F2):} \textit{determined by
\begin{equation}
\tag{F2}
\label{F2}
\begin{gathered}
 a(u,v;\s) \coloneqq \int_\D(\s\cdot\nabla u + \sigma(\cdot,\s)u)v \,\dx,\quad a(u,v)\coloneqq \int_{\S} a(u(\cdot,\s), v(\cdot,\s);\s) \,\ds,\\[2mm]
\begin{aligned}
 U(\s) &= H_{0,\Gamma_-(\s)}(\s;\D), \quad & V(\s) &= V =L_2(\D),\\[2mm]
 U &= H_{0,-}(\D\times \S),\quad & V &= L_2(\D\times \S).
\end{aligned}
\end{gathered}
\end{equation}
}
\vspace*{2mm}

Endowing $U = H_{0,-}(\D\times \S)$ with the norm $\tripnorm{w}_{\Tr}\coloneqq \|\Tr w\|_{L_2(\D\times \S)}$,
the same type of argument as in the proof of Theorem \ref{thm:invert} again combined with Theorem \ref{thm:NE} yields the following result,
see also \cite{DHSW2012}.

\begin{prop}
\label{rem:variant}
For data $f \in L_2(\D), L_2(\D\times \S)$, respectively the variational problems
\begin{align}
\label{varvariant}
a(u(\s),v;\s) &= \langle v,f\rangle, \quad v\in V(\s),\ \s\in \S, &
a(u,v) &= \langle v,f\rangle,\quad v\in V,
\end{align}
have unique solutions in $U(\s)$, $U$, defined by \eqref{F2}, respectively, which depend continuously on the data.
\end{prop}

\begin{remark}
The solutions in \eqref{varvariant} are required to have more regularity than in the first version \eqref{F1}, requiring, in particular,
that $f\in L_2(\D\times \S)$. Moreover, boundary conditions on $\Gamma_-(\s)$, $\Gamma_-$
are now essential boundary conditions that need to be built into the ansatz. Our interest in the formulation \eqref{F2} is a duality argument to be used later for the variational formulation of the full equation \eqref{opeq}.
\end{remark}

\subsection{Norm Equivalences}\label{ssec:2.2}

We establish next the norm equivalences in Theorem \ref{thm:NE}. \dw{As indicated earlier, a main reason for revisiting the proof  is
to prepare for Section \ref{ssec:contr} by determining the dependence of constants on the optical parameters .}
We use similar arguments as in \cite{ES2014}
(see also \cite{DHSW2012} for related discussions).

Let the time of escape of free moving particles from $\D$ be
\begin{equation}
\label{eq::ellDefinition}
\ell_\pm (x,\s) \coloneqq \inf\{ t>0 \mid x \pm t\s \notin \D\}.
\end{equation}
Then,
\begin{equation}
\ell(x,\s) \coloneqq \ell_-(x,\s) + \ell_+(x,\s)
\end{equation}
is the length of the longest line segment through $x$ in direction $\s$ completely contained in $\D$ and
\begin{equation}
\label{escape}
\hat\ell \coloneqq \sup_{(x,\s)\in \D \times \S} \ell(x,\s) = \diam(\D)
\end{equation}
is the maximum time of escape. For a given $\s\in\S$, we can express any $x\in \D$ in terms of characteristic coordinates as follows. Denoting $x_-(x,\s)\in \Gamma_-(\s)$ the intersection of the line $x+t\s$, $t\in \R$, with $\Gamma_-(\s)$, we can write
\begin{align}
\label{charcord}
x &= x_-(x,\s)+ \ell_-(x,\s) \s.
\end{align}
In these terms, define for $v\in L_2(\D\times \S)$ and almost every
$x= x_-(x,\s) + \ell_-(x,\s)\s\in \D$, $x_-(x,\s)\in \Gamma_-(\s)$
\begin{align}
w(x,\s)
&=
w\left(x_-(x,\s) + \ell_-(x,\s)\s,\s\right) \nonumber\\
&\coloneqq \int_0^{\ell_-(x,\s)} e^{-\int_r^{\ell_-(x,\s)} \sigma\left(x_-(x,\s)+\s\theta,\s\right)\,\mathrm d\theta}v\left(x_-(x,\s)+r\s,\s\right)\,\mathrm dr .\label{Tinverse}
\end{align}
One readily verifies that $w$ as well as $\Tr_{\s}w(\cdot,\s) = v(\cdot,\s)$ belong to $L_2(\D\times \S)$. Moreover
\begin{align}
\label{lowerw}
\|\Tr w\|^2_{L_2(\D\times\S)} &= \int_{\S}\|\Tr_{\s}w(\cdot,\s)\|^2_{L_2(\D)}\,\ds \le
C_1\|w\|^2_{H(\D\times \S)},
\end{align}
where $C_1$ depends on $\sigma_{\max}$,
where we abbreviate
\begin{align*}
\sigma_{\min} &\coloneqq \inf_{(x,\s)\in\D\times \S}\sigma (x,\s), &
\sigma_{\max} &\coloneqq \sup_{(x,\s)\in\D\times \S}\sigma (x,\s).
\end{align*}

We first derive a bound on $\Tr^{-1}$ as an operator mapping $L_2(\D\times \S)$ into itself.
\begin{lemma}
\label{lemma::Lcontinuous}
If $0\le \sigma \in L^\infty (\D \times \S)$, then $\Tr^{-1}$ is a continuous operator from $L^2 (\D \times \S)$ to $L^2 (\D \times \S)$ and
\begin{equation}
\label{contTminus}
\Vert \Tr^{-1}\Vert_{\cL(L^2(\D\times \S), L^2(\D\times \S))} \leq
\sqrt{ \hat{\ell} \dfrac{1-e^{-2\hat\ell \sigma_{\min} } }{ 2\sigma_{\min} } } \le \min\Big\{\hat \ell, \sqrt{\hat\ell/2\sigma_{\min}}\Big\}.
\end{equation}
Defining the formal adjoint of $\Tr$, by $\int_{\D\times\S}(\Tr^*v)w \,\dx\,\ds = \int_{\D\times\S}(-\s\cdot\nabla v+\sigma v)w\,\dx\,\ds$,
the same bound holds for $\Vert \Tr^{-*}\Vert_{\cL(L^2(\D\times \S), L^2(\D\times \S))}$.
\end{lemma}
\begin{proof}

For $v \in L_2(\D\times \S)$, we consider $w$ as defined in \eqref{Tinverse}. One readily checks that $w$ satisfies \eqref{varT} for $v=f$.
For $(x,\s) = (x_-+\ell_-(x,\s)\s,\s)$ and $0\leq \ell_-(x,\s) \leq \ell(x_-,\s)$, it follows from \eqref{Tinverse} and the Cauchy-Schwarz inequality
\begin{align*}
\left| w(x,\s) \right|^2
&\leq
\left( \int_0^{\ell(x_-,\s)} e^{-2\int_r^{\ell(x_-,\s)} \sigma(x_-+\theta \s,\s)\,\mathrm d\theta} \,\mathrm dr \right)
\left( \int_0^{\ell(x_-,\s)} v(x_-+r\s,\s)^2 \,\mathrm dr\right).
\end{align*}
Since
\begin{align*}
\int_0^{\ell(x_-,\s)} e^{-2\int_r^{ \ell(x_-,\s) } \sigma(x_-+\theta \s,\s)\,\mathrm d\theta} \,\mathrm dr
& \leq
\int_0^{\ell(x_-,\s)} e^{-2\int_r^{ \ell(x_-,\s) } \sigma_{\min} \,\mathrm d\theta} \,\mathrm dr
=
\dfrac{1-e^{-2\ell(x_-,\s)\sigma_{\min} }}{ 2\sigma_{\min} }\\
&
\leq
\dfrac{1-e^{-2\hat\ell\sigma_{\min} }}{ 2\sigma_{\min} },
\end{align*}
we derive
\begin{equation}
\label{eq::boundLu}
\left| w(x,\s) \right|^2
\leq
\dfrac{1-e^{-2\hat\ell\sigma_{\min} }}{ 2\sigma_{\min} }
\int_0^{\ell(x_-,\s)} |v(x_-+r\s,\s)|^2 \,\mathrm dr.
\end{equation}
Integrating \eqref{eq::boundLu} over $\D\times\S$,
\begin{align*}
\Vert w \Vert_{L_2(\D\times \S)}^2
&= \int_{(x_-,\s) \in \Gamma_-} \int_{t=0}^{\ell(x_-,\s)} \left| w (x_-+t\s,\s) \right|^2 |\s\cdot\n| \,\dt \,\dGamma_- \\
&\leq \dfrac{1-e^{-2\hat\ell\sigma_{\min} }}{ 2\sigma_{\min} } \int_{(x_-,\s) \in \Gamma_-} \int_{t=0}^{\ell(x_-,\s)}
 \int_{r=0}^{\ell(x_-,\s)} |v(x_-+r\s,\s)|^2 \,\mathrm dr |\s\cdot\n| \,\dt \,\dGamma_- \\
&\leq \hat \ell \dfrac{1-e^{-2\hat\ell\sigma_{\min} }}{ 2\sigma_{\min} } \Vert v \Vert^2_{L_2(\D\times \S)},
\end{align*}
where we have used that $\int_{t=0}^{\ell(x_-,\s)} \,\mathrm dt \leq \hat\ell$ for all $(x_-,\s) \in \Gamma_-$ to derive the last bound. This yields the first bound for $\Vert \Tr^{-1} \Vert_{\cL(L^2(\D\times \S), L^2(\D\times \S))}$ given in \eqref{contTminus}. The second bound follows directly from the fact that $\hat \ell (1-e^{-2\hat\ell\sigma_{\min}})/(2\sigma_{\min}) \leq \min \left\{ \hat \ell^2,\hat \ell/(2\sigma_{\min}) \right\}$ since $1-e^{-x}\leq \min\{1,x\}$ for any $x\geq 0$. The argument for $\Tr^*$ is the same.
\end{proof}

\noindent
\begin{proof}[Proof of Theorem \ref{thm:NE}]
The inequality \eqref{contTminus} says that
\begin{equation}
\label{lowerT}
\|v\|_{L_2(\D\times\S)}
\le \min\Big\{\hat \ell, \sqrt{\hat\ell/2\sigma_{\min}}\Big\}
\begin{cases}
\|\Tr  v\|_{L_2(\D\times\S)}, & v\in H_{0,-}(\D\times \S),\\[2mm]
\|\Tr^*  v\|_{L_2(\D\times \S)}, & v\in H_{0, +}(\D\times \S),
\end{cases} .
\end{equation}
Integrating \eqref{eq::boundLu} only over $x\in \D$ leads to analogous statements for the fibers $\Tr_{\s}$, $\Tr^*_{\s}$, namely
\begin{equation}
\label{lowers}
\|v\|_{L_2(\D)} \le \min\Big\{\hat \ell, \sqrt{\hat\ell/2\sigma_{\min}}\Big\}
\begin{cases}
\|\Tr_{\s} v\|_{L_2(\D)}, & v\in H_{0,\Gamma_-(\s)}(\s;\D),\\[2mm]
\|\Tr^*_{\s} v\|_{L_2(\D)}, & v\in H_{0,\Gamma_+(\s)}(\s;\D),
\end{cases} \quad \s\in \S.
\end{equation}
We infer from \eqref{lowerT} that, for instance,
\begin{align*}
\|\Tr_\s v\|_{L_2(\D )}&\le \|\s\cdot\nabla v\|_{L_2(\D )} + \sigma_{\max}\|v\|_{L_2(\D )} \le (1+ \sigma_{\max}^2)^{1/2}\|v\|_{H(\s;\D)}.
\end{align*}
Conversely, one has
\begin{align}
\label{conversely}
\|v\|_{H(\s;\D)}&\le
 \|\s\cdot\nabla v\|_{L_2(\D )} + \|v\|_{L_2(\D )}
\le \|\Tr_\s v\|_{L_2(\D )} + (1+\sigma_{\max})\|v\|_{L_2(\D )}\nonumber\\
&\le \Big(1 + (1+\sigma_{\max}) \min\Big\{\hat \ell, \sqrt{\hat\ell/2\sigma_{\min}}\Big\}\Big)  \|\Tr_\s v\|_{L_2(\D )} .
\end{align}
The remaining assertions of Theorem \ref{thm:NE} are derived analogously.
\end{proof}

\begin{remark}
\label{rem:smaller}
$\Vert \Tr^{-1}\Vert_{\cL(L^2(\D\times \S), L^2(\D\times \S))}$ is small
when either $\diam(\D)$ is small or when $\sigma_{\min}$ is large relative to $\hat\ell$.
\end{remark}
\subsection{Variational Formulation of the Radiative Transfer Problem \texorpdfstring{(\ref{opeq})}{the operator equation} }
\label{ssec:varB}

Throughout this section we let $g\equiv 0$, i.\,e., we treat homogeneous inflow boundary conditions.
Also, we assume that the kernel $K$ satisfies
\begin{equation}
\label{K}
K(x,\s,\s')\ge 0 \quad (x,\s,\s')\in \D\times \S \times \S,\quad K\in L_\infty (\D;L_2(\S\times \S))\subset L_2(\D\times \S\times \S)
\end{equation}
so that we have
\begin{equation}
\label{Kbounded}
\cK, \, \cK^* \in \cL(L_2(\D\times \S),L_2(\D\times \S)).
\end{equation}
Following the same lines as before for the pure transport operator $\Tr$
we can define the operator $\op$ by
\begin{equation}
\label{B}
b (w,v) =\langle \op w,v\rangle
\coloneqq \int_{\S} a (w(\cdot,\s),v(\cdot,\s);\s)\,\ds - k (w,v),
\quad \forall \,w \in U, \, v\in V,
\end{equation}
where $k(w,v) = \langle\cK w,v\rangle$, and the spaces $U$, $V$ are
chosen according to the formulations \eqref{F1}, \eqref{F2}, respectively.

A key property in what follows is \emph{accretivity} of $\op$. In the present context this means that there exists some
positive $\alpha$ such that
\begin{equation}
\label{accretive}
(\op v,v)\ge \alpha \|v\|^2_{L_2(\D\times \S)},\quad v\in H_{0,-}(\D\times \S).
\end{equation}
We postpone for a moment listing conditions on the optical parameters which imply \eqref{accretive} but
present first the central result in this section.

\begin{theorem}
\label{thm:Binverse}
Assume that \eqref{K} and \eqref{accretive} hold.
Then, for either one of the two formulations \eqref{F1}, \eqref{F2} and any $f\in V'$ the problem: find $u\in U$ such that
\begin{equation}
\label{varprobB}
b(u,v) = \langle f,v\rangle,\quad v\in V,
\end{equation}
has a unique solution satisfying
\begin{equation}
\label{opstab}
\|u\|_{U } \lesssim \|f\|_{V'},
\end{equation}
with constants depending only on the optical parameters.

The operator $\op$, defined by \eqref{B} is in either setting a linear norm-isomorphism from
$U $ onto $V'$, i.\,e., has a finite condition $\kappa_{U,V'}(\op) <\infty$.
\end{theorem}

The proof makes use of the following norm equivalences.

\begin{lemma}
\label{lem:BT}
Let $\Tr'$, $\op'$ denote the formal adjoints of $\Tr$, $\op$, respectively.
Then, under the assumptions \eqref{ass}, \eqref{K} on $\sigma$ and $K$ one has
\begin{equation}
\label{BH}
\begin{array}{ll}
\|w\|_{H(\D\times\S)}\sim \|\op w\|_{L_2(\D\times \S)} \sim \|\Tr w\|_{L_2(\D\times \S)}
, & w\in H_{0,-}(\D\times \S),\\[2mm]
\|w\|_{H(\D\times\S)}\sim \|\op' w\|_{L_2(\D\times \S)} \sim \|\Tr' w\|_{L_2(\D\times \S)}
, & w \in H_{0,+}(\D\times \S),
\end{array}
\end{equation}
where
the constants depend on the optical parameters.
\end{lemma}
\noindent
\begin{proof}[Proof of Lemma \ref{lem:BT}]
By \eqref{Kbounded},  we have for some constant $C_1$
\begin{equation}
\label{upperT}
\|\op w\|_{L_2(\D\times \S)}\le \|\Tr w\|_{L_2(\D\times \S)} + C_1\|w\|_{L_2(\D\times \S)}
\le (1+ C_1C_2) \|\Tr w\|_{L_2(\D\times \S)} ,
\end{equation}
where we have used \eqref{lowerT} in the last step. Conversely, again by \eqref{Kbounded},  \eqref{accretive},
and using Young's inequality yields
\begin{align*}
\|\Tr w\|_{L_2(\D\times \S)} &\le \|\op w \|_{L_2(\D\times \S)} + \|\cK w\|_{L_2(\D\times \S)}\nonumber\\
&\le \|\op w \|_{L_2(\D\times \S)} +C_1\|w\|_{L_2(\D\times \S)}\nonumber\\
&\le \|\op w \|_{L_2(\D\times \S)} +\frac{C_1}{\sqrt{\alpha}} (\op w,w)^{1/2}\nonumber\\
&\le \|\op w \|_{L_2(\D\times \S)} +\frac{C_1}{\sqrt{\alpha}}\Big(\frac{\|\op w\|_{L_2(\D\times \S)}}{2\delta}+ \delta\|w\|_{L_2(\D\times \S)}\Big)\nonumber\\
&\le \|\op w \|_{L_2(\D\times \S)} +\frac{C_1}{\sqrt{\alpha}}\Big(\frac{\|\op w\|_{L_2(\D\times \S)}}{2\delta}+ \delta C_2\|\Tr w\|_{L_2(\D\times \S)}\Big).
\end{align*}
where $C_2= \min\Big\{\hat \ell, \sqrt{\hat \ell/\sigma_{\min}}\Big\}$ is the constant from \eqref{lowerT}. Choosing $\delta$ small enough
to ensure that $C_1C_2\delta/\alpha <1$, the relation $\|\op w\|_{L_2(\D\times \S)} \sim \|\Tr w\|_{L_2(\D\times \S)}$ follows.
The first line in \eqref{BH} follows then from Theorem \ref{thm:NE} proving
the assertion   for $\op$. The argument for $\op'$ is analogous.
\end{proof}

We are now in position of proving Theorem \ref{thm:Binverse}.

\begin{proof}[Proof of Theorem \ref{thm:Binverse}]
First, under the given assumptions we clearly have for either formulation \eqref{F1} or \eqref{F2}
with respective pairs $U,\,V$, that $\op$ is bounded
\begin{equation*}
\op \in \cL(U,V').
\end{equation*}
Then, it  follows from Theorem \ref{Th:BNB} and \eqref{accretive} that under the above assumptions
\begin{equation}
\label{L2B}
\|\op^{-1}\|_{\cL(L_2(\D\times\S),L_2(\D\times\S))} \le \alpha^{-1}.
\end{equation}
 To prove the last statement of the theorem note that
in view of \eqref{BH}, injectivity of $\Tr$ and $\Tr'$ implies
injectivity of $\op$ and $\op'$. Suppose $\op$ were not surjective.
Then there exists a $w_0 \neq 0$ in $L_2(\Omega)$ such that
$\langle\op w, w_0\rangle = 0$ for all $w\in H_{0,-}(\Omega)$.
By boundedness of $\op$ and denseness of $H_{0,-}(\Omega)$ in
$L_2(\Omega)$, this leads to a contradiction to \eqref{accretive}.
We can argue in the same way for $\op'$ to conclude that $\op$ and $\op'$
are bijections for their respective pairs of spaces.
This holds by duality, since $(\op')^*$ agrees with $\op$ as a mapping from
$L_2(\Omega)$ to $(H_{0,-}(\Omega))'$. In view of Lemma~\ref{lem:BT},
the proof of Theorem~\ref{thm:Binverse} can now be completed with the
aid of Theorem~\ref{Th:BNB} in exactly the same way as the proof of
Theorem~\ref{thm:invert}.
\end{proof}

When the specific choice of the settings \eqref{F1} or \eqref{F2} is clear from the context, we view \eqref{varprobB}
as an operator equation
\begin{equation*}
\op u = f
\end{equation*}
with data $f$ in the respective dual space $V'$.

We discuss next two general conditions on the optical parameters that entail \eqref{accretive}.
Defining the kernel averages
\begin{equation}
\label{eq::scatteringCrossSections}
\bsig (x,\s) \coloneqq \int_\S K(x,\s,\s') \,\ds'
,\quad \text{and} \quad
\bsig'(x,\s) \coloneqq \int_\S K(x,\s',\s) \,\ds',
\end{equation}
a first frequently studied general  class of optical parameters is signified by the fact that there exist $0< \alpha, M_a<\infty$ such that for all
$(x,\s)\in \D\times \S$,
\begin{equation}
\label{ass}
\sigma(x,\s) - \bsig(x,\s) \ge \alpha,
\quad \sigma(x,\s) - \bsig'(x,\s) \ge \alpha,
\quad \bsig(x,\s) \le M_a,
\quad \bsig'(x,\s) \le M'_a.
\end{equation}
Note that this implies that the absorption coefficient $\sigma$ is not  allowed to vanish in $\D$.
For this class we recall the following well-known result (see e.\,g.\ \cite[Chapter XXI, §2, Theorem 4]{DL1993b}).

\begin{prop}
\label{prop1}
 {
If $\sigma$ and $K$ satisfy assumptions \eqref{K} and \eqref{ass}, then the operator $\op$ is \emph{accretive}, i.\,e., for any $v\in H_{0,-}(\D\times \S)$,
\begin{equation*}
(\op v,v)\ge \alpha \|v\|^2_{L_2(\D\times \S)},
\end{equation*}
where the constant $\alpha$ is the one appearing in \eqref{ass}.
}
\end{prop}
For the convenience of the reader we sketch the simple argument.
 It follows from conditions \eqref{ass}, \eqref{K} that
$(\sigma v- \cK v,v) \ge \alpha \|v\|_{L_2(\D\times \S)}^2$
 {on $L_2(\D\times\S)$},
which, combined with the accretivity of $\cA$ on $H_{0,-}(\D\times\S)$ , defined by $\langle\cA w,v\rangle = \int_{\D\times \S} \s\cdot \nabla w(x,\s)v(x,\s)\,\dx\,\ds$,  i.\,e., $(\cA v,v) \ge 0$  for all $v \in H_{0,-}(\D\times\S)$,
yields the conclusion.\\

We emphasize that condition \eqref{ass} is not necessary for \eqref{accretive} to hold as can be seen from the
following class of frequently used  kernels with slightly more specified structure.
Consider
\begin{equation}
\label{ker1}
K(x,\s,\s') = \kappa (x)\kk (\s,\s'), \quad \kk (\s,\s')=\kk(\s',\s),\quad \kk(\s,\s')\ge 0,\quad \s,\s'\in\S,\quad \kappa \ge \kappa_0>0,
\end{equation}
with the normalization
\begin{equation}
\label{normaliz}
\int_{\S }\kk(\s,\s')\,\ds'  = \int_{\S }\kk(\s,\s')\,\ds
= 1, \quad \s,\, \s'\in \S.
\end{equation}
Once the integral over one argument is a constant, this latter relation can always be realized by rescaling $\kappa$.
Assuming always that $\ds$ is the Haar measure, it also follows that $\int_{\S\times\S}\kk(\s,\s')\,\ds\,\ds' = 1$.
Moreover, we split
\begin{equation}
\label{split}
\sigma = \sigma_a + \kappa,
\end{equation}
where $\sigma_a\ge 0$ is the so-called absorption coefficient. Hence in this case  $\sigma(x,\s)- \bar\sigma(x,\s)=\sigma(x,\s)-\bar\sigma'(x,\s)= \sigma_a(x,\s)$ so that \eqref{ass} does not hold whenever $\sigma_a$ vanishes somewhere in $\D$.
On the other hand, let $\cC_+\subset L_2(\D\times \S)$ the cone of non-negative functions in $L_2(\D\times\S)$ (in the weak sense) and define
\begin{equation*}
\cK_0 v \coloneqq \int_\S \kk(\cdot,\s')v(\s')\,\ds'.
\end{equation*}
Under the above conditions the largest eigenvalue of $\cK_0$ is one, it is simple and has the constant as the corresponding eigenfunction.
Therefore,
\begin{equation*}
\sup\,\bigl\{(v,\cK_0 v) \bigm\vert v\in \cC_+ \cap H_{0,-}(\D\times\S),\,\|v\|_{L_2(\D\times \S)}=1\bigr\} \eqqcolon \beta <1.
\end{equation*}
Thus, the accretivity condition \eqref{accretive} holds with
\begin{equation*}
\alpha \ge (\sigma_a)_{\min}+ \kappa_0 (1-\beta).
\end{equation*}
 which is strictly larger than zero even if the absorption coefficient vanishes in $\D$.

\vspace*{2mm}
In principle, one could base a numerical method on both formulations \eqref{F1}, \eqref{F2}, where the latter one would
seek approximations in a stronger norm. However, in what follows we focus on the setting \eqref{F1} where the solution is sought in $U=L_2(\D\times\S)$
and where boundary conditions are natural ones.

\dw{
\begin{remark}
\label{rem:Neuman}
There is of course an alternate way of establishing bounded invertibility of $\op\in \cL(U,V')$ whenever the condition
\begin{equation}
\label{contr2}
\|\Tr^{-1}\cK\|_{\cL(U,U)}\le \rho <1
\end{equation}
holds. While continuity of $\op$ is immediate, a straightforward Neuman-series argument shows that then
\begin{equation*}
\|\op^{-1}\|_{\cL(V',U)} \le (1-\rho)^{-1} \|\Tr^{-1}\|_{\cL(V',U)}.
\end{equation*}
We refer to the regime of problems where \eqref{contr2} is valid as the \emph{weakly transport dominated case}.
\end{remark}
In addition condition \eqref{contr2} will be seen to be  crucial for the identification of preconditioners $\cP$ in the idealized iteration \eqref{form}.
We therefore address the derivation of bounds for $\|\op^{-1}\|_{\cL(V',U)}$ in the next section.}

\subsection{Contractivity of \texorpdfstring{$\Tr^{-1}\cK$}
            {transport solver after kernel application}}\label{ssec:contr}

We begin with the following result taken from \cite[Chapter XXI, §2, Lemma 1]{DL1993b}.
\begin{prop}
\label{prop:K}
Assume that \eqref{K} and \eqref{ass} hold. Then
$\cK$ maps $L_2\coloneqq L_2(\D\times \S)$ boundedly into itself, with
\begin{equation}
\label{Kbound}
\|\cK\|_{\cL(L_2,L_2)}\le \big(M_a M_a')^{1/2},
\end{equation}
where $M_a$, $M_a'$ are the constants from \eqref{ass}. Moreover, $\cK$ maps $L_2^+(\D\times\S)$, the cone of non-negative functions in $L_2(\D\times\S)$, into itself.
\end{prop}

\dw{To specify bounds for the operator norm $\|\Tr^{-1}\cK\|_{\cL(U,U)}$ we introduce the quantities}
\begin{align}
\label{gamma}
\gamma &\coloneqq \sup_{(x,\s)\in\D\times \S}\Big\{ \frac{\bar\sigma(x,\s)}{\sigma(x,\s)}, \frac{\bar\sigma'(x,\s)}{\sigma(x,\s)}\Big\},
&
\zeta &\coloneqq \frac{\gamma\sigma_{\max}}{\sigma_{\min}}.
\end{align}

\begin{lemma}
\label{la:normT-1K}
Under assumptions \eqref{ass} on the optical parameters,
\begin{equation}
\label{TK1}
\|\Tr^{-1}\cK\|_{\cL(U,U)}
\le
\min \left\{ \zeta, (\sigma_{\max}-\alpha)/\sigma_{\min}, (M_aM_a')^{1/2} \min\left\{\hat \ell, \sqrt{\hat\ell/2\sigma_{\min}}\right\} \right\}.
\end{equation}
\end{lemma}
\begin{proof}
Combining \eqref{Kbound} and \eqref{contTminus} yields that
\begin{equation*}
\|\Tr^{-1}\cK\|_{\cL(U,U)}
\leq (M_aM_a')^{1/2} \min\Big\{\hat \ell, \sqrt{\hat\ell/2\sigma_{\min}}\Big\}.
\end{equation*}
To prove that $\|\Tr^{-1}\cK\|_{\cL(U,U)} \leq \min\{\zeta,(\sigma_{\max}-\alpha)/\sigma_{\min}\}$, we proceed as follows. For any $\varphi\in L_2(\D\times\S)$ we have $\cK\varphi\in L_2(\D\times\S)$ so that there exists a unique $w\in H_{0,-}(\D\times\S)$ such that
$\Tr w = \cK\varphi$. Thus, it suffices to prove that $\|w\|_{L_2(\D\times\S)} \le \min\{\zeta,(\sigma_{\max}-\alpha)/\sigma_{\min}\} \|\varphi\|_{L_2(\D\times \S)}$.
Since $\cA$ is accretive on $H_{0,-}(\D\times\S)$, we have
\begin{equation}
\label{Klower}
(\cK \varphi,w) = (\cA w,w) +(\sigma w, w) \geq (\sigma w,w) \geq \sigma_{\min} \Vert w \Vert^2_{L^2(\D\times \S)}.
\end{equation}
Furthermore,
\begin{align}
(\cK \varphi,w)
&\leq \int_{\D\times \S \times \S} \vert w(x,\s) \vert K(x,\s',\s) \vert \varphi(x,\s') \vert \,\dx\,\ds\,\ds' \nonumber\\
&\leq \left( \int_{\D\times \S} \vert w(x,\s) \vert^2 \bar\sigma'(x,\s) \,\dx\,\ds\right)^{1/2} \left( \int_{\D\times \S} \vert \varphi (x,\s') \vert^2 \bar\sigma(x,\s') \,\dx\,\ds' \right)^{1/2} \nonumber\\
&\leq \min \{ \sigma_{\max}-\alpha,\gamma \sigma_{\max}\} \Vert w \Vert_{L^2(\D\times \S)} \Vert \varphi \Vert_{L^2(\D\times \S)} \label{eq::contractionIneq2}
\end{align}
where we have used Cauchy--Schwarz' inequality. Combining this with \eqref{Klower}
 yields the desired inequality $\Vert w \Vert_{L_2(\D\times \S)} \leq \min\{\zeta,(\sigma_{\max}-\alpha)/\sigma_{\min}\} \Vert \varphi \Vert_{L_2(\D\times \S)}$.
\end{proof}
It follows from \eqref{TK1} that having
\begin{equation}
\label{eq:condContraction}
\min \left\{ \zeta, (\sigma_{\max}-\alpha)/\sigma_{\min}, (M_aM_a')^{1/2} \min\left\{\hat \ell, \sqrt{\hat\ell/2\sigma_{\min}}\right\} \right\} < 1
\end{equation}
is a sufficient condition for $\Tr^{-1}\cK$ to be a contraction. From this we can distinguish two different ``physical regimes'' that ensure contractivity:
\begin{itemize}
\item having $\zeta < 1$ or $(\sigma_{\max}-\alpha)/\sigma_{\min} <1$ can be interpreted as quantifying the dominance of transport with respect to scattering with $\sigma(x,\s)$ not varying
too much in its arguments.
This condition is a quantification of the well-known fact that DOM converges at a slower rate when collisions become more and more significant with respect to transport.
\item having $(M_aM_a')^{1/2} \min\left\{\hat \ell, \sqrt{\hat\ell/2\sigma_{\min}}\right\}<1$
happens when $\hat \ell =\diam(\D)$ is sufficiently small or $\sigma_{\min}/M_aM_a'$ sufficiently large, which is another expression to quantify how much transport effects dominate with respect to the scattering.
\end{itemize}
Of course, these conditions cannot be expected to hold in all relevant application scenarios.
However, they are going to play a crucial role in what we call \emph{preconditioning} on the continuous level,
ensuring convergence in the infinite dimensional continuous case.

\section{Step (II)---Idealized Iterations}
\label{sec:strategy}
\dw{We are now prepared to identify viable outer iterations of the form
\begin{equation}
\label{geniter}
u_{n+1} = u_n + \cP(f- \op u_n),\quad n=0,1,2,\ldots ,
\end{equation}
(see Step (II) in Section \ref{ssec:1.2}). In the following, we will work with
the pair of trial and test spaces $U,\,V$, given in \eqref{F1}, that is
\begin{align*}
U &= L_2(\D\times \S), &
V &= H_{0,+}(\D\times \S),
\end{align*}
where we abbreviate in what follows $\|v\|_V\coloneqq \|v\|_{H(\D\times \S)}$.
Of course, the \emph{preconditioner}
 $\cP\in \cL(V',U)$ is a  to be chosen in such a way that
\begin{equation}
\label{errorred}
\exists\, \rho <1\,\, \mbox{such that }\,\, \|u_{n+1}-u\|_U \le \rho \|u_n - u\|_U,\quad n\in \N,
\end{equation}
which holds if and only if $\|\id -\cP\op\|_{\cL(U,U)}\le \rho <1$. Note that for the variational formulation
(F1) the residual $f-\op v$ is, by Theorem \ref{thm:Binverse}, well-defined in $V'$ for any $v\in U$.

Recalling Remark \ref{rem:Neuman}, we consider two distinct problem regimes.
}
 \begin{remark}
 \label{rem:ess}
\dw{The operator equation $\op u=f$ implies homogeneous inflow-boundary conditions.}  Incorporating \dw{inhomogeneous} boundary
 conditions could be treated by taking any function $w$ in the domain of $\op$ that satisfies the required
 boundary conditions and subtract $f_b \coloneqq \op w$ from $f$ reducing the problem to homogeneous conditions.
 \end{remark}
 \vspace*{-4mm}

\om{
\subsection{Dominating Transport\texorpdfstring{: $\|\Tr^{-1}\cK\|_{\cL(U,U)}\le \rho <1$}{}}
\label{ssec:ideal}
If we have the contraction
\begin{equation}
\label{contr2-0}
\|\Tr^{-1}\cK\|_{\cL(U,U)}\le \rho <1,
\end{equation}
then $\cP\coloneqq\Tr^{-1}$ is an admissible preconditioner. In fact, iteration \eqref{geniter} becomes
 \begin{equation}
\label{fp1}
u_{n+1}= u_n + \Tr^{-1}(f- \op u_n) = \Tr^{-1}(\cK u_n + f),\quad n\in \N_0,
\end{equation}
and obviously satisfies \eqref{errorred}, ensuring convergence in $U$ to the solution $u$ of the radiative transfer problem
\begin{equation}
\label{opeq2}
\op u = (\Tr -\cK)u = f.
\end{equation}
In particular, it follows that for any initial guess $u_0$
\begin{equation}
\label{errorn}
\|u- u_n\|_{U} \le \rho^n\|u-u_0\|_{U}.
\end{equation}
}
\vspace*{-3mm}
\om{
\subsection{Dominating Scattering\texorpdfstring{: $\|\Tr^{-1}\cK\|_{\cL(U,U)}\ge 1$}{}}
\label{ssec:domscat}
Throughout this section we continue to assume that \eqref{accretive} holds with some $\alpha >0$.

To find a substitute for the preconditioner $\cP=\Tr^{-1}$ of the transport dominated regime, consider for some fixed $a>0$
\begin{equation*}
\label{hatT}
\hatTr{a} \coloneqq \Tr + a\, \id, \qquad
\hatop{a} \coloneqq \hatTr{a} - \cK
\end{equation*}
and take $\cP \coloneqq \hatop{a}^{-1}$ in \eqref{geniter}. This leads to the (ideal)  iteration
\begin{equation}
\label{fp2}
u_{n+1} = u_n + (\hatTr{a} -\cK)^{-1}(f- (\Tr -\cK)u_n)
= a\,\hatop{a}^{-1}\big( u_n + a^{-1}f\big),\quad n\in\N_0,
\end{equation}
where we have used that $(\hatTr{a} -\cK)^{-1}(\Tr -\cK)= (\hatTr{a} -\cK)^{-1}(\hatTr{a} -\cK - a\id) = -\id + a(\hatTr{a} -\cK)^{-1}$.

Thus, to ensure convergence we need that $\|a(\hatTr{a} -\cK)^{-1}\|_{\cL(U,U)}$ is a contraction. Note that this is satisfied for any $a>0$ since, by Proposition \ref{prop1}, we have that $(\hatop{a} v,v)\ge \alpha + a$, which by Theorem \ref{thm:Binverse} gives
\begin{equation}
\label{invbound}
\|a(\hatTr{a} -\cK)^{-1}\|_{\cL(U,U)} \le \frac{a}{a+\alpha} < 1.
\end{equation}
So \eqref{fp2} converges in $U=L_2(\D\times \S)$ to the true solution $u$ with the error reduction rate $a/(a+\alpha)$ for any fixed $a>0$.

\begin{remark}
Notice that $\cP = \hatop{a}^{-1}$ can be derived from a different perspective. Consider the time dependent initial-boundary value problem
\begin{align}
\label{timedep}
\partial_t u + \Tr u - \cK u &= f, &
u(0,\cdot) &= u^0 \text{ in $\D$} &
\restr{u}{\Gamma_-} &= 0,
\end{align}
(where $f$, $\Tr$, $\cK$ are still independent of $t$). Denoting by
$u_n$ the approximation of $u(t_n)$, $t_n = n\tau$, its backward-Euler
semi-discretization in time reads
\begin{equation*}
\frac{u_{n+1}- u_n}{\tau} + \Tr u_{n+1} - \cK u_{n+1} = f,\quad n\in \N_0,
\end{equation*}
which gives
\begin{equation}
\label{bEuler}
(\tau^{-1} \id + \Tr -\cK)u_{n+1} = \tau^{-1} u_n + f,\quad n\in \N_0.
\end{equation}
This coincides with \eqref{fp2} for $a = \tau^{-1}$.
\end{remark}
}

\om{
\section{Step (III)---Perturbed Iterations and the Main Algorithm}\label{sec:periter}
The practical realization of the scheme boils down to two tasks:
\begin{itemize}
\item[(T1)] {\it Formulate a perturbed version of algorithms \eqref{fp1} and \eqref{fp2} with suitable error tolerances $\eta_n$ that still guarantee convergence to the exact continuous solution.}
\end{itemize}
For this task, it will be convenient to use the following notational
convention: Given an operator $\cG \in \cL(U,Y)$, we denote for any
$\eta >0$ by $[\cG, w;\eta]$ an element in $Y$ satisfying $\|\cG w - [\cG,w;\eta]\|_Y \le \eta$. Specifically, for our purposes we require a routine to approximately apply the kernel, that is,
 \begin{equation}
\label{apply}
[\cK,v;\eta] \to z_\eta\quad \mbox{such that }\quad \|\cK v- z_\eta\|_{V'} \le \eta.
\end{equation}
Likewise the source is generally not given exactly and has to be approximated
\begin{equation}
\label{rhs}
[f;\eta] \to f_\eta \quad \mbox{such that }\quad \|f- f_\eta\|_{V'}\le \eta.
\end{equation}
The approximation $[f;\eta]$ of $f$ depends on how the data are given. Finally, given a right hand side $g\in V'$, we have to provide a transport solver
\begin{equation}
\label{solve}
[\Tr^{-1},g;\eta] \to u_\eta \quad \mbox{such that }\quad \|u_\eta - \Tr^{-1}g\|_U\le \eta,
\end{equation}
where, as before, $\Tr$ is viewed as a mapping from $U$ onto $V'$ with $U=L_2(\D\times \S)$, $V= H_{0,+}(\D\times \S)$.
\begin{itemize}
\item[(T2)] {\it Specify how to realize the above routines  in \eqref{apply}, \eqref{rhs}, and \eqref{solve}.}
\end{itemize}
In this section we concentrate only on (T1) and \emph{assume} for the moment that the routines \eqref{apply}, \eqref{rhs}, and \eqref{solve} are available. These routines are detailed later on in Sections \ref{sec:4} and \ref{sec:5}.

\subsection{Dominating Transport\texorpdfstring{: $\|\Tr^{-1}\cK\|_{\cL(U,U)}\le \rho <1$}{}}
\label{ssec:asti}
An approximate realization of the ideal scheme \eqref{fp1} is
\begin{align}
\label{fp1p}
\bar u_{n+1} = [ \Tr^{-1}, [\cK, \bar u_n; \eta_\cK] + [f;\eta_f];\eta_\Tr ], \quad n\geq 0.
\end{align}
\dw{In the following we take for simplicity $u_0=0$. Any other choice for $u_0$ that exploits additional information would, of course,} be possible.
\dw{We choose the individual tolerances proportional to
\begin{equation}
\label{etan}
\eta_n = (1+n)^{-\beta}\rho^n,
\end{equation}
for some fixed $\beta>1$ ($\beta=1.5$ in later numerical experiments). Specifically,
we set
$$
\eta_\cK \coloneqq \kappa_1\eta_n, \quad \eta_f \coloneqq \kappa_2\eta_n, \quad \eta_\Tr \coloneqq \kappa_3\eta_n,
$$
where the parameters $\kappa_1,\kappa_2,\kappa_3 \geq 0$ satisfy
\begin{equation}
\label{kappa}
C_\Tr (\kappa_1 + \kappa_2) + \kappa_3 \le 1,
\end{equation}
with the upper bound $\|\Tr^{-1}\|_{\cL(V',U)} \le C_\Tr$ from \eqref{contTminus}.}

\dw{
In addition we need an upper bound for $\|u\|_U$. A first simple estimate that can be obtained  from \eqref{accretive} or \eqref{L2B}
\begin{equation}
\label{err0}
\Vert u \Vert_U
\leq
\|\op^{-1}\|_{\cL(V',U)}\|f\|_{V'}\le \alpha^{-1}\|f\|_{L_2(\D)}.
\end{equation}
Since this may be rather pessimistic
when $\alpha$ is small we
take
$$
b_0(u)\coloneqq \alpha^{-1}\|f\|_{L_2(\D)}
$$
only as an \emph{initialization} which is refined during the course of the iteration based on a posteriori information. In the following, we will work with
\begin{equation*}
\ub[n+1] \coloneqq \min\big\{\ub[n], \|\bar u_{n+1}\|_U + (\rho \,\ub[n] +\zeta(\beta))\rho^{n-1}\big\},\quad n\geq 0,
\end{equation*}
which is an upper bound that converges to $\|u\|_U$.
}


\dw{We are now prepared to present a detailed account of the perturbed iteration \eqref{fp1p} in terms of the following Algorithm \ref{alg:asti}
called \emph{Adaptive Source Term Iteration} (\ASTI).}
We prove in Theorem \ref{th:terminates} that for dominating transport $\ASTI[\Tr,\cK,f;\e]$ computes an approximate solution $u_\e$ such that $\|u-u_\e\|_{U}\le \e$.


\begin{algorithm}[h]
  \begin{algorithmic}[1]
   \State Fix $\kappa_1$, $\kappa_2$, $\kappa_3$ according to \eqref{kappa}, fix $\beta >1$,
   estimate $\rho$ by \eqref{TK1}, and choose $\ub[0]$ e.\,g., as in \eqref{err0}.
    \State $n \gets 0$ \label{alg:initn}
    \State $\bar u_n \gets 0$
    \State $\text{err} \gets \ub[0]$
    \State $\ub \gets \ub[0]$
    \While{$\text{err} > \e$}
      \State $\eta_n \gets  {(1+n)^{-\beta}\rho^n}$
      \State $w \gets [\cK,\bar u_n ;\kappa_1\eta_n]$
      \State $g \gets [f;\kappa_2\eta_n]$
      \State $\bar u_{n+1} \gets [\Tr^{-1}, w + g;\kappa_3\eta_n]$
      \State $\text{err} \gets  { (\rho \, {\ub} + \zeta(\beta)) \rho^n}$
      \State $\ub \gets \min\big\{\ub, \|\bar u_{n+1}\|_U+(\rho \ub +\zeta(\beta))\rho^{n-1}\big\}$
      \State $n \gets n+1$
    \EndWhile
    \State $u_\e\gets \bar u_n$
  \end{algorithmic}
  \caption{$\ASTI[\Tr,\cK,f;\e]\to u_\e$}
  \label{alg:asti}
\end{algorithm}

\subsection{Dominating Scattering\texorpdfstring{: $\|\Tr^{-1}\cK\|_{\cL(U,U)}\ge 1$}{}}
\label{ssec:geone}
For a given $a>0$, the approximate realization of the scheme \eqref{fp2} takes the form
\begin{equation}
\label{fp2p}
\bar u_{n+1} = [a\hatop{a}^{-1}, \bar u_n + [a^{-1}f;\eta_{n}];\eta_n], \quad n\in \N_0,
\end{equation}
where the stage dependent tolerances $\eta_n$ are chosen as in \eqref{etan}.

To \dw{render the approximate application of the preconditioner $a\hatop{a}^{-1}$ practical, we choose the parameter $a$ in such a way that the operator $\hatop{a}$ is \emph{transport dominated},  so that we can resort to the \ASTI~algorithm for its approximate inversion.
To that end, recall from \eqref{TK1} that $\|\Tr^{-1}\cK\|_{\cL(U,U)}$ is estimated in terms of quantities $\zeta, \gamma$ from \eqref{gamma}.
When $\Tr$ is replaced by $\hatTr{a}$ these quantities depend on $a$ and are therefore denoted for clarity by $\gamma_a$, $\zeta_a$.
Since the quantities $\bar\sigma$, $\bar\sigma'$ are not affected by the parameter $a$, we have
\begin{align*}
\gamma_a &\le
\frac{\sigma_{\max}-\alpha}{\sigma_{\min}+a}, &
\zeta_a &\le
\frac{(\sigma_{\max}-\alpha)(\sigma_{\max}+a)}
     {(\sigma_{\min}+a)(\sigma_{\min}+a)}.
\end{align*}
In view of the bound \eqref{invbound} for $\Vert a\hatop{a}^{-1}\Vert_{\cL(U,U)}$,  by choosing the parameter $a=a^*$ as the unique solution of
\begin{equation}
\label{besta}
\frac{a}{a+\alpha}
= \frac{(\sigma_{\max}-\alpha)(\sigma_{\max}+a)}
       {(\sigma_{\min}+a)(\sigma_{\min}+a)},
\end{equation}
one obtains simultaneously
\begin{equation}
\label{simult}
\Vert  a^* \hatop{a^*}^{-1} \Vert_{\cL(U,U)} \le \rho^*
\quad\text{and} \quad
\|\hatTr{a^*}^{-1}\cK\|_{\cL(U,U)}\le \rho^* \quad\text{for some $\rho^* <1$.}
\end{equation}
Thus, an error controlled application of the preconditioner $a^*\hatop{a^*}^{-1}$ is given for any right hand side $g$ and accuracy $\eta$ as
\begin{equation}
\label{nowB}
[\hatop{a^*}^{-1}, g;\eta] = \ASTI[\hatTr{a^*},\cK,g;\eta].
\end{equation}
}

Note that the algorithm consists now in nesting the outer iteration with an inner \ASTI~iteration for the application of the preconditioner. It is thus straighforward to formulate a general \emph{Nested ASTI}~scheme, where $\solve[\op,f;\e]$ generates an approximate solution $u_\e$ such that $\|u-u_\e\|_{U}\le \e$ even when scattering dominates in $\op$ (see Algorithm \ref{alg:solve}).


\begin{algorithm}
  \begin{algorithmic}[1]
  \State $\rho \gets$ Estimate $\Vert \Tr^{-1} \cK \Vert_{\cL(U,U)}$ using upper bound of \eqref{TK1}.
  \If {$\rho < 1$}
    \Comment{Dominating transport}
    \State $u_\e \gets \ASTI[\Tr, \cK, f; \e]$
  \Else
    \Comment{Dominating scattering}
    \State Estimate $a^*$ from \eqref{besta}, estimate $\rho^*$ from \eqref{simult}, fix $\beta >1$.
    \State $n \gets 0$
    \State $\bar u_n \gets 0$
    \State $\text{err} \gets \ub[0]$
    \State $\ub \gets \ub[0]$
    \While {$\text{err}>\e$}
      \State $\eta_n \gets  {(1+n)^{-\beta}(\rho^*)^n}$
      \State $g \gets \bar u_n + [(a^*)^{-1} f; \eta_n]$
      \State $\bar u_n = a^* \ASTI[\hatTr{a^*}, \cK, g; \e]$
      \State $\text{err} \gets \left( \rho^* \ub + (1+a^*)\zeta(\beta) \right) (\rho^*)^n $
      \State $\ub \gets \min\big\{\ub, \|\bar u_{n+1}\|_U+((\rho^*) \ub + (1+a^*)\zeta(\beta))(\rho^*)^{n-1}\big\}$
      \State $n \gets n+1$
    \EndWhile
    \State $u_\e\gets \bar u_n$
  \EndIf\\
  \Return $u_\e$
  \end{algorithmic}
  \caption{$\solve[\op,f;\e]\to u_\e$}
  \label{alg:solve}
\end{algorithm}

\subsection{Convergence of \texorpdfstring{$\solve[\op,f;\e]$}{SOLVE[B,f;e]}}\label{ssec:conv}
\begin{theorem}
\label{th:terminates}
For any target accuracy $\e>0$, Algorithm \ref{alg:solve} terminates and its output
\begin{equation*}
u_\e \coloneqq \solve[\op,f;\e]
\end{equation*}
satisfies
\begin{equation}
\label{uepsilon}
\|u-u_\e\|_U \le \e,
\end{equation}
where $u$ is the exact solution of \eqref{opeq} with respect
to the variational formulation (F1).
\end{theorem}
\begin{proof}
We first consider the transport dominated case where $\Vert \Tr^{-1} \cK \Vert_{\cL(U,U)}<1$. The algorithm then reduces to \ASTI, that is,
$$
u_e = \ASTI[\Tr, \cK, f; \e].
$$
Let $u_n$ denote the exact iterates of \eqref{fp1} and $\bar u_n$ the ones from the perturbed version \eqref{fp1p}.
By the definition of the respective routines we have for given tolerances
$\eta_\Tr$, $\eta_\cK$, $\eta_f$
\begin{align*}
u_{n+1} - \bar u_{n+1}
&= \Tr^{-1}(\cK u_n + f)
   - [\Tr^{-1}, [\cK,\bar u_n;\eta_\cK] + [f;\eta_f];\eta_\Tr] \\
&= \Tr^{-1}\bigl( \cK (u_n-\bar u_n) \bigr)
   + \Tr^{-1}(\cK\bar u_n - [\cK,\bar u_n;\eta_\cK])
   + \Tr^{-1}(f - [f;\eta_f]) \\
&\phantom{{}={}} + \Tr^{-1}([\cK,\bar u_n;\eta_\cK] + [f;\eta_f])
   - [\Tr^{-1}, [\cK,\bar u_n;\eta_\cK] + [f;\eta_f];\eta_\Tr].
\end{align*}
By the triangle inequality, bound \eqref{contTminus} on $\|\Tr^{-1}\|_{\cL(V',U)}$, and the properties of the routines, we obtain
\begin{equation*}
\|u_{n+1}- \bar u_{n+1}\|_U
\leq \rho \|u_n -\bar u_n\|_U
   + C_\Tr(\eta_\cK + \eta_f) + \eta_\Tr .
\end{equation*}
For $\bar u_0=u_0$ and with the choice
$\eta_\cK \coloneqq \kappa_1\eta_n$, $\eta_f \coloneqq \kappa_2\eta_n$
and $\eta_\Tr \coloneqq \kappa_3\eta_n$ and \eqref{kappa}, we get
\begin{equation*}
\|u_{n+1}- \bar u_{n+1}\|_U
\le \rho \|u_n -\bar u_n\|_U + \eta_n,
\end{equation*}
which, by induction, yields
\begin{align}
\label{1step}
\| \bar u_{n+1}- u_{n+1}\|_U &
\le \sum_{j=0}^{n} \rho^j \eta_{n-j}.
\end{align}
Specifically, taking the same $\eta_n$ as in \eqref{etan} for some fixed $\beta > 1$, we obtain
\begin{align}
\label{ubaru}
\| \bar u_{n+1}- u_{n+1}\|_U
&\leq
 \sum_{j=0}^{n} \rho^j \rho^{n-j} (1+(n-j))^{-\beta}
 = \rho^n \sum_{j=0}^n (1+j)^{-\beta}
\le \zeta(\beta) \rho^n,
\end{align}
where
$ \zeta(\beta) \coloneqq \sum_{j\in \N}j^{-\beta}$
 is the $\zeta$-function.
 Hence, by triangle inequality
\begin{equation}
\label{uideal}
\|u- \bar u_{n+1}\|_U
\leq \rho^{n+1} \|u\|_U + \zeta(\beta)\rho^n.
\end{equation}
Thus, whenever at the $n$th stage of the algorithm  $\|u\|_U \le \ub[n]$, we conclude that
\begin{equation}
\label{better}
\ub[n+1] \coloneqq \min\big\{\ub[n], \|\bar u_{n+1}\|_U + (\rho \,\ub[n] +\zeta(\beta))\rho^{n-1}\big\}
\end{equation}
a bound for $\|u\|_U$ which converges to $\|u\|_U$.
This yields the computable error bound
\begin{equation}
\label{purple}
\|u - \bar u_{n+1}\|_U
\leq (\rho \,\ub[n+1] + \zeta(\beta)) \rho^n
\end{equation}
which completes the proof for the transport dominated case.

For dominating scattering, denoting by $u_n$ the exact iterates
\begin{equation*}
u_{n+1} = a^* \hatop{a^*}^{-1}(u_n + (a^*)^{-1}f), \quad n\in \N_0,
\end{equation*}
we readily obtain
\begin{align*}
\bar u_{n+1}- u_{n+1} &= [a^*\hatop{a^*}^{-1}, \bar u_n + [(a^*)^{-1}f;\eta_n];\eta_n] -  a^*\hatop{a^*}^{-1}(\bar u_n +  [(a^*)^{-1}f;\eta_n])\\
&\quad + a^*\hatop{a^*}^{-1}(\bar u_n +  [(a^*)^{-1}f;\eta_n])-
 a^*\hatop{a^*}^{-1}(\bar u_n + (a^*)^{-1}f)
+
a^*\hatop{a^*}^{-1}(\bar u_n-u_n).
\end{align*}
Hence,
\begin{equation}
\label{111}
\|\bar u_{n+1}- u_{n+1}\|_U
\le a^*\eta_n + \rho^* \eta_n + \rho^*\|\bar u_{n} - u_{n}\|_U.
\end{equation}
We obtain as earlier with $\bar u_0=u_0$
\begin{align*}
\| \bar u_{n+1}- u_{n+1}\|_U &
\le (1+ a^*) \sum_{j=0}^{n} (\rho^*)^j \eta_{n-j}.
\end{align*}
Specifically, taking $\eta_n$ from \eqref{etan} we get, on account of
\eqref{errorn},
\begin{equation}
\label{barun+1+}
\|u-\bar u_{n}\|_U
\le \bigl( \rho^*
\|u-u_0\|_U+ (1+ a^*) \zeta(\beta)\bigr) (\rho^*)^{n-1}
, \quad n\in \N,
\end{equation}
and hence the same type of bound as in \eqref{uideal} for the transport dominated case.
\end{proof}

\begin{remark}
\label{rem:betterbu}
The recursion \eqref{better} successively mitigates a possibly
over-pessimistic initial bound $b_0(u)$.
It can be further improved by using the a posteriori bound
$\|u-u_n\|_U \le \frac{\rho}{1-\rho}\|u_n-u_{n-1}\|_U$.
We also have (for $n\ge 2$)
\begin{align*}
\|u-u_n\|_U & \le \frac{\rho}{1-\rho}\big\{\|\bar u_n-\bar u_{n-1}\|_U +\|u_n-\bar u_n\|_U +\|u_{n-1}-\bar u_{n-1}\|_U\big\}\\
&\le
 \frac{\rho}{1-\rho}\big\{\|\bar u_n-\bar u_{n-1}\|_U + \zeta(\beta) (\rho^{n-1}+\rho^{n-2})\big\},
\end{align*}
which is a computable bound replacing $\|u-u_n\|_U$.
However, the calculation of these a posteriori quantities would require storing two consecutive outer iterates.
\end{remark}
}
\dw{
\subsection{Complexity}
We conclude with some qualitative complexity estimates. Further quantifications depend on the realizations of the involved routines.
The number $n(\e)$ of outer iteration steps required to realize $\|u-\bar u_{n(\e)}\|_U \le \e$ is given by
\begin{equation}
\label{ne}
n(\e)= \left\lceil \frac{|\ln \e|+ \ln (\rho
  \ub + a^*\zeta(\beta))}{|\ln \rho|}\right\rceil.
\end{equation}
As detailed in the subsequent section the approximate application of the scatterer is typically dominated by the approximate
inversion of the transport operator. As a consequence, in either version of the outer iteration the computational work per
outer iteration step $n$ is dominated by the computational cost $\cost_\cP(\eta_n)$ of the preconditioner. Hence, the
complexity  $\cost_{\op^{-1}}(\e)$ of solving $\op u=f$ within accuracy $\e$ can be bounded as
\begin{equation}
\label{sumcost}
\cost_{\op^{-1}}(\e) \lesssim \sum_{j=1}^{n(\e)} \cost_{\cP}(\eta_n).
\end{equation}
Assuming that $\cost_{\cP}(\eta)\lesssim \eta^{-\vartheta}$ holds for some positive $\vartheta$ (which is actually realistic as
will be seen later), this yields
\begin{align}
\label{totalcost2}
\cost_{\op^{-1}}(\e)
&\lesssim \sum_{j=1}^{n(\e)} \rho^{-j\vartheta} (1+j)^{\beta\vartheta}
\le (1+n(\e))^{\beta\vartheta} \sum_{j=0}^{n(\e)} \rho^{-j\vartheta}\nonumber\\
&
\le \frac{\rho^{-n(\e)\vartheta}}{1-\rho^\vartheta}(1+n(\e))^{\beta\vartheta}\le C \e^{-\vartheta}|\ln \e|^{\beta\vartheta},
\end{align}
where $C=C(\beta,\vartheta,\rho,u)$ is a constant depending on $\beta$, $\vartheta$, $\rho$ and a bound $\ub$ for $\|u-u_0\|_U$.
As a result, the cost of approximately inverting $\op$ is, up to a logarithmic factor, of the order of the one for the application of the preconditioner
with the same accuracy, that is
\begin{equation}
\label{best}
\cost_{\op^{-1}}(\e) \lesssim |\ln \e|^{\beta\vartheta}\cost_{\cP}(\e).
\end{equation}
The cost of the preconditioner, in turn, depends on the problem regime.
 For dominating transport $\cost_\cP(\e) = \cost_{\Tr^{-1}}(\e)$, while for dominating scattering
the approximate application of $a^*\hatop{a^*}$ within accuracy $\e$ requires (in the inner iteration) invoking
$\ord(|\ln \e|/|\ln \rho^*|)$ times an $\e$-accurate transport solve, i.\,e.,
$\cost_\cP(\e) \lesssim \cost_{\Tr^{-1}}(\e) |\ln \e|/|\ln \rho^*|$.

In summary, the overall computational complexity for a given target accuracy is essentially determined by the cost
of error-controlled transport solves (provided that a reasonably efficient approximate application scheme for the
scatterer is at hand).
A posteriori bounds for transport solvers are therefore pivotal. Moreover, since the target tolerances $\eta_n$ are
gradually tightened, early stages of the outer iteration (and its preconditioners) require only correspondingly
cruder accuracy tolerances so that (up to logarithmic factors) the total complexity is dominated by the cost of
the last outer iteration step.

The remainder of the paper is devoted to realizations of $[\cK,v;\eta]$ and $[\Tr^{-1},g;\eta]$.

}

\section{The \texorpdfstring{routine $[\cK,v;\eta]$}{scattering routine}}
\label{sec:4}
\subsection{Introductory comments}
The scheme $\ASTI$ requires the application of the global operator $\cK$ within dynamically updated accuracy tolerances.
\dw{We present in this section an efficient error-controlled approximate application scheme that makes use of \emph{wavelet-compression}
and \emph{low-rank} approximations. Fully nonlinear versions with even better scaling are postponed to forthcoming work.}

We confine the discussion to  the class of kernels of the form \eqref{ker1}, that is
$K(x,\s,\s') = \kappa (x)\kk(\s,\s')$, $\kk(\s,\s') = \kk(\s',\s)$, with
$\kk(\s,\s')\ge 0$, $\s,\s'\in\S$, $\kappa \ge \kappa_0>0$,
 and the normalization
\begin{equation}
\label{normaiizeG}
\int_{\S}\kk(\s,\s')\,\ds' = \int_{\S}\kk(\s,\s')\,\ds = 1, \quad \s,\, \s'\in \S.
\end{equation}
In the following, we adhere to the notation
\begin{equation*}
\cK_0 v \coloneqq \int_\S \kk(\cdot,\s')v(\s')\,\ds'.
\end{equation*}
The simplest examples are \emph{isotropic} and  \emph{Rayleigh type scattering} which are respectively of the form
\begin{equation}
\label{isoscat}
\kk(\s,\s') \coloneqq |\S|^{-1},\quad \kk(\s,\s')=c\left(1+(\s\cdot\s')^2\right).
\end{equation}
Another variant of interest, used in \cite{Kanschat2009}, is given in terms of the similar expansion
\begin{equation}
\label{Tcheb}
\kk(\s,\s')=\sum_{n=0}^\infty a_n T_n(\s\cdot\s'),
\end{equation}
with $a_n\geq 0$ and $T_n$ being the $n$th Chebyshev polynomial,
$
T_n(x)\coloneqq \cos \left(n\arccos(x) \right)$, for $ |x|\leq 1$. It is shown in \cite[Lemmata 2 and 3]{Kanschat2009} that $\cK$ is positive semi-definite with this type of kernel.

In our numerical scheme we focus on
\emph{Henyey--Greenstein} type scattering represented by
\begin{align}
\label{eq:henyey}
\kk_\gamma (\s,\s') \coloneqq
\begin{cases}
\frac{1}{2\pi}\frac{1-\gamma^2}{1+\gamma^2-2\gamma \s\cdot\s'},& \text{if $d_{\S}=1$},\\
\frac{1}{4\pi}\frac{1-\gamma^2}{(1+\gamma^2-2\gamma \s\cdot\s')^{3/2}},& \text{if $d_{\S}=2$},
\end{cases}
\end{align}
where $d_{\S}=d-1$ denotes the dimension of the parameter domain. This scattering model is widely used among physicists and was introduced in \cite{HG1941} to describe anisotropic effects via the parameter $-1\leq \gamma\leq 1$. When $\gamma\geq 0$, the scattering is called \emph{forward-peaked} and $\cK_0$ is positive semi-definite.
Moreover, for $d_{\S}=2$ one has the expansion
\begin{equation}
\label{Legend}
\frac{1}{(1+\gamma^2-2\gamma \s\cdot\s')^{3/2}}=\sum_{n=0}^\infty \gamma^n P_n(\s\cdot\s')
\end{equation}
where $P_n$ is the Legendre polynomial of degree $n$.
Note that the closer $\gamma$ comes to one, the slower is the decay and the larger is the model error when replacing $\kk$ by a truncated
expansion in favor of an efficient application of the scatterer to a given input.

Our focus on Henyey--Greenstein type scattering is mainly motivated by the fact that varying
the parameter $\gamma$ allows us to quantitatively investigate different scattering regimes guiding the
search for possibly different ways of exploiting sparsity.

The specification of $[\cK, \bar u;\cdot]$ depends on the following \emph{input format} of $\bar u\in L_2(\D\times\S)$.
As explained in  Section \ref{sec:5}, $\bar u$ is the output of a Discontinuous Petrov--Galerkin transport solver. It is
a piecewise polynomial of degree $m$, subordinate to some current partition $\pD$
 of the spatial domain $\D$ and whose coefficients are piecewise polynomials in the direction parameter $\s\in \S$.
Thus,  $\bar u$ has the form
\begin{equation}
\label{inputbaru}
\bar u(x,\s) = \sum_{T\in \pD, i\in \cI_T} v_{T,i}(\s)\varphi_{T,i}(x),
\end{equation}
where the spatial shape functions $\varphi_{T,i}$, $i\in \cI_T$ are an orthonormal basis for $\PP_m(T)$ and each
parameter dependent coefficient $v_{T,i}$
is an element of $\PP_M(\pS)$ where $\pS$ is a partition of $\S$. Hence,
\begin{equation}
\label{Kbaru}
(\cK \bar u)(x,\s) = \sum_{T\in \pD_h, i\in \cI_T} (\cK_0v_{T,i})(\s)\kappa(x)  \varphi_{T,i}(x).
\end{equation}
The simplest realization of $[\cK,\cdot;\cdot]$ rests on computing $\eta$-accurate approximations $w_{T,i} = [\cK_0, v_{T,i};\eta]$
to $(\cK_0v_{T,i})$
so that (by orthonormality),
\begin{align}
\label{Keta}
[\cK,\bar u;\eta]
\coloneqq \sum_{T\in \pD_h,i\in \cI_T} w_{T,i}\kappa \varphi_{T,i},
\qquad
\big\|\cK\bar u - [\cK,\bar u;\eta]\big\|_U &\le \eta.
\end{align}
We focus therefore in what follows on the approximate application of $\cK_0$ in the domain $\S$.

\subsection{Matrix representations of \texorpdfstring{$\cK_0$}{the scattering kernel}, Alpert wavelets}
\label{ssec:Alpert}
Suppose that $\Psi =\{\psi_\lambda \mid \lambda \in \Lambda\}$ is an \emph{orthonormal} basis of $L_2(\S)$
where $\Lambda$ is a suitable infinite index set. Then, defining
\begin{align}
\label{bkkwaverep}
\kk^\Psi_{\lambda,\lambda'} &\coloneqq
(\kk,\psi_\lambda\otimes\psi_{\lambda'})_{\S\times \S}
= (\psi_\lambda, \cK_0\psi_{\lambda'})_\S, &
\bkk^\Psi &\coloneqq \bigl(\kk^\Psi_{\lambda,\lambda'}\bigr)
                                    _{\lambda,\lambda'\in \Lambda}.
\end{align}
one has
\begin{equation}
\label{Gwave}
\kk(\s,\s') = \sum_{\lambda,\lambda' \in\Lambda}\kk^\Psi_{\lambda,\lambda'}\psi_\lambda(\s) \psi_{\lambda'}(\s'),
\end{equation}
i.\,e., $\bkk^\Psi$ is an \emph{exact representation} of
the kernel $\kk$ and the associated operator in terms of an \emph{infinite} matrix. By orthonormality of $\Psi$
we have
 \begin{equation}
 \label{spectralnorms}
 \|\bkk^\Psi\|\coloneqq \|\bkk^\Psi\|_{\cL(\ell_2(\Lambda),\ell_2(\Lambda))} = \|\cK_0\|_{\cL(L_2(\S),L_2(\S))}.
 \end{equation}
 An $\eta$-accurate application of $\cK_0$ will be accomplished by identifying a ``compressed'' finite submatrix $\bkk^\Psi_\eta$ of
 $\bkk^\Psi$ that reduces the approximate application of $\cK_0$ to an efficient matrix-vector multiplication.

As an appropriate choice for $\Psi$ we advocate so called \emph{Alpert wavelet bases} of (at least) degree $M$ from \eqref{inputbaru}.
For the convenience of the reader we briefly recapitulate some basic features of Alpert wavelets and refer to \cite{Alpert1993} for further
details.

Starting from some initial partition $\pS_0$ of $\S$ (which could be the trivial one $\{\S\}$) and fixing a rule for splitting each cell $\cell$  in a given
partition into a fixed number  of ``children'' forming the refinement $\cC(\cell)$ of $\cell$, repeated refinements generate an
 infinite ``master-tree'' $\mathbb{T}$
whose nodes are cells and whose edges connect parents with children. We call a finite subtree of $\mathbb{T}$ complete
if a child of a cell  $\cell$ belongs to the subtree if and only if all of $\cC(\cell)$ is contained in the subtree. We consider only complete subtrees.
Then the set of \emph{leaves} of  such a finite subtree forms a so called ``admissible''  partition $\pS$ of $\S$ whose ``refinement history'' is determined by
the subtree, i.\,e., there is a one-to-one correspondence between such (possibly very non-uniform) partitions $\pS$  and subtrees   $\mathbb{T}_\pS$
of $\mathbb{T}$. The $\s$-dependent coefficients $v_{T,i}, w_{T,i}$ in \eqref{ubaru}, \eqref{Keta} will always be piecewise
polynomials of degree $M$ on such admissible partitions. We will make use of two different representations of such piecewise polynomials as described next.
%

 Let $\PP_M(\cell)$ denote the space of polynomials of (total) degree at most $M$ over  the cell $\cell$.
Given   an admissible partition $\pS$ of $\S$,
let $\PP_{M}(\pS)$
denote the space of piecewise polynomials of degree at most $M$, subordinate to the partition $\pS$.
A canonical basis for
$\PP_{M}(\pS)$ is obtained by associating with each cell $\cell\in\pS$ an orthonormal basis
\begin{equation}
\label{PhiP}
\Phi_\cell = \{\phi_\nu \coloneqq \chi_\cell P_{\cell,i} \mid \nu \coloneqq (\cell,i),\,P_{\cell,i}\in \PP_M(\cell), \, i \in \cI_M \coloneqq \{1,\ldots, \dim \PP_M\}\},
\end{equation}
which gives rise to what is sometimes referred to as the orthonormal \emph{scaling function} basis
\begin{equation}
\label{PhiS}
\Phi_{\pS} \coloneqq \bigcup_{\cell\in \pS}\Phi_\cell = \bigl\{\phi_\nu \mid \nu \in \Gamma_{\pS} \bigr\}, \quad  \Gamma_{\pS} \coloneqq
       \bigl\{(\cell,i) \mid \cell \in \pS,\ i\in \cI_M\}\bigr\},
\end{equation}
to be always understood with respect to the uniform Haar measure on $\S$ induced by a convenient parametrization,
i.\,e., $\int_\S \ds = 1$ and $(v,w)_\S=\int_\S vw\,\ds$.

 Alpert wavelets provide alternative bases for such spaces of piecewise polynomials that encode ``updates'' obtained by passing to a refined partition.
 They are therefore better suited for meeting variable target accuracies.
Since $\PP_M(\cell) \subset \PP_M(\cC(\cell))$ one can determine an \emph{orthonormal} set of piecewise polynomials in $\PP_M(\cC(\cell))$.
Setting $ \cJ_M \coloneqq\{1, \ldots, \operatorname{dim}( \PP_M(\cC(\cell)) - \operatorname{dim} \PP_M(\cell))\}$,
\begin{equation}
\label{Alpert}
\Psi_\cell\coloneqq \{\psi_\lambda  \mid \lambda
\coloneqq (\cell,r), \, r\in \cJ_M\}\subset \PP_m(\cC(\cell))
\end{equation}
 spanning the orthogonal complement
$ 
\WW(\cell) \coloneqq \PP_M(\cC(\cell)) \ominus \PP_M( \cell)
$  
between two successive levels of piecewise polynomials.
%
Obviously,
\begin{equation}
\label{APsi}
\Psi \coloneqq \{\psi_\lambda \mid \lambda \in \Lambda\},\quad \Lambda
\coloneqq \{\lambda = (\cell,r)\mid r\in \cJ_M,\, \cell \in \mathbb{T}\}
\end{equation}
is an orthonormal basis for $L_2(\S)$. Clearly, for any admissible partition $\pS$ of $\S$ one easily identifies the subset $\Psi_\pS=\{\psi_\lambda: \lambda \in \Lambda_\pS\}\subset\Psi$ which forms a basis for $\PP_M(\pS)$,
namely
\begin{equation*}
\Lambda_{\pS}
\coloneqq \{\lambda =(\cell,r) \mid r\in \cJ_M,\ \cell \in \TT_{\pS}\}.
\end{equation*}

Alpert bases are easy to construct, in particular, for domains like $\S$.
It is well known that changing from a scaling function representation of an element in $\PP_M(\pS)$ to its Alpert wavelet representation
(and vice versa) can be done at $\ord(\#\pS)$ cost with the aid of the fast wavelet transform. Accordingly, one can efficiently
pass from a scaling function representation of a compressed kernel to its wavelet representation and vice versa.

Moreover, $\psi_\lambda$, $|\lambda|>0$, have \emph{vanishing moments} of  order $M+1$, i.\,e.,
\begin{equation}
\label{vanmom}
(P,\psi_\lambda)_\S = 0 \quad \forall P\in \PP_M(\operatorname{supp} \psi_\lambda).
\end{equation}
This has two important consequences.
First, whenever a submatrix $\bkk^\Psi_\eta$ of $\bkk^\Psi$ is obtained by discarding entries $\kk^\Psi_{\lambda,\lambda'}$ with $|\lambda|+|\lambda'|>0$
the corresponding kernel $\kk_\eta$ still satisfies
\begin{equation}
\label{energy}
\int_{\S\times\S}\kk_\eta(\s,\s') \,\ds\,\ds'= 1.
\end{equation}
Second, \eqref{vanmom} will be shown next to imply that $\bkk^\Psi$ is \emph{nearly sparse} which provides the basis for
an error controlled efficient application of $\cK_0$ through \emph{matrix compression}.

\subsection{Compression of \texorpdfstring{$\bkk^\Psi$}{the wavelet representation of the scattering kernel}}\label{ssec:compress}
As a guiding example, let us consider the case $d=2$ (two spatial variables) such that $\S$ is the unit circle and has dimension $d_{\S} = d-1=1$. Note that the Henyey--Greenstein kernel is then of the form
\begin{equation}
\label{Henyeyex}
\kk_\gamma(\theta, \theta') = c(H_\alpha\circ \delta)(\theta, \theta'),\quad H_\alpha(\varphi)\coloneqq \frac{1}{1-\alpha\cos\varphi}
\qquad\text{and}\qquad
\delta(\theta, \theta') = \theta-\theta'
\end{equation}
where $c=\frac {1-\gamma^2}{2\pi(1+\gamma^2)}$ and $\alpha = \frac{2\gamma}{1+\gamma^2}$.
\begin{prop}
\label{prop:Gcompress}
In the above terms one has
\begin{align}
\label{Gdecay}
 \big|(\kk_\gamma)_{\lambda, \lambda'}\big| &\lesssim 2^{-\big(M+1+ \frac{d_{\S}}2\big)\big||\lambda' |-|\lambda|\big|} 2^{-(M+1+d_{\S})\min\{|\lambda|,|\lambda'|\}}
\max_{\ell\le M+1}\big( \dist(S_\lambda,S_{\lambda'})+ 2^{-|\lambda|} \big)^{M+1-\ell}
\nonumber\\
&\quad\qquad \times \sup_{\theta\in S_\lambda, \theta'\in S_{\lambda'}}|H_\alpha^{(2M+2-\ell)}(\theta - \theta')|.
\end{align}
\end{prop}
\begin{proof}
Recall that for $\lambda = (\cell,r)$ one has
$S_\lambda \coloneqq \supp \psi_{\lambda} = \cell$. Let us denote then by $\theta_\lambda$ the center of gravity of $S_\lambda$.
Without loss of generality we can assume that $|\lambda|\le |\lambda'|$.
Taylor expansion of $\kk_\gamma$ at $\theta_\lambda$, using a $(M+1)$st order vanishing moments of $\psi_\lambda$,
yields for integration with respect to $\theta$
\begin{equation*}
\int_{-\pi}^\pi H_\alpha(\theta - \theta')\psi_\lambda(\theta)\,\mathrm d\theta
= \int_{-\pi}^\pi (\theta -\theta')^{M+1}
    H_\alpha^{(M+1)}(\tilde\theta_\lambda-\theta')
    \psi_\lambda(\theta)\,\mathrm d\theta,
\end{equation*}
where $\tilde\theta_\lambda$ is some point in $S_\lambda$. Expanding $Y(\theta')\coloneqq (\theta -\theta')^{M+1}H_\alpha^{(M+1)}(\tilde\theta_\lambda-\theta')$
at $\theta_{\lambda'}\in S_{\lambda'}$, yields upon integrating now first with respect to $\theta'$ and using again $(M+1)$st order vanishing moments,
\begin{equation*}
\big|(\kk_\gamma)_{\lambda, \lambda'}\big| \lesssim
\int_{-\pi}^\pi \int_{-\pi}^\pi
  |\psi_\lambda(\theta)| |\psi_{\lambda'}(\theta')|
  |\theta' - \theta_{\lambda'}|^{M+1}
  |Y^{(M+1)}(\tilde\theta_{\lambda'})|
\,\mathrm d\theta'\,\mathrm d\theta.
\end{equation*}
Since $|\theta' - \theta_{\lambda'}|\lesssim 2^{-|\lambda'|}$,
$\|\psi_\lambda\|_{L_1(S_\lambda)}\lesssim 2^{-d_{\S}|\lambda|/2}$ and
since by Leibniz' rule
\begin{align*}
|Y^{(M+1)}(\tilde\theta_{\lambda'})|&\le C_M \max_{\ell\le M+1}\big( \dist(S_\lambda,S_{\lambda'})+ 2^{-|\lambda|} \big)^{M+1-\ell}
\nonumber\\
&\quad \times \sup_{\theta\in S_\lambda, \theta'\in S_{\lambda'}}|H_\alpha^{(2M+2-\ell)}(\theta - \theta')|.
\end{align*}
the assertion follows.
\end{proof}

Of course, for $\alpha <1$ the terms
\begin{equation*}
C(M,\alpha,\lambda,\lambda')\coloneqq
\max_{\ell\le M+1}\big( \dist(S_\lambda,S_{\lambda'})+ 2^{-|\lambda|} \big)^{M+1-\ell}
 \sup_{\theta\in S_\lambda, \theta'\in S_{\lambda'}}|H_\alpha^{(2M+2-\ell)}(\theta - \theta')|
\end{equation*}
are finite. The closer $\alpha$ (and hence $\gamma$) gets to one the larger one expects the second factor to become for
small $\dist(S_\lambda,S_{\lambda'})$. On the other hand, for larger $\dist(S_\lambda,S_{\lambda'})$ the second factor
turns out to be very small. In summary $C(M,\alpha,\lambda,\lambda')$ is bounded by a constant that possibly grows when
$\gamma$ tends to one but for fixed $\gamma$ decreases when $|\lambda|$, $|\lambda'|$ grow regardless of the distance between
the respective supports. $C(M,\alpha,\lambda,\lambda')$ in turn becomes very small when $\dist(S_\lambda,S_{\lambda'})> c_\gamma$
where $c_\gamma$ decreases when $\gamma$ tends to one. This is illustrated in Figure \ref{fig:finger} reflecting
the strong near-sparsity of the representation.
\begin{figure}[ht]
\centering
\includegraphics[width=0.32\textwidth]{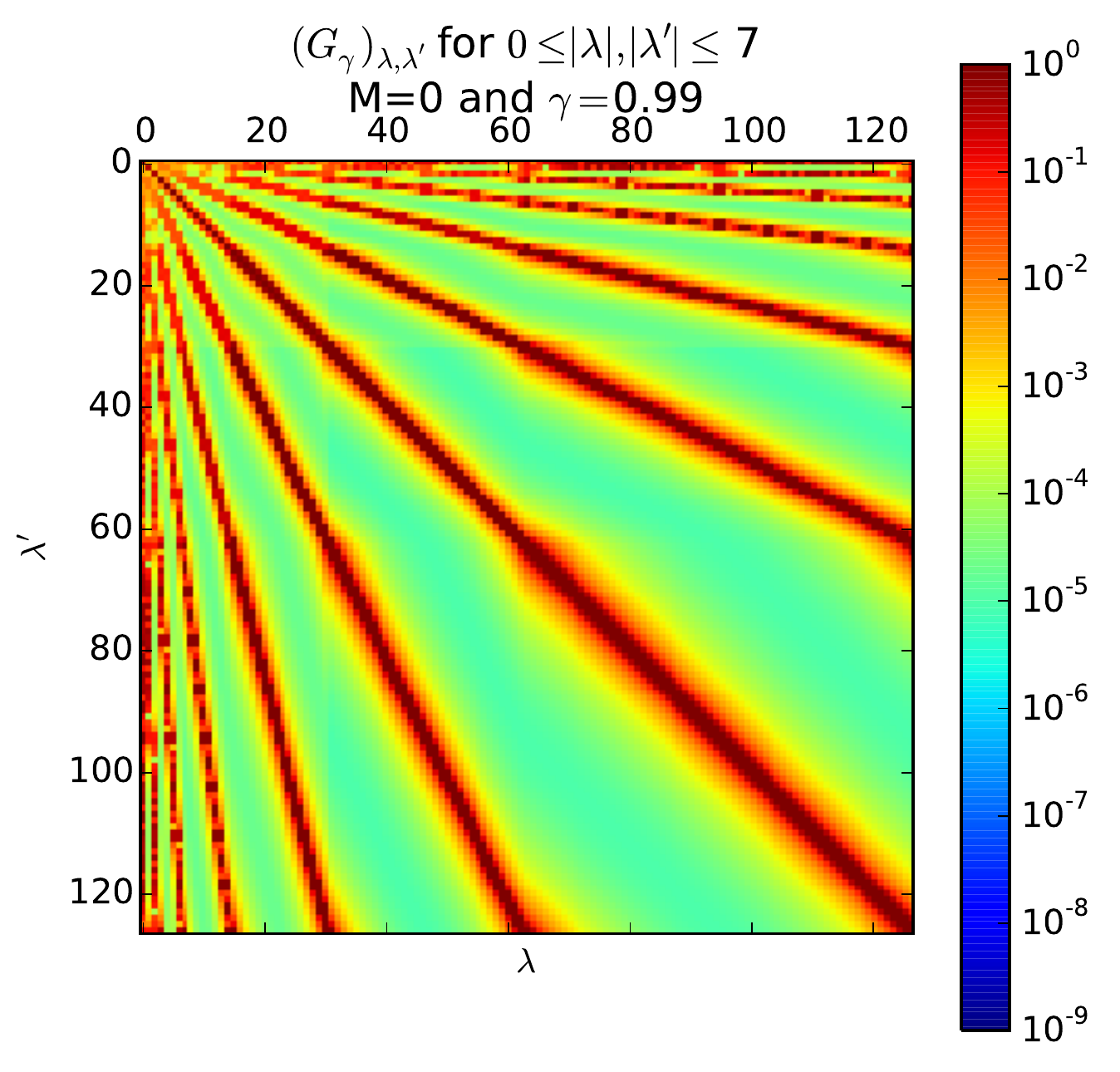}
\includegraphics[width=0.32\textwidth]{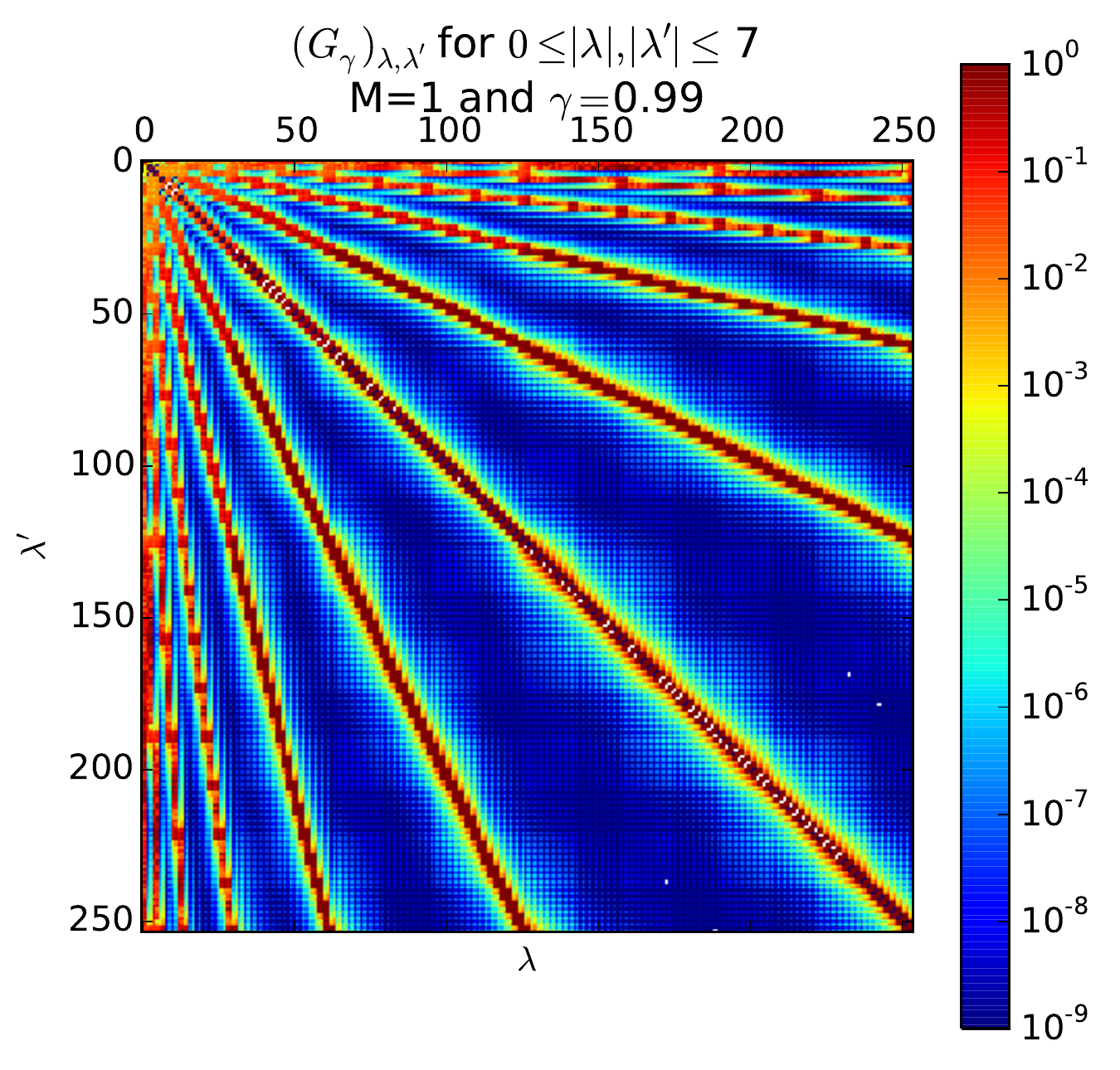}
\includegraphics[width=0.32\textwidth]{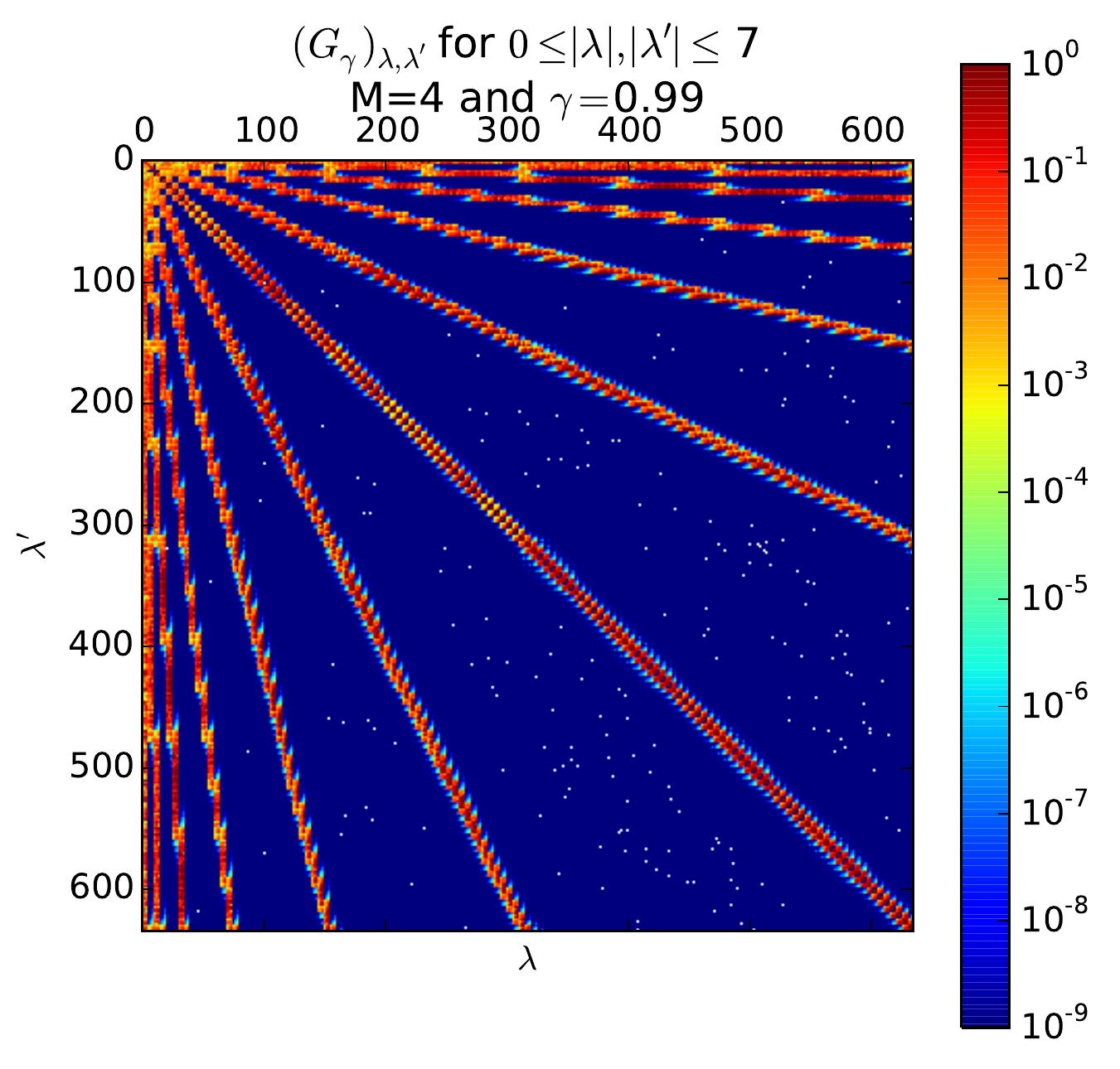}
\caption{Alpert wavelet representation of $\kk_\gamma(\cdot,\cdot)$ of degrees $M= 0,\,1$ and $4$ and $\gamma =0.99$.}
\label{fig:finger}
\end{figure}
 Moreover, defining
\begin{equation*}
d(\lambda,\lambda') \coloneqq 2^{\min\{|\lambda|,|\lambda'|\}}\dist(S_\lambda,S_{\lambda'}),
\end{equation*}
and keeping in mind that $\dist(S_\lambda,S_{\lambda'})$ remains uniformly bounded because of the boundedness of $\S$,
one trivially has $d(\lambda,\lambda')\lesssim 2^{\min\{|\lambda|,|\lambda'|\}}$. Therefore, \eqref{Gdecay} yields the bounds
\begin{equation}
\label{Gdecay2}
 \big|(\kk_\gamma)_{\lambda, \lambda'}\big| \lesssim
 \frac{C(M,\alpha,\lambda,\lambda')
       2^{-\big(M+1+\frac{d_{\S}}2\big) \big||\lambda'|-|\lambda|\big|}}
      {(1 + d(\lambda,\lambda'))^{M+1+d_{\S}}}.
\end{equation}
Treating the term $C(M,\alpha,\lambda,\lambda')$ as a constant, this format allows us to directly invoke results on
wavelet based matrix compression and corresponding \emph{adaptive} approximate application tools, see e.\,g.\ \cite{CDD2001}.
In particular, defining $s^* \coloneqq (M+1)/d_{\S}= M+1$,
\eqref{Gdecay2} ensures that for each $s< s^*$ there exist positive summable sequences $(\alpha_j)_{j\in \N_0}$, $(\beta_j)_{j\in \N_0}$
 and \emph{compressed} versions $\bkk_J$ of $\bkk_\gamma =\bkk$, defined by concrete rules for replacing entries of $\bkk_\gamma$ by zero,
 such that
 \begin{equation}
 \label{Gcompress}
 \|\bkk - \bkk_J\| \le \beta_J 2^{-sJ}, \quad \#(\mbox{entries per row/column}) \le \alpha_J 2^J,\quad J\in \N.
 \end{equation}
 Here $\|\cdot\|= \|\cdot\|_{\cL(\ell_2,\ell_2)}$ denotes the spectral norm.


\subsection{A linear compression scheme}\label{ssec:wavecompr}


Since $\cK_0$ is compact, (unlike the operators considered in \cite{DHS2006}) the entries of $\bkk^\Psi$
decay uniformly with increasing level. Thus, given any ``final'' target accuracy $\e$, one can use \eqref{Gdecay2} to find a level $L=L(\e)\in \N$ such that the finite matrix
$\bkk^\Psi_L\coloneqq \big(\kk_{\lambda,\lambda'}\big)_{|\lambda|,|\lambda'|\le L}$ satisfies $\|\bkk^\Psi-\bkk^\Psi_L\|\le \e$ and hence
\begin{equation}
\label{wish}
\|\cK_0-\cK_{0,L}\|_{\cL(\S,\S)}\le \e,
\end{equation}
which, in turn, controls the accuracy of $\cK$ as pointed out in \eqref{Keta}. $\bkk^\Psi_L$ is computed in a preprocessing step
but could later be updated due to the hierarchical nature of $\Psi$.

Then for any larger tolerance $\eta \ge \e$, arising in the outer iteration, one can combine the compression rules from \cite[Section 4]{DHS2006} with the decay estimates in Proposition \ref{prop:Gcompress} such that the resulting compressed matrix $\bkk^\Psi_\eta$
satisfies
\begin{equation}
\label{wish2}
\|\bkk^\Psi -\bkk^\Psi_\eta\|\le \eta \quad \Leftrightarrow\quad \|\cK_0-\cK_{0,\eta}\|_{\cL(\S,\S)}\le \eta.
\end{equation}
Roughly speaking,
the computational cost for applying $\cK_0$ to an element in $\PP_M(\pS)$
scales like $\#\pS \cdot (\log \#\S)^a$ for some $a>0$.
A first variant of $[\cK_0,\cdot;\cdot]$ is then given by
\begin{align}
\label{first}
[\cK_0,v;\eta] &= \cK_{0,\eta'}v , &
\eta' &\coloneqq \eta /\|v\|_{L_2(\S)},
\end{align}
where the compressed operator $\cK_{0,\eta'}$ is applied \emph{exactly}. In fact, since the approximations
$\bar u$ use the same piecewise polynomial degrees as the kernel representations, orthonormality yields for
$\bar u(x,\s)= \sum_{T\in \pD_h, i\in \cI_T}\Big(\sum_{\lambda\in\Lambda_\pS}
v_\lambda^{T,i}\psi_\lambda(\s)\Big) \varphi_{T,i}(x)$ the scattering
\begin{equation}
\label{Ku}
(\cK_{0,\eta'} \bar u)(x,\s) = \sum_{T\in \pD_h, i\in \cI_T} w_{T,i} (\s)   \varphi_{T,i}(x),
\text{ with }
w_{T,i} (\s) = \sum_{\lambda\in \Lambda_{\eta'}}\Big(\sum_{\lambda'\in \Lambda_\pS}(\kk^\Psi_\eta)_{\lambda,\lambda'}v^{T,i}_{\lambda'}\Big)\psi_{\lambda}(\s),
\end{equation}
where $\Lambda_\eta$ contains the range of indices of $\bkk^\Psi_\eta$. Thus, the $\s$-dependent coefficients $w_{T,i}$ are obtained
by compressed matrix-vector multiplication.

In summary, the computational cost of the resulting routine $[\cK,\bar u;\eta]$
can be reduced to $\ord(\#\pD\cdot\#\pS\cdot(\log(\#\pS))^a)$, where of course $\#\pD$ and $\#\pS$ depend on $\eta$,
typically in an algebraic fashion. For the Henyey--Greenstein kernel such schemes are still
effective when the parameter $\gamma$ gets close to one, see Figure \ref{fig:finger}.

\subsection{
Hilbert--Schmidt expansion of \texorpdfstring{$\kk$}{the scattering kernel}}\label{ssec:HilbertSchmidt}
There is an alternative
way of efficiently applying the scattering operator when the parameter $\gamma$ in the Henyey--Greenstein kernel stays bounded away from one. It uses the fact that,
by our assumptions, the kernel $\kk$ possesses a \emph{Hilbert--Schmidt} decomposition of the form
\begin{equation}
\label{HS}
\kk(\s,\s')= \sum_{k=1}^\infty \sigma_k g_k(\s)g_k(\s'),\quad \sigma_k\ge 0,\quad \sum_{k\in\N}\sigma_k^2 = \|\kk\|_{L_2(\S\times \S)}^2 \le 1,
\end{equation}
where
\begin{equation}
\label{orthonormalg}
(g_k,g_l)_\S = \delta_{k,l},\quad k,l\in \N.
\end{equation}
An approximate
Hilbert--Schmidt decomposition of $\kk$ results from the singular value decomposition (SVD) of the matrix $\bkk^\Psi_L$ from \eqref{wish}
which we denote for simplicity again as $\bkk^\Psi$.

The singular value decomposition then yields vectors
$\bg_k$
such that 
\begin{equation}
\label{HSb}
\bkk^\Psi = \sum_{k=1}^{N_\tau} \sigma_k' \bg_k\otimes \bg_k
\end{equation}
where $N_\tau$ is the rank of $\bkk^\Psi_L$ and $\bg_k $ is the vector of expansion coefficients of $g_k$ with respect to $\Psi$,
i.\,e.,
\begin{align}
\label{agree}
\sigma_k &= \sigma_k', &
g_k &= \sum_{\mu\in\nabla}g_{k,\mu}\theta_\mu \eqqcolon \bg_k^T\Psi, &
k&\in \N. 
\end{align}
We can then consider \emph{low-rank approximations} by further
truncating \eqref{HSb}
\begin{align*}
\kk^r &\coloneqq \sum_{k\le r}\sigma_k g_k \otimes g_k, &
\bkk^r &\coloneqq \sum_{k\le r}\sigma_k \bg_k \bg_k^T.
\end{align*}
This yields
\begin{equation}
\label{truncest}
\|\cK_0 - \cK_0^r\|_{\cL(L_2(\S),L_2(\S))} = \|\bkk^r- \bkk\|_{\cL(\ell_2,\ell_2)} = \sigma_{r+1}.
\end{equation}
The application of the truncated operator $\cK_0^r$ for coarser accuracy tolerances, however, requires further reduction
compressing the arrays $\bg_k$.
The coefficient vectors $\bg_k$, consisting of wavelet coefficients, can easily be compressed by thresholding
providing best $n$-term approximations of desired accuracy. In particular, notice that $\|\bg-\tilde\bg\|_{\ell_2}\le \delta$ implies that
\begin{equation*}
\|\bg \bg^T - \tilde\bg \tilde\bg^T\|
= \|\bg \bg^T - \tilde\bg \tilde\bg^T\|_{\cL(\ell_2,\ell_2)}
\le 2\delta.
\end{equation*}
Thus, thresholding for a given tolerance $\eta$ the basis vectors $\bg_k $ so as to obtain approximations $\bg_{k,\eta}$
satisfying
\begin{equation}
\label{gcompress}
\|\bg_k - \bg_{k,\eta}\|_{\ell_2} \le \frac{\gamma_k \eta}{2\sigma_k},
\end{equation}
with positive weights $\sum_k \gamma_k \leq 1$, one can verify that for
the truncated kernel $\bkk^{\Psi,r}_\eta \coloneqq
\sum_{k=1}^r \sigma_k \bg_{k,\eta}\bg_{k,\eta}^T$ one has
\begin{equation}
\label{Grcompress}
\|\bkk^{\Psi,r} - \bkk^{\Psi,r}_\eta\| \le \eta.
\end{equation}
(Updating the SVD for $\bkk^{\Psi,r}_\eta$ would improve stability.)
As a consequence one obtains for the corresponding operator approximation $\cK_0^{r,\eta}$ and a given
$v(\s) = \sum_{\lambda\in \Lambda_{\pS}}v_\lambda\psi_\lambda(\s)$
\begin{equation}
\label{Kappr}
\|(\cK _0- \cK_0^{r,\eta})v\|_{L_2(\S)} \le \Big\{\sigma_{r+1}\|v\|_{L_2(\S)} +
 \eta\Big(\sum_{\lambda\in \Lambda_{\pS}}|v_\lambda|_{L_2(\D)}^2\Big)^{1/2}\Big\} = (\sigma_{r+1} +\eta)\|v\|_{L_2(\S)}.
\end{equation}
Hence, choosing $r$ such that $\sigma_{r+1}\le \frac{\eta}{\|v\|_{L_2(\S)}}$, $\eta'\le \frac{\eta}{\|v\|_{L_2(\S)}}$, with this variant we take
\begin{equation}
\label{2nd}
[\cK_0,v;\eta] \coloneqq \cK_0^{r,\eta'} v.
\end{equation}

This strategy is particularly efficient when the singular values $\sigma_k$ decay rapidly. For the Henyey--Greenstein kernel, as illustrated
in Figure \ref{fig:svd-G}, this is the more the case the larger $1-\gamma$.

\section{The routine \texorpdfstring{$[\Tr^{-1},{F};\eta]$}{[T\{-1\},g;eta]}}
\label{sec:5}
The numerical realization of the routine $[\Tr^{-1},\cdot;\cdot]$ is based on solving \emph{fiber problems}
\begin{equation}
\label{fiber}
\Tr_{\s} u \coloneqq \s\cdot \nabla u + \sigma(\s)u = \int_{\S} K(\cdot,\s,\s')v(\cdot,\s')\,\ds' + f \eqqcolon F(\s),\quad \s \in \S,
\end{equation}
for properly selected parameters $\s\in \S$ where $F\in L_2(\D\times \S)$ is given. Achieving a given target accuracy depends on solving each fiber problem
with sufficient accuracy and also on solving sufficiently many of them.

The approximate solution of \eqref{fiber} will be based on the Discontinuous Petrov--Galerkin (DPG) scheme developed and analyzed in
\cite{BDS2017,DS2019} whose main features we briefly recall for the convenience of the reader in Sections \ref{ssec:DPG} and \ref{ssec:apost}. In Section \ref{ssec:lift}, we explain how to use the set of solutions to the fiber problems in order to adaptively build an approximation to $u$ in $L_2(\D\times\S)$ which will be the output of $ [\Tr^{-1},F;\eta]$.
\subsection{A DPG Transport Solver for the fiber problems}\label{ssec:DPG}

We outline the numerical transport solver that is the core constituent of the current realization of $[\Tr^{-1},{F};\eta]$.
We denote  by $\pD_h$, $h>0$ a family of uniformly shape regular partitions of the spatial domain $\D$.
More specifically,
in what follows we always assume that all spatial partitions $\pD_h$ are (possibly local) refinements of a hierarchy
of dyadic partitions of $\D$. These partitions therefore induce dyadic partitions of the boundary $\partial \D$ as well.

While typically $h$ stands for a mesh size parameter in a quasi-uniform mesh, here
 $h$ is a locally varying mesh size function covering local refinements of the above dyadic hierarchy. With a given $\pD_h$ we
associate the skeleton $\skelh$, which however depends strictly speaking on an associated convective direction $\s\in\S$.
In fact, in analogy to \eqref{Gammas}, for a given $\s\in\S$ we define $\partial T_\pm(\s)$ for any given cell $T\in \pD_h$ and set
\begin{equation*}
\skelh= \skelh(\s) \coloneqq \bigcup\,\{\partial T_-(\s), T_+(\s) \mid T\in \pD_h\},
\end{equation*}
suppressing at times the dependence of $\skelh$ on $\s$.
Note that for polyhedral domains
$\Gamma_-(\s)$ remains the same on certain neighborhoods in $\S$.

Following \cite{BDS2017}, the   DPG scheme is based on the \emph{infinite dimensional} mesh-dependent variational formulation
over the trial and test space
\begin{equation*}
\U_{\s}\coloneqq L_2(\D) \times H_{0,\Gamma_-(\s)}(\s;\skelh), \quad \V_{\s}\coloneqq H(\s;\pD_h) = \prod_{T\in\pD_h}H(\s;T),
\end{equation*}
endowed with the norms
\begin{equation}
\label{skeletonnorm}
\|\theta\|_{H_{0,\Gamma_-(\s)}(\s;\skelh)}\coloneqq \inf_{w\in H_{0,\Gamma_-(\s)}(\s;\D):\ \restr{w}{\skelh}=\theta}\|w\|_{H(\s;\D)},\quad
\|v\|_{H(\s;\pD_h)}^2 \coloneqq \sum_{T\in\pD_h} \|v\|_{H(\s;T)}^2,
\end{equation}
where as before $\|v\|_{H(\s;T)}^2 = \|v\|_{L_2(T)}^2 + \|\s\cdot\nabla v\|^2_{L_2(T)}$.
Recall from~\cite{BDS2017} that the introduction of the additional unknown field $\theta\in H_{0,\Gamma_-(\s)}(\s;\skelh)$,
living on the skeleton $\skelh$,  is necessary because
the trace terms encountered in the usual derivation of DG bilinear forms may not exist for general elements in $L_2(\D)$.
 \begin{remark}
\label{rem:consistent}
The spaces $\U_{\s}$, $\V_{\s}$ depend on the directions $\s$ and on  $\pD_h$, and so will the solution $[u(\s),\theta(\s)]$. However, when
the solution component $u(\s)$ is regular enough, i.\,e., $u(\s)\in H_{0,\Gamma_-(\s)}(\s;\D)$, one can show that
$u(\s)$ is the solution of \eqref{fiber} and $\theta(\s)$ is its trace on $\skelh$.
\end{remark}
%

%
Defining
\begin{equation}
\label{bilinear}
b_h(u,\theta,v;{ \s})=\sum_{T\in\pD_h}
\underbrace{ \int_T(\sigma(\s)v- { \s}\cdot\nabla v)u \,\dx + \int_{\partial T}\n\cdot \s \theta v \,\dGamma}_{\eqqcolon b_T(u,\theta,v;\s)},
\end{equation}
and given $F(\s)\in L_2(\D)$, we then wish to find $u(\s)\in L_2(\D)$, $\theta \in H_{0,\Gamma_-(\s)}(\s;\skelh)$ such that
\begin{equation}
\label{DPGform}
b_h(u(\s),\theta(\s),v;{ \s})= \int_\D F(\s)v\,\dx,\quad v\in \V_{\s}= H(\s;\pD_h).
\end{equation}
\begin{remark}
\label{rem:well-posed} It immediately follows from
\cite[Theorem 3.1]{BDS2017} that \eqref{bilinear} is a uniformly stable variational formulation for the transport
equation $\Tr_\s u_\s = F(\s)$, i.\,e., continuity and inf-sup conditions according to Theorem \ref{Th:BNB} hold uniformly in $\s\in \S$ and
in $\pD_h$.
\end{remark}
\paragraph{A fully discrete scheme:}
The discretization of \eqref{DPGform} requires two hierarchies of partitions $\pD_\uh$, $\pD_h$ where the
$\pD_h$ is a refinement of (locally) constant depth) of $\pD_\uh$, i.\,e., $\pD_\uh\prec \pD_h$.
 (In fact, practical experiments usually indicate that depth-0 suffices, i.\,e., $h=\uh$.)
In that sense we can write $\uh=\uh(h)$ and $h=h(\uh)$. Given $\pD_\uh$, $\pD_h$, we fix a polynomial degree $m\in\N$
and consider the finite dimensional trial spaces
\begin{equation}
\label{Uh}
\U^\uh_{ \s} \coloneqq \Big( \prod_{\uT\in\pD_\uh}\PP_{m-1}(\uT)\Big) \times
\restr{\Big(H_{0,\Gamma_-({ \s})}({ \s};\D) \cap
            \prod_{\uT\in\pD_\uh}\PP_m(\uT)\Big)}{\skelh}.
\end{equation}
Note that the second component consists of traces of globally continuous piecewise polynomials of one degree higher than
for the discontinuous bulk-component but evaluated on the skeleton of the (possibly) \emph{finer} mesh $D_h$.

Given the finite dimensional trial space $\U_{\s}^\uh$, it is critical to construct a suitable \emph{test space} that
renders also the finite dimensional corresponding Petrov--Galerkin problem inf-sup stable, ideally with inf-sup constants independent
of the trial and test space dimensions. We follow again \cite{BDS2017} and fix the so called \emph{test search space}
 as
discontinuous piecewise polynomials of one  degree higher on a \emph{subgrid} $\pD_h$ of $\pD_\uh$, namely
\begin{equation}
\label{testsearch}
\hat\V_{\s}^h\coloneqq \prod_{T\in \pD_h}\PP_{m+1}(T).
\end{equation}
The actual \emph{test space} $\V^h_{\s}$ is then defined as the following
$H(\s;\pD_h)$-projection 
to the \emph{test search} space $\hat\V_{\s}^h$
\begin{equation}
\label{trialtotesth}
\V_{\s}^h\coloneqq \big\{\breve t(u,\theta) \in \hat\V_{\s}^h \mid (\breve t(u,\theta), v)_{\V_{\s}} = b_h(u,\theta,v;\s),\,\, v\in \hat\V_{\s}^h
,\, [u,\theta]\in \U^h_{\s}\big\}.
\end{equation}
Since the local test search spaces over each cell $T\in \pD_{\uh}$ have uniformly bounded finite dimension the overall computational work
still remains proportional to the dimension of the trial spaces.


This gives rise to the \emph{Petrov--Galerkin} formulation:
find $[u_\uh(\s),\theta_\uh(\s)]\in \U^{\uh}_{\s}$ such that
\begin{equation}
\label{DPGh}
b_h(u_\uh(\s),\theta_\uh(\s),v_h;\s) = \int_\D F(\s)v \,\dx \eqqcolon F(\s)(v),\quad v\in \V^h_{\s},
\end{equation}
for $\V_{\s}^h$ defined by \eqref{trialtotesth}. Here and below we sometimes use the shorthand notations
$u_\uh = u_{\pD_\uh}, b_h=b_{\pD_h}, U^\uh= U^{\pD_\uh}$.

Before stating the corresponding stability result, we mention a variant where the skeleton component $\theta_\uh(\s)$
is replaced by the globally conforming piecewise polynomial $w_\uh$ in $H_{0,\Gamma_-({ \s})}({ \s};\D) \cap
     \prod_{T\in\pD_\uh}\PP_m(\uT)\Big)\restr{\vphantom{\Big)}}{\skelh}.$
Then the local bilinear forms $b_T(u_\uh(\s),\theta_\uh(\s),v_h;\s)$ from \eqref{bilinear} can be rewritten as
\begin{align}
\label{bK}
b_T(u_\uh ,\theta_\uh ,v_h;\s)& = b_T(u_\uh ,w_\uh ,v_h;\s)
= \int_T(\sigma(\s) v_h -\s\cdot\nabla v_h)u_\uh \,\dx + \int_{\partial T} \n\cdot\s w_\uh v_h \,\mathrm d\Gamma\nonumber\\
&= \int_T \sigma (\s)v_h(u_\uh -w_\uh)+ \partial_\s v_h(w_\uh-u_\uh) + (\sigma w_\uh+\partial_\s w_\uh)v_h \,\dx,\quad T\in \pD_h.
\end{align}
Using $[u_\uh, w_\uh]$ as unknowns one obviously has
$
\|w\|_{H_{0,\Gamma_-(\s)}(\s;\partial \pD_h)} \le \|w\|_{H(\s;\D)}$.
We will adopt this variant in what follows where it is now understood to use the norm
\begin{equation}
\label{now}
\|[u_\uh, w_\uh]\|_{\U_\s}^2 \coloneqq \|u_\uh\|_{L_2(\D)}^2 + \|w_\uh\|_{H(\s;\D)}^2.
\end{equation}
The following facts are immediate consequences of
the results in \cite{BDS2017,DS2019}.
\begin{theorem}
\label{thm:DPGh}
For a fixed but sufficiently large subgrid-depth $\uh/h$, (depending on the shape parameters of the involved partitions) the scheme
\eqref{DPGh} is uniformly in $h\ge 0$, $\s\in \S$, inf-sup stable, i.\,e.,
\begin{equation}
\label{inf-suph}
\inf_{[u_\uh,w_\uh]\in \U^h_{\s}}\sup_{v_h\in \V^h_{\s}}\frac{b_h(u_\uh,w_\uh,v_h;\s)}{\|[u_\uh,w_\uh]\|_{\U_{\s}} \|v_h\|_{\V_{\s}}}\ge \bar\beta >0,\quad h\ge 0, \s\in \S,
\end{equation}
where $\bar\beta$ depends on the shape parameters of the underlying partitions, on $\|\Tr_{\s}^{-1}\|_{\cL(L_2(\D),H_{0,\Gamma_-(\s)}(\s;\D))}$
and on $\|\sigma\|_{L_\infty(\S,W^1(L_\infty(\D)))}$.
\end{theorem}

It is well known that the system matrices arising in \eqref{DPGh} are always \emph{symmetric positive definite} despite the
asymmetric nature of transport equations.

While the conforming formulation \eqref{F1} does not require incorporating boundary conditions on $\Gamma_-$ into the trial space,
the skeleton component requires an adjustment in the DPG formulation. To that end, following \cite[Remark 3.6]{BDS2017},
let $w_0(\s)\in H(\s;\D)$ satisfy $w_0(\s)= g(\s)$ on $\Gamma_-(\s)$. Then, the (infinite-dimensional) DPG formulation
of the problem
$\Tr_\s \bar u = f - \Tr_\s w_0$, in $\D$,
$\bar u = 0\ \text{in $\Gamma_-(\s)$}$,
is given by
 \begin{equation}
 \label{bc}
 b_h(\bar u (\s), \bar w (\s),v ;\s) = \langle f,v \rangle - b_h(\bar w_0(\s), \bar w_0(\s),v ;\s) \eqqcolon \langle f- F_b(w_0,\s), v\rangle, \quad v\in \V.
 \end{equation}
Now one has
$
\bar w|_{\partial\pD_h}= \bar u|_{\partial \pD_h} = (u- w_0)|_{\partial\pD_h}$,
i.\,e., it suffices to discretize \eqref{bc}.
\subsection{A Posteriori Error Estimates}\label{ssec:apost}
As an immediate consequence of the fact that the DPG-induced transport operators
$\Tr_{\s,h}$ are norm isomorphisms, uniformly in $h\ge 0, \s\in\S$, 
errors
in $\|\cdot\|_{\U_\s}$ are equivalent to residuals in $\|\cdot\|_{\V_\s'}$, i.\,e.,
\begin{equation}
\label{resAPosteriori}
\bigl\|[u(\s),u(\s)] - [u_\uh(\s),w_\uh(\s)]\bigr\|_{\U_{\s}}
\sim \bigl\| F(\s) - \Tr_{\s,h} ([u_\uh(\s),w_\uh(\s)])\bigr\|_{\V_{\s}'},
\quad h\ge 0,\ \s\in\S,
\end{equation}
holds with uniform constants. Thus, as soon as one can tightly estimate the dual norm
$\| F(\s) - \Tr_{\s,h} ([u_\uh(\s),w_\uh(\s)])\|_{\V_{\s}'}$ of the residual, one obtains
efficient and reliable a posteriori error bounds. Such tight bounds are established in \cite{DS2019} which we briefly
recall.
 Define for $T\in \pD_h$ the Riesz lifts $\breve R_T(u_\uh,w_\uh,\bar F(\s))$ of the local residuals by
\begin{equation}
\label{RT}
\bigl(\breve R_T(u_\uh,w_\uh,\bar F(\s)),v_h\bigr)_{H(\s;T)} = b_T(u_\uh,w_\uh,v_h;\s)- \bar F(\s)(v_h),\quad v_h\in \hat\V^h_{\s},
\end{equation}
where $\bar F(\s)|_T\in \PP_m$ is a piecewise polynomial approximation to $F(\s)$ and where $\hat\V^h_{\s}$ is the same
test search space as used before for the Petrov--Galerkin scheme. Thus, the computational cost per cell $T$ is again uniformly
bounded. Defining then
\begin{equation}
\label{Rh}
\|\breve R_\uh(u_\uh,w_\uh,\bar F(\s))\|^2_{H(\s;\pD_\uh)}= \|\breve R_{\pD_\uh}(u_\uh,w_\uh,\bar F(\s))\|^2_{H(\s;\pD_\uh)}\coloneqq \sum_{T\in \pD_\uh}\|\breve R_T(u_h,w_h,\bar F(\s))\|^2_{H(\s;T)},
\end{equation}
the following holds, see \cite[Theorem 4.1 and (4.4)]{DS2019}.
\begin{theorem}
\label{thm:apost}
If the operators $\Tr_{\s,h}$ are norm isomorphisms uniformly in $h\ge 0$ and $\s\in\S$,
then for a fixed maximal subgrid depth there exist constants $\underline c$, $\bar C$, depending
on $\bar\beta$ from \eqref{inf-suph}, but independent of $\s$, $\pD_h$, such that
\begin{equation}
\begin{aligned}
\label{eq:residualbounds}
\underline c\|\breve R_h(u_\uh,w_\uh,\bar F(\s))\|_{H(\s;\pD_h)}
&\le \|[u(\s),u(\s)|_{\skeluh}] - [u_\uh(\s),w_\uh(\s)]\|_{\U_{\s}} \\
& \le \bar C \|\breve R_h(u_\uh,w_\uh,\bar F(\s))\|_{H(\s;\pD_h)}.
\end{aligned}
\end{equation}
\end{theorem}
In the present context it is particularly important to control the dependence of a posteriori bounds on the direction
parameter $\s\in \S$. In this regard, the following further result from \cite[Proposition 4.4]{DS2019} is relevant: there exists a constant
$c_0 > 0$ such that the Petrov--Galerkin solution satisfies for each $T'\in \pD_\uh$
\begin{equation}
\label{newres}
\begin{aligned}
\MoveEqLeft
c_0 \Big(\|u_\uh(\s) - w_\uh(\s)\|^2_{L_2(T')}
         + \|\s\cdot\nabla w_\uh(\s) + \sigma u_\uh(\s)
                                     - \bar F(\s)\|^2_{L_2(T')}\Big) \\
&
\le \sum_{T\in \pD_h, T\subset T'}
       \|\breve R_T(u_\uh,w_\uh,\bar F(\s))\|^2_{H(\s;T)} \\
& \le \|u_\uh(\s) - w_\uh(\s)\|^2_{L_2(T')}
      + \|\s\cdot\nabla w_\uh(\s) + \sigma u_\uh(\s) - \bar F(\s)\|^2_{L_2(T')}.
\end{aligned}
\end{equation}
For $d=2$, i.\,e., $\S$ is the circle we can identify $\s = (\cos t,\sin t)^\top$ and the space $\PP_M(\pS)$ consists for a given admissible
partition $\pS$ of $\S$ of $2\pi$-periodic piecewise polynomials in $t\in (-\pi,\pi]$. Hence, the above error indicators are nearly piecewise
polynomial in $t$ when the components $u_\uh$, $w_\uh$ are of the form \eqref{inputbaru} with $\s$-dependent coefficients in $\PP_M(\pS)$, see
Section \ref{ssec:Alpert}.

The above DPG scheme and the associated a posteriori error bounds form the core constituent of the routine
$[\Tr^{-1},\cdot;\cdot]$. \eqref{eq:residualbounds} can be used to
contrive adaptive mesh refinement strategies based on
so called \emph{Dörfler marking} or \emph{bulk chasing}.
This means one marks those cells for subsequent refinement whose
combined energy exceeds a fixed portion of the total lifted residual. It is shown in \cite{DS2019} that this
entails a fixed error reduction for each refinement sweep and associated complexity estimates.

\begin{remark}
\label{rem:error}
Convergence to zero of either one of the above residual error bounds guarantees convergence of errors in the spaces $U^\s$.
The DPG output has two components, namely a piecewise polynomial $u_h$ of degree $m$ on the underlying mesh $\pD_h$
as well as a skeleton component which can be identified with the trace of a conforming piecewise polynomial of degree $m+1$.
Therefore the
a posteriori error bounds control in particular the convergence of the $u$-component in $L_2(\D)$.
For the realization of $[\Tr^{-1},F;\eta]$ below we always use only the $u$-component for the outer iteration.
\end{remark}

\subsection{An Adaptive Solver in
            \texorpdfstring{$U=L_2(\D\times\S)$}{the Phase Space}}
\label{ssec:lift}

We describe next how $[\Tr^{-1},\cdot;\eta]$ is realized based on approximately solving, with the aid of the DPG scheme
described above, fiber problems
$\Tr_\s \bar u = F$ for the elements $\s$ from a stage-dependent discrete subset $\cQ_\eta$ of the parameter domain $\S$.
Both $\cQ_\eta$ as well as the meshes for each fiber solution are generated adaptively.

\vspace*{-4mm}
\paragraph{The data:}
The data $F= F(x,\s)$ required by each call of $[\Tr^{-1},F;\eta]$ have a piecewise polynomial representation of the type \eqref{Ku}.
Specifically, they are of the form
\begin{equation}
\label{F}
F = w + g\in L_2(\D\times \S),
\end{equation}
where $w$ is the output of the routine $[\cK,\cdot;\cdot]$ and $g$ is a stage-dependent
approximation to the source term. More precisely, in the case of inhomogeneous boundary conditions $g$ consists of two parts,
namely
$g=g_0+ g_1$ where $g_0$ stands for the ``lifted boundary data'' needed to correct the right hand side so as to reduce
the problem to the homogeneous case, see \eqref{bc}. Both $w$ and $g$ need to be computed within the currently given
accuracy tolerance. We omit the details concerning the computation of $g$.
\vspace*{-3mm}
\paragraph{Output format:} The output of $[\Tr^{-1},\cdot;\eta]$ is a piecewise polynomial of degree $m$ of the form (see \eqref{inputbaru})
\begin{equation}
\label{Toutput}
\bar u_\eta(x,\s) = \sum_{T\in \pD_\eta}\sum_{i\in \cI_T} v_{T,i}(\s) \varphi_{T,i}(x),
\end{equation}
where the $\varphi_{T,i}$ are polynomial basis functions of degree $m$ supported in $T\in \pD_\eta$ and $\pD_\eta$ is
a partition of the spatial domain $\D$. The parameter dependent coefficients $v_{T,i}(\s)$ are elements of
a space $\PP_M(\pS_{\eta})$ of piecewise polynomials of degree $M$ subordinate to a partition $\pS_\eta$ of $\S$.
We describe next how to compute the $v_{T,i}(\s)$ as well as the partition $\pD_\eta$.
\vspace*{-3mm}
\paragraph{Computation of fiber solutions:} The realization of $[\Tr^{-1},F;\eta]$ is based on approximately solving
fiber transport problems $\Tr_\s u_\s = F(\cdot,\s)$ for parameters $\s$ in a suitable finite subset of $\S$,
Specifically, given a partition $\pS$ of the parameter domain $\S$,
we associate
with each cell $\cell\in \pS$ a set of ``quadrature points'' $\cQ_\cell$ whose union
\begin{equation}
\label{Qeta}
 \cQ_{\pS} \coloneqq \bigcup_{\cell\in \pS} \cQ_\cell
\end{equation}
is the discrete set of parameters for which we first compute error controlled approximate fiber solutions. Before describing this
in more detail,
a few preparatory comments are in order. The realization of $[\cK,\cdot;\cdot]$ is reduced to a frequent but efficient approximate
application of a global operator acting in functions in $d-1$ variables. The bulk of computation therefore lies in $\#\cQ_{\pS}$
approximate \emph{inversions} of transport boundary value problems in $d$ variables. It is therefore of primary importance
to keep the size of each fiber transport problem as small as possible. In view of the inherently low regularity of the transport solutions
(especially in the presence of rough boundary and source data) we opt for employing an adaptive DPG scheme for each fiber problem.
The price to be paid is that then each fiber solution $\bar u_{\pS}(\cdot,\s)$, $\s\in \cQ_{\pS}$, comes with its own adaptive
partition $\pD_{\s}$, see Figure \ref{fig:s-dependent}.
We refer to \cite{BDS2017,DS2019} for the details on an adaptive fiber transport solver
\begin{align*}
[\Tr_\s^{-1},F;\eta] &\to (\pD_{\s},\bar u_\s), &
\bar u_\s(x) &= \sum_{T\in \pD_{\s}, i\in \cI_T}
                c_{T,i,\s}\varphi_{T,i}(x).
\end{align*}
It consists in repeating the standard cycle
\begin{center}
$\textsc{Mark} \quad \rightarrow \quad \textsc{Refine} \quad \rightarrow \quad \textsc{Solve}$
\end{center}
until the sum of squared indicators (in either \eqref{eq:residualbounds} or \eqref{newres}) is below the current
threshold $\eta^2$.
Here one needs for each $\cell\in \pS$ a good initial guess. If $\cell\in\pS$ was already obtained in the representation
of the final DPG solution of the previous outer iteration we choose this one. Otherwise one can take the union of those
fiber meshes associated with those parameter cells from the preceding outer iteration that intersect the current parameter cell.

\dw{For $\textsc{Mark}$ we use a simple bulk criterion identifying for each selected quadrature point $\s$ a possibly small set of cells in the current partition such that the sum of the corresponding squared indicators exceeds  a fixed portion of the full sum of squared indicators.
Hence, the adaptively generated meshes depend on the directions $\s$.}  However, the approximate application of the scattering kernel in
$[\cK,\cdot;\cdot]$ requires an \emph{aggregated} approximate solution
$\bar u(x,\s)$ as a function of the spatial and parametric variables which needs to be represented on a single mesh that
is obtained by \emph{merging} the parameter-dependent fiber meshes.
Note that even the merged mesh involves a total number of degrees of freedom which is significantly smaller than
that corresponding to a uniform mesh with the highest required resolution, see the rightmost picture in Figure \ref{fig:s-dependent}.

A more detailed algorithmic description is beyond the present scope and can be found
in \cite[Section 6.3.2]{Gruber2018}.
\vspace*{-3mm}
\paragraph{Aggregating fiber Solutions:} We discuss first how to generate an approximate solution $\bar u_{\pS}\in L_2(\D\times\S)$
which is only based on approximate fiber solutions for $\s\in \cQ_{\pS}$ where at this point $\pS$ is a \emph{given} partition of $\pS$, e.\,g.\ generated
by an error controlled approximate application of $\cK$.
This can be formulated as
a (preparatory) routine
\begin{align}
\label{prepu}
[\Tr^{-1},F,\pS;\eta] &\to (\pD_{\pS},\bar u_{\pS}), &
\bar u_{\pS}(x,\s) &= \sum_{T\in \pD_{\pS},i\in\cI_T} v_{T,i}(\s)\varphi_{T,i}(x),
\end{align}
that outputs a mesh $\pD_{\pS}$ and a piecewise polynomial $u_{\pS}(x,\s)$ in $x$ subordinate to $\pD_{\pS}$
with parameter dependent coefficients $v_{T,i}\in \PP_M(\pS)$ and a spatial mesh $\pD_{\pS}$ such that
\begin{equation}
\label{resreq}
\|\breve R_{\D_{\pS}}(u_{\pS},\theta_{\pS}, F(\s))\|_{H(\s;\pD_h)}\le \kappa_\Tr \eta, \quad \s \in \cQ_{\pS}.
\end{equation}

The workhorse called by $[\Tr^{-1},F,\pS;\eta]$ is therefore the following subroutine providing a parameter dependent approximate transport solution
over a given cell $\cell$ in the current parameter partition $\pS$:\\

\noindent
$[\cell,F;\eta] \to (\pD_{\cell}, \bar u_{\cell})$
\begin{itemize}
\item[C1:]
For $\s\in \cQ_\cell$ invoke $[\Tr_\s^{-1},F;\eta]$;
\item[C2:]
generate the mesh $\pD_{\cell}$ by merging the meshes $\pD_{\s}$, $\s\in\cQ_\cell$ to obtain merged representations
$\bar u_\s(x)=\sum_{T\in \pD_{\cell}, i\in \cI_T}
\tilde c_{T,i,\s}\varphi_{T,i}(x)$;
\item[C3:]
Determine the polynomial $v_{\cell,T,i}(\s)\in \PP_M(\cell)$ that (quasi-)interpolates the values $\tilde c_{T,i,\s}$, $\s\in \cQ_\cell$
and aggregate
\begin{equation*}
\bar u_{\cell}(x,\s)
\coloneqq \sum_{T\in \pD_{\cell}} \tilde v_{\cell,T,i}(\s)\varphi_{T,i}(x).
\end{equation*}
\end{itemize}

The output in \eqref{prepu} of $[\Tr^{-1},F,\pS;\eta]$ is then given by
\begin{equation*}
\bar u_{\pS}(x,\s)
= \sum_{\cell\in \pS}\bar u_{\cell,\eta}(x,\s)
= \sum_{T\in \pD_{\pS}, i\in \cI_T} v_{T,i}(\s)\varphi_{T,i}(x),
\end{equation*}
where $\D_{\pS}$ is obtained by merging the cell-dependent meshes
$\pD_{\cell}$, $\cell\in \pS$ produced by $[\cell,F;\eta]$.
\vspace*{-3mm}
\paragraph{Finding $\pS_\eta$:}
The accuracy requirement in $[\Tr^{-1},F;\eta]$ requires a mean square control over the parameter domain $\S$.
The output of the routine $[\Tr^{-1},F,\pS;\eta]$ for a given parameter partition $\pS$ guarantees that the residual bounds
satisfy the required accuracy $\eta$ only at the quadrature points $\cQ_{\pS}$
 but a priori not necessarily for all parameter values in $\S$. Our current approach is therefore
to adaptively generate also a further refinement $\pS_\eta$ (if necessary) of some initial partition of $\S$ (dictated solely by
the accuracy in the application of $\cK$). We then apply quadrature with respect to $\cQ_{\pS_\eta}$ to estimate the error in $L_2(\D\times\S)$.
Here we use that by \eqref{newres}, the true errors are rigorously sandwiched by error indicators that are piecewise defined
as products of polynomials and trigonometric functions. Specifically, we apply the following steps:
\begin{itemize}
\item[S1:]
Take the partition $\pS=\pS_{\cK,\kappa_\cK\eta}$ generated by $[\cK,\bar u;\kappa_\cK\eta]$ as initial guess;
\item[S2:]
given a partition $\pS$ of $\S$ compute $\bar u_{\pS}=[\Tr^{-1},F,\pS;\eta]$;
\item[S3:]
subdivide each cell in $\pS$ to obtain a refined partition $\pS_r$;
\item[S4:]
evaluate the residual bounds (e.\,g.\ \eqref{newres}) for the current approximation $\bar u_{\pS}(\cdot,\s)$ at
the new quadrature points $\s\in \cQ_{\pS_r}\setminus \cQ_{\pS}$ and mark all cells $\cell\in \pS_r$
containing a quadrature point for which a fixed threshold $\omega\eta$ ($\omega\le 1$ fixed) is exceeded.
If no cell is marked stop and set $\pS \to \pS_\eta$;
\item[S5:]
the parents in $\pS$ of the marked cells are refined to generate a refined partition $\pS_{\rm new}$ of $\pS$;
\item[S6:]
replace $\pS$ by $\pS_{\rm new}$ and go to S2.
\end{itemize}

\begin{figure}
\includegraphics[width=.32\textwidth]
                {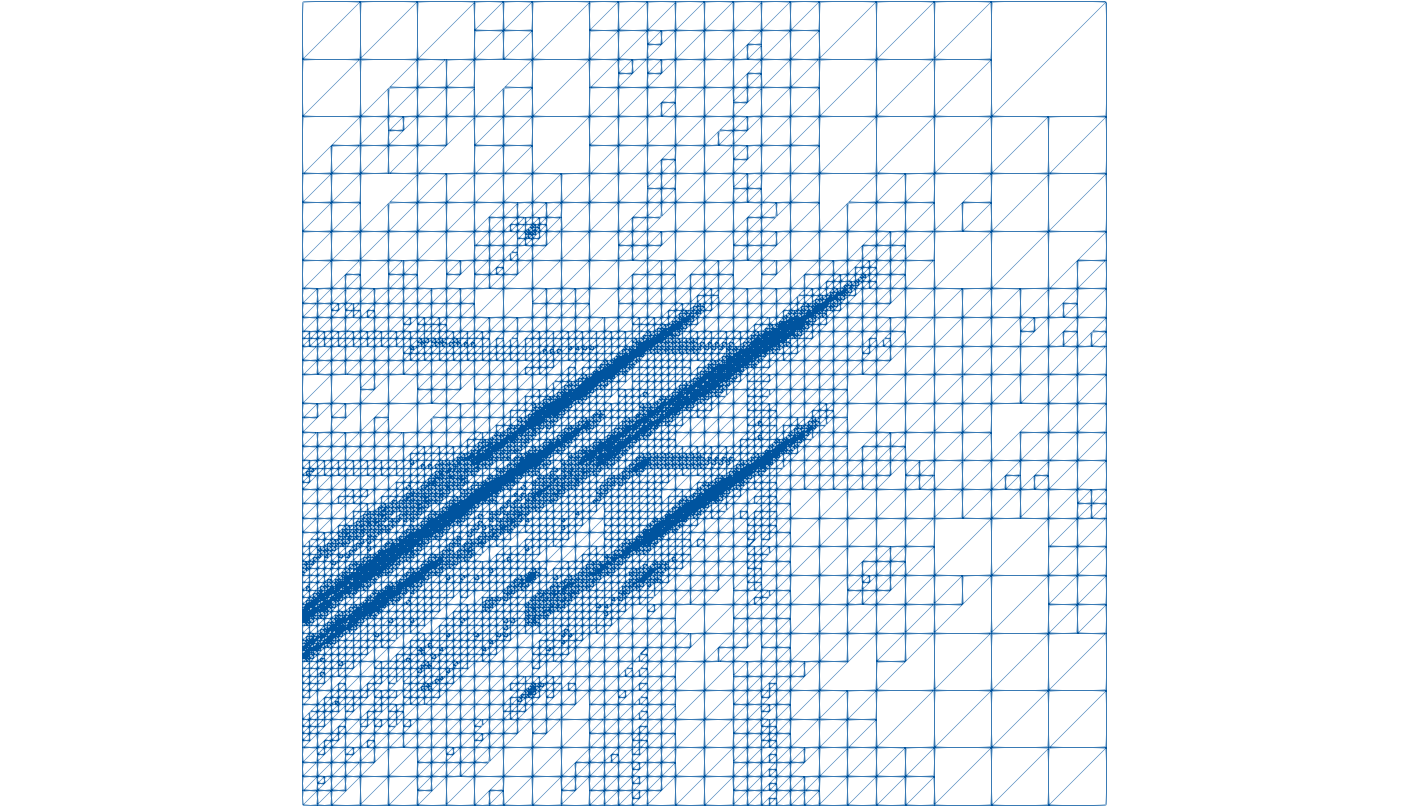}
\includegraphics[width=.32\textwidth]
                {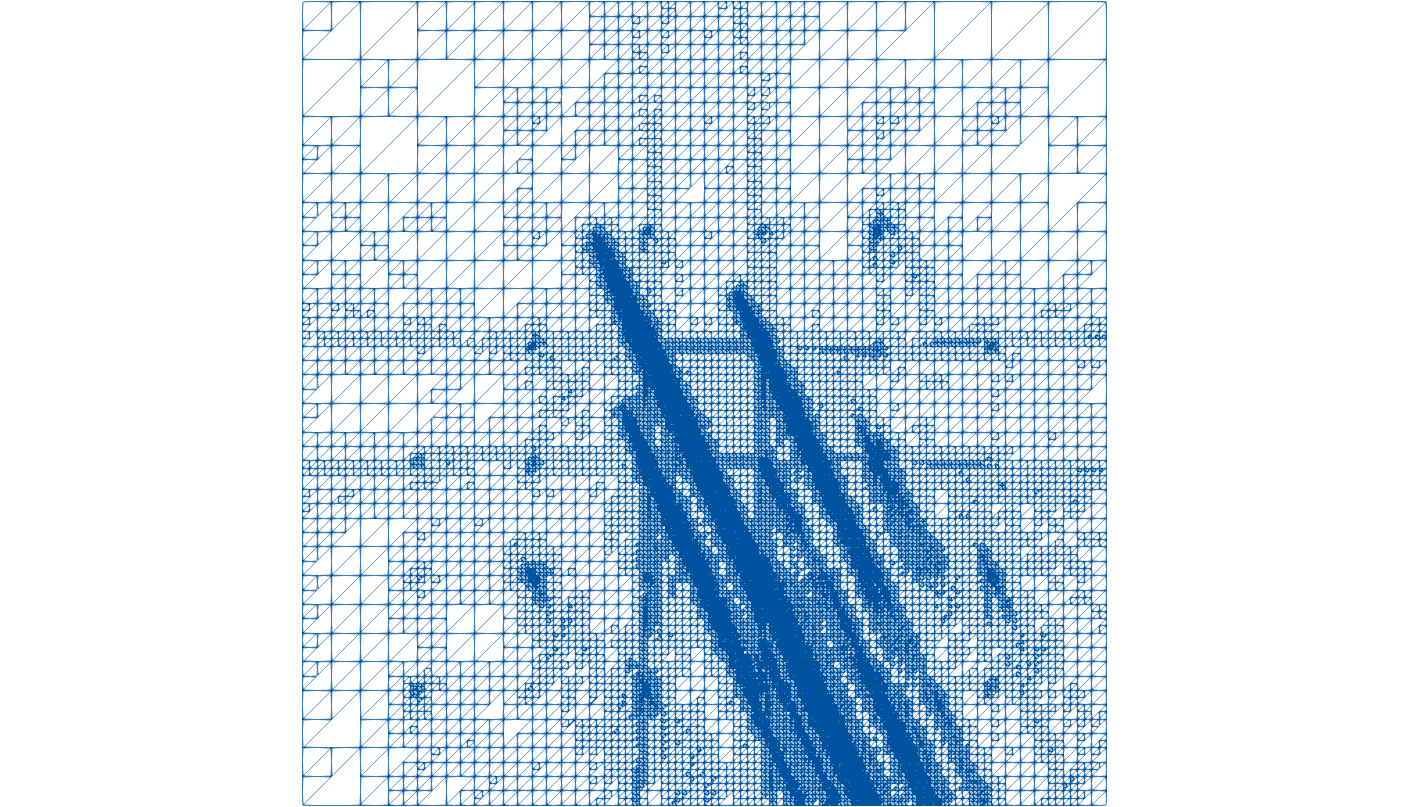}
\includegraphics[width=.32\textwidth]
                {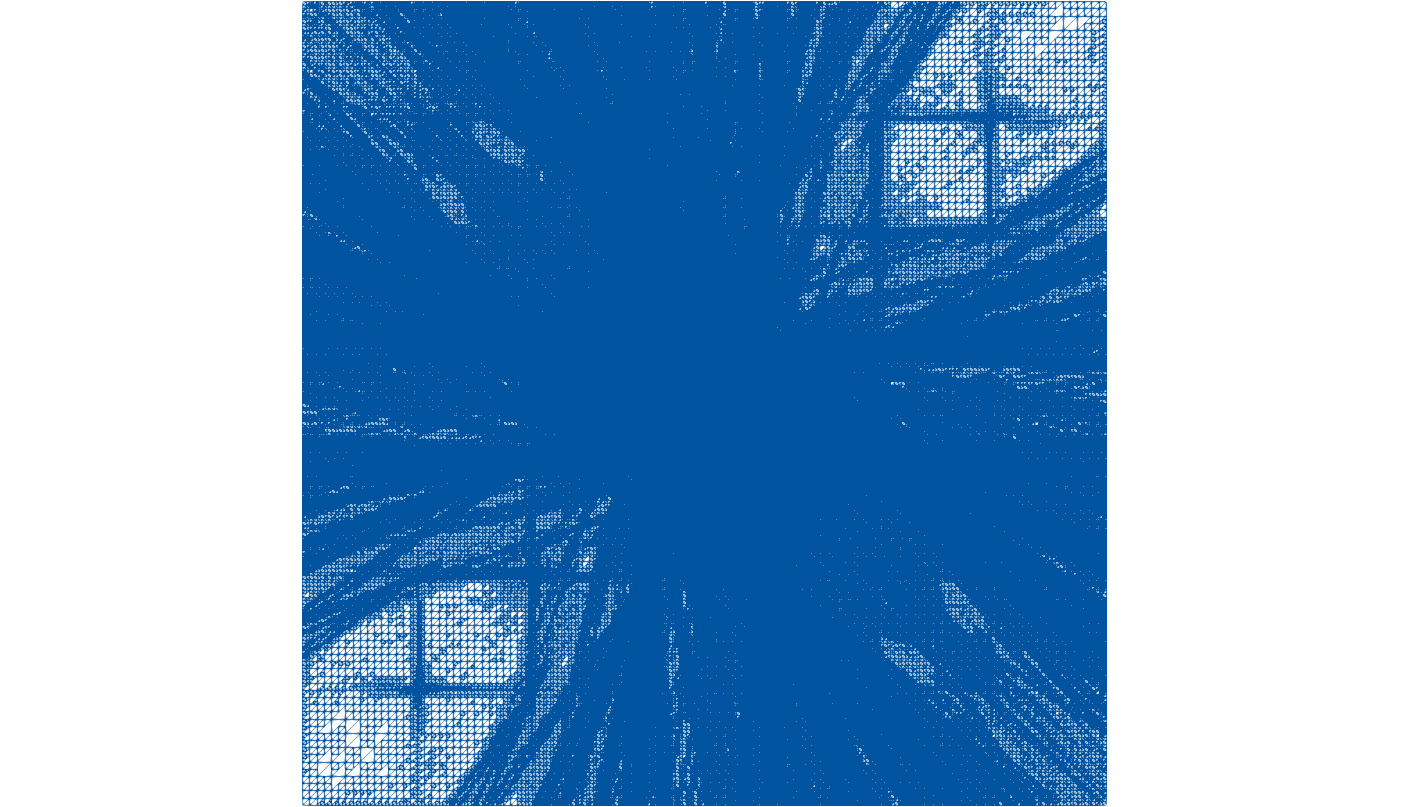}
\caption{Adaptive meshes for fiber transport solutions with respect to two different directions as well as the merged mesh at iteration step 10.}
\label{fig:s-dependent}
\end{figure}

\vspace*{-3mm}
\section{Numerical Experiments}\label{sec:numres}

We consider the radiative transfer problem \eqref{eq::radTrans} on the unit square domain $\D=[0,1]^2$
with homogeneous boundary conditions. The structure of the source term $f$ and absorption coefficient $\sigma$
is illustrated by Figure \ref{fig:checkerboard}. More precisely, we take $f=0$ in the white and gray areas whereas $f=1$
in the black area. Similarly, we set $\sigma=10$ in the gray areas and $\sigma=2$ everywhere else.
Such checkerboard structure serves as a classical benchmark in the literature of radiative transfer
and can be found in other works, see e.\,g., \cite{Brunner2002}.

The scattering is of Henyey--Greenstein type
(see
formula~\eqref{eq:henyey})
\begin{equation}
K(x,\s,\s') = \kk(\s,\s')= \frac{1}{2\pi}\frac{1-\gamma^2}{1+\gamma^2-2\gamma \s\cdot\s'}, \quad \forall x \in \D.
\label{eq:henyey-comp}
\end{equation}
\vspace*{-5mm}
\begin{figure}[h]
\centering
\begin{minipage}{.4\textwidth}
  \centering
 \begin{tikzpicture}[scale=0.64]
\coordinate (P1) at (0cm,0cm); 
\coordinate (P2) at (7cm,0cm); 
\coordinate (P3) at (7cm,7cm); 
\coordinate (P4) at (0cm,7cm); 

\draw (P1) -- (P2) -- (P3) -- (P4) -- cycle; 

\foreach \i in {1,2,...,5}
{
  \foreach \j in {1,2,...,5}
  {
    \coordinate (C1) at (\i cm, \j cm);
    \coordinate (C2) at (\dimexpr\i cm +1cm, \j cm);
    \coordinate (C3) at (\dimexpr\i cm +1cm, \dimexpr\j cm+1cm);
    \coordinate (C4) at (\i cm, \dimexpr\j cm+1cm);
    \ifthenelse{\intcalcMod{\i+\j}{2}=0 \AND \NOT \(\j=5 \AND \i=3\)
                \AND \NOT \(\j=3 \AND \i=3\)}{
      \draw[fill=gray] (C1) -- (C2) -- (C3) -- (C4) -- cycle;
    }{}
  }
}
\draw[fill=black] (3cm, 3cm) -- (4cm,3cm) -- (4cm,4cm) -- (3cm,4cm) -- cycle;
\end{tikzpicture}
\caption{Geometry of the checkerboard benchmark.}
\label{fig:checkerboard}
\end{minipage}%
\quad
\begin{minipage}{.4\textwidth}
  \centering
\includegraphics[height=5cm]{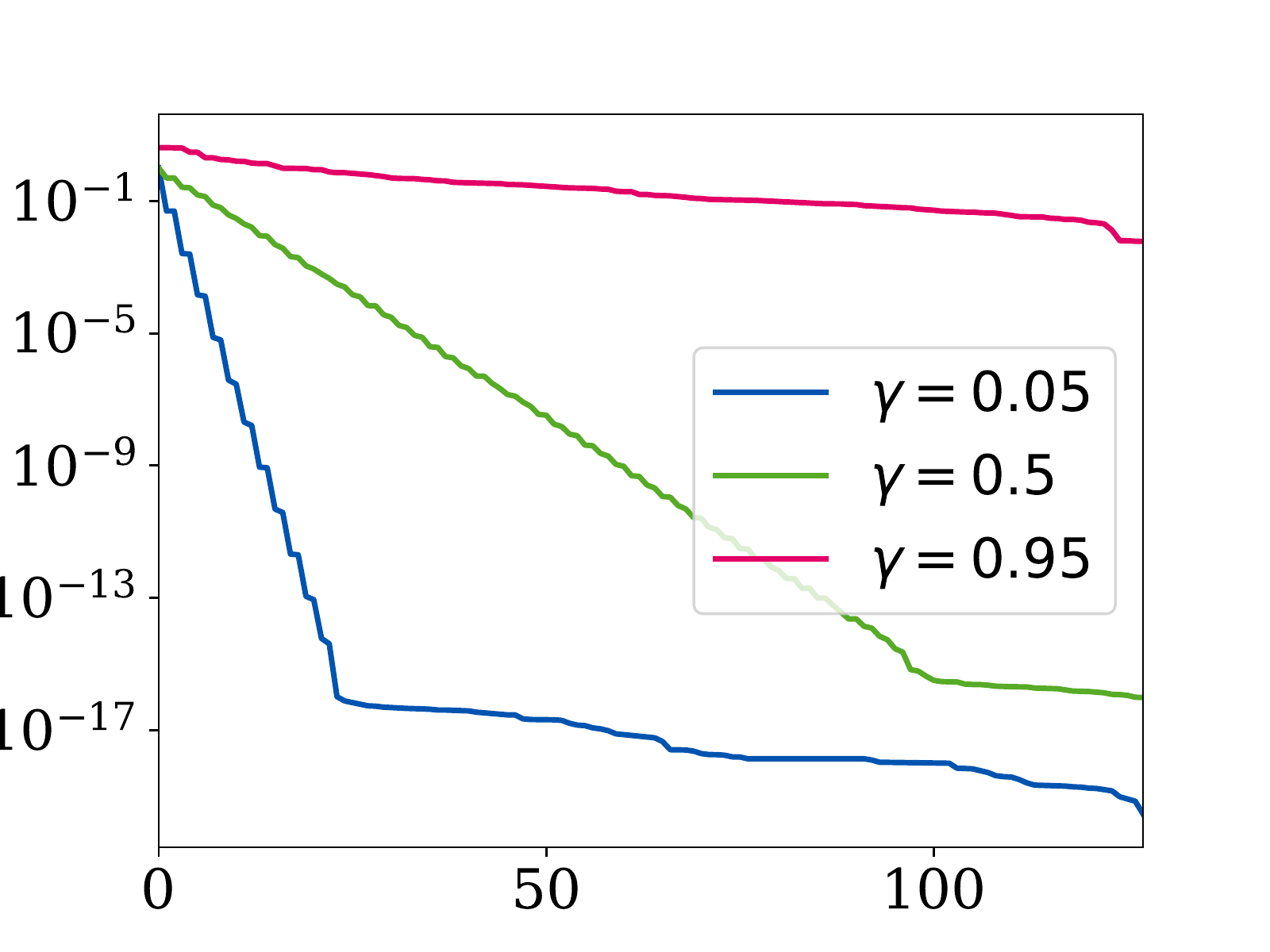}
\caption{SVD of the matrix representation $\bkk^{\Xi}$,
$\Xi\in \{\Psi,\Phi\}$ of $\kk$ for different values of $\gamma$. }
\label{fig:svd-G}
\end{minipage}
\end{figure}

\vspace*{-3mm}
Figure~\ref{fig:svd-G} shows the decay of singular values of a highly accurate
matrix representation $\bkk^{\Xi}$, $\Xi\in \{\Psi,\Phi\}$, of the scattering kernel $\kk$ for different values of $\gamma$.
For $\gamma$ close to one this decay is very slow but Figure \ref{fig:finger} reveals that the wavelet representation
is nevertheless extremely sparse.
Here we confine the subsequent discussion
to moderately isotropic scattering $\gamma = 0.5$. The singular values still decay rapidly (see Figure~\ref{fig:svd-G})
which allows us to apply the method outlined in Section \ref{ssec:HilbertSchmidt} based on Hilbert--Schmidt decompositions.
We present results with Alpert wavelets of degree 2.

We set $\e = 1.1\cdot 5.10^{-3}$ as the final target accuracy. The problem is of transport-dominated nature ($\rho\leq 1$) so we can solve it with the \ASTI~algorithm. Table \ref{tab:values-cts} gives the estimated values $C_{\Tr}, \rho, b_0(u)$ and $\kappa_1,\kappa_2,\kappa_3$. Note that $\kappa_2=0$ since we can evaluate the source term exactly. The remaining two parameters $\kappa_1$ and $\kappa_3$ balance the accuracy tolerances for the approximate application of
the scattering operator and the approximate inversion of $\Tr$. Specifically, $\kappa_1$ determines on the one hand the number of quadrature points
and hence the number of fiber transport problems to be solved and, on the other hand, $\kappa_3$ affects the spatial discretizations of these fiber problems.

\begin{table}[h]
\begin{center}
\begin{tabular}{ | c | c | c | c | c | c |}
\hline
$C_{\Tr}$  & $\rho$  & $\ub[0]$ & $\kappa_1$ & $\kappa_2$ & $\kappa_3$ \\
\hline
0.594604  & 0.594604 & 1/7 & $0.2 / C_\Tr$ & 0 & $0.8$ \\
\hline
\end{tabular}
\end{center}
\vspace{-0.2cm}
\caption{Values of the constants required to run the \ASTI~Algorithm \ref{alg:asti}.}
\label{tab:values-cts}
\end{table}

Figure~\ref{fig:checkerboardConvergence}, displays the convergence history and degrees of freedom for the above choice of parameters.
The left plot gives an approximation error of the scattering application $|| \cK(\bar u_n) - [\cK, \bar u_n; \kappa_1 \eta_n] ||_{L_2(\D\times \S)}$
(dark blue curve), the a posteriori error of the transport solves $|| u_n - \bar u_n ||_{L_2(\D\times\S)}$ (light blue curve), and
a bound for the global error $|| u - \bar u_n ||_{L_2(\D\times\S)}$ (purple curve) based on \eqref{purple}.
Recall that it is composed of the bounds for $\rho^n\|u\|_\U$ and the the above two error tolerances. By the definition \eqref{etan} of the
tolerances $\eta_n$, the interior solution accuracies need to be somewhat finer which explains
the gradual divergence between the global error bound and the interior error tolerances.
To avoid this would require total a posteriori bounds based on the bilinear form $b(w,v)= ((\Tr-\cK)(w))(v)$
in combination with coarsening strategies, which is the subject of future work. The shaded blue regions in the right plot
 indicate
 statistics about the number of degrees of freedom that are associated   for each selected angular direction.

\newcommand{\checkerboardconvplots}[1]{%
\begin{subfigure}[h]{\textwidth}
\hspace*{0.7cm}\includegraphics[width=.42\textwidth]
        {{figures/checkerboard_#1-conv}.pdf}
\hspace*{1.1cm}
\includegraphics[width=.42\textwidth]
        {{figures/checkerboard_#1-num-dofs}.pdf}
\label{fig:checkerboardConvergence_#1}
\end{subfigure}
}

\begin{figure}[h]
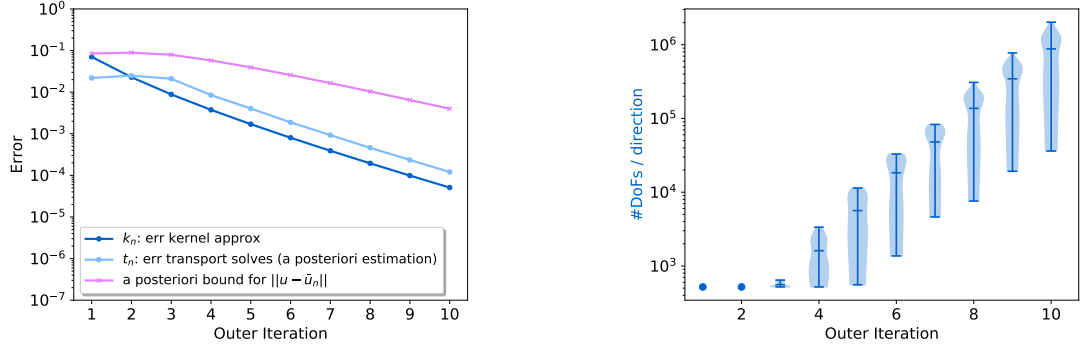

\centering
\checkerboardconvplots{0.2}
\caption{Convergence and number of DoFs for
         $\kappa_1 = \xi / C_\Tr$,
         $\kappa_2 = 0$,
         $\kappa_3 = (1-\xi) / 2$ with $\xi=0.2$.}
\label{fig:checkerboardConvergence}
\end{figure}

The table below gives the precise values of the a posteriori error and the total degrees of freedom:
\begin{center}
\begin{tabular}{r|l|r}
iteration & a posteriori error & \#DoFs \\
\hline
1 & 0.0850598 & 6228 \\
2 & 0.0891398 & 12456 \\
3 & 0.079258 & 13392 \\
4 & 0.0578653 & 38664 \\
5 & 0.039463 & 135236 \\
6 & 0.0258249 & 440648 \\
7 & 0.0165168 & 1151102 \\
8 & 0.010397 & 6586094 \\
9 & 0.00647563 & 16570210 \\
10 & 0.00400132 & 42179602 \\
\end{tabular}

\end{center}

Figure~\ref{fig:transport-solutions} shows solutions $\bar u_n(\cdot,\s)$ with their corresponding grids for the final iterate once the accuracy $\e$ has been reached. Finally, Figure~\ref{fig:checkerboardIntegratedSolutions} shows the final averaged densities $\int_\S \bar u_n(\cdot,\s) \,\ds$. They are computed on the merged grids.

We note that no special \emph{structure preserving} measures had to be imposed on the numerical schemes to produce physically meaningful results.

\begin{remark}
The code to reproduce the numerical part of this article is available online at:
\begin{center}
\url{https://gitlab.dune-project.org/felix.gruber/dune-dpg}
\end{center}
The implementation makes use of \textsc{Dune-DPG} 0.4.2, a C++ based library which is built upon the
multi-purpose finite element package \textsc{DUNE} \cite{dune2.4}.
Details of the \textsc{Dune-DPG} library can be found in \cite{GKM2017,Gruber2018}.
\end{remark}

\newcommand{\checkerboardsolutionplot}[1]{%
\begin{subfigure}{\textwidth}
\includegraphics[width=.41\textwidth]
        {figures/{checkerboard_integrated_n#1_0.2}.png}
~
\includegraphics[width=.41\textwidth]
        {figures/{checkerboard_integrated_n#1_0.2_grid}.png}
\caption{integrated solution and grid for iteration step #1.}
\label{fig:checkerboardIntegratedSolutions_#1_0.2}
\end{subfigure}%
}

\begin{figure}
\centering
\checkerboardsolutionplot{2}
\checkerboardsolutionplot{6}
\checkerboardsolutionplot{8}
\checkerboardsolutionplot{10}
\caption{Integrated solutions $\int_\S u_n(\cdot,\s) \,\ds$ and corresponding merged grids.}
\label{fig:checkerboardIntegratedSolutions}
\end{figure}

\begin{figure}
\centering
\includegraphics[width=.41\textwidth]
                {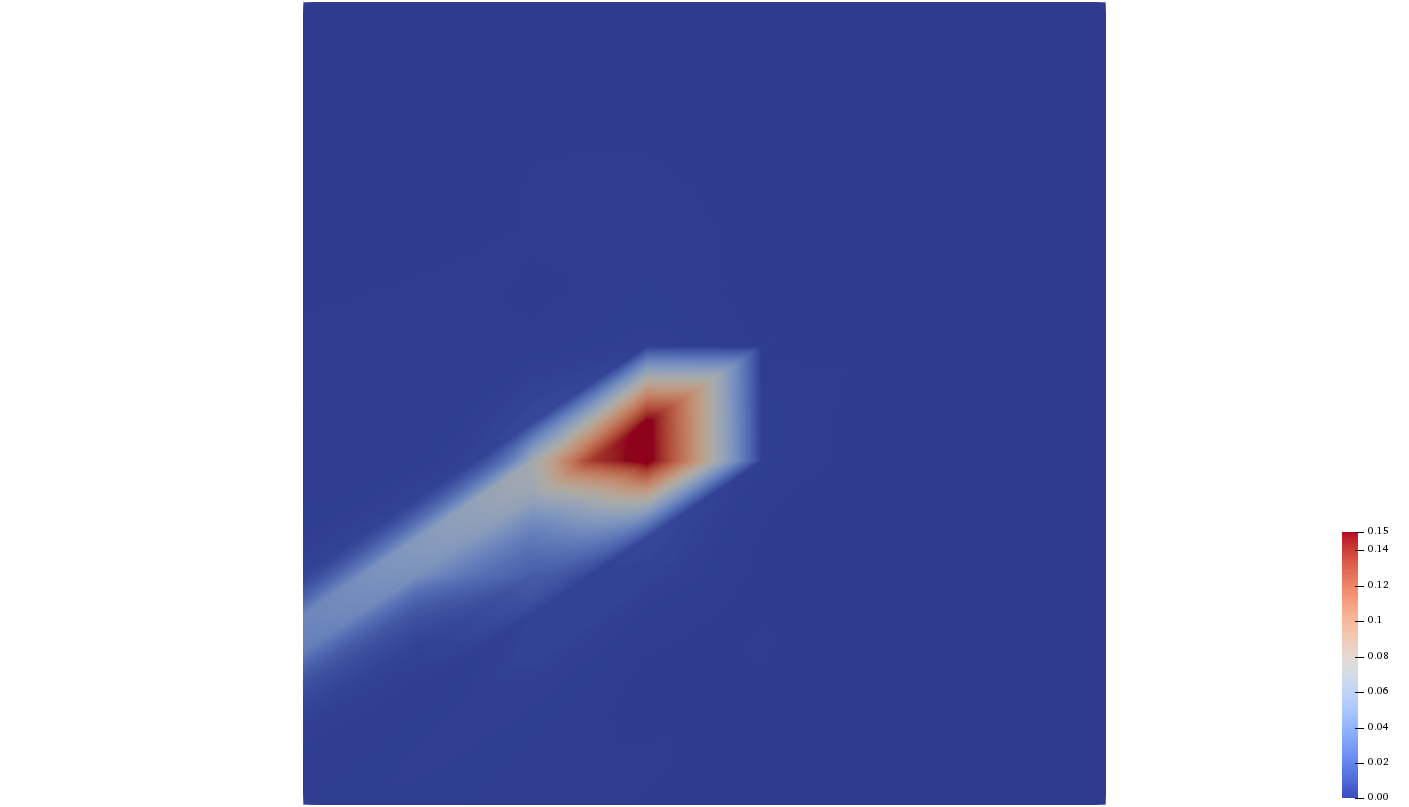} %
\includegraphics[width=.41\textwidth]
                {figures/{checkerboard_solution_n10_s5_0.2_grid}.png} \\[5mm]
\includegraphics[width=.41\textwidth]
                {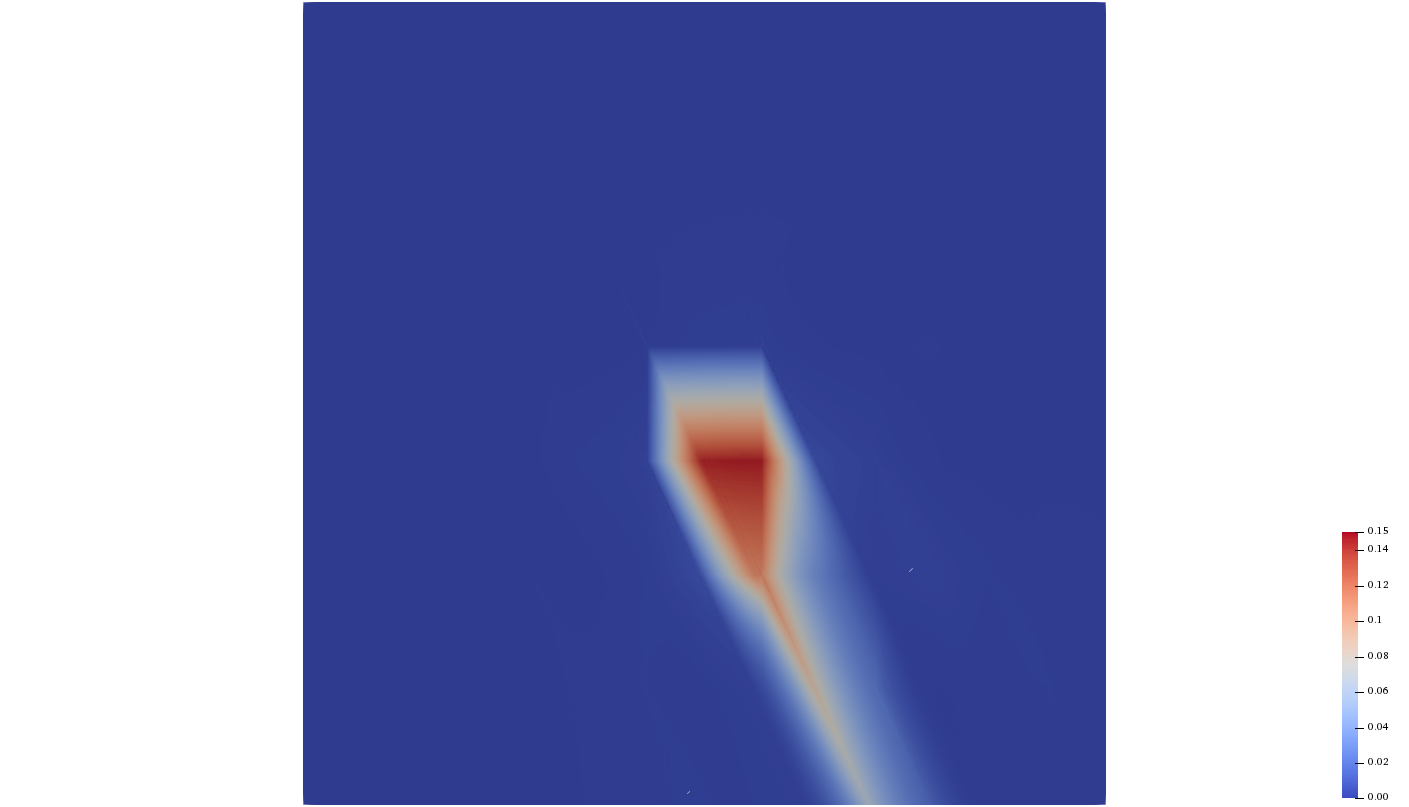} %
\includegraphics[width=.41\textwidth]
                {figures/{checkerboard_solution_n10_s16_0.2_grid}.png}
\caption{Solutions $\bar u_n(\cdot,\s)$ for different directions $\s$ in final outer iterate.}
\label{fig:transport-solutions}
\end{figure}

\newpage
\begin{small}
\renewcommand{\s}{\char"FF} 
\bibliographystyle{bibstyle}
\bibliography{literature}

\begin{thebibliography}{10}
\providecommand{\url}[1]{\texttt{#1}}
\providecommand{\urlprefix}{URL }

\bibitem{Alpert1993}
Alpert, B.~K.
\newblock A Class of Bases in $L^2$ for the Sparse Representation of Integral
  Operators.
\newblock \emph{SIAM Journal on Mathematical Analysis}, 24(1), pp. 246--262,
  1993.
\newblock \urlprefix\url{http://dx.doi.org/10.1137/0524016}.

\bibitem{Asadzadeh1988}
Asadzadeh, M.
\newblock $L_2$-Error Estimates for the Discrete Ordinates Method for
  Three-dimensional Neutron Transport.
\newblock \emph{Transport Theory and Statistical Physics}, 17(1), pp. 1--24,
  1988.
\newblock \urlprefix\url{http://dx.doi.org/10.1080/00411458808230852}.

\bibitem{Asadzadeh1998}
Asadzadeh, M.
\newblock A Finite Element Method for the Neutron Transport Equation in an
  Infinite Cylindrical Domain.
\newblock \emph{SIAM {J}ournal on {N}umerical {A}nalysis}, 35(4), pp.
  1299--1314, August 1998.
\newblock \urlprefix\url{http://dx.doi.org/10.1137/S0036142992238119}.

\bibitem{ACP2011}
Avila, M., Codina, R. and Principe, J.
\newblock Spatial Approximation of the Radiation Transport Equation Using a
  Subgrid-scale Finite Element Method.
\newblock \emph{Computer Methods in Applied Mechanics and Engineering},
  200(5–8), pp. 425--438, 2011.
\newblock \urlprefix\url{http://dx.doi.org/10.1016/j.cma.2010.11.003}.

\bibitem{Bal2009}
Bal, G.
\newblock Inverse Transport Theory and Applications.
\newblock \emph{Inverse Problems}, 25(5), 2009.
\newblock \urlprefix\url{http://dx.doi.org/10.1088/0266-5611/25/5/053001}.

\bibitem{dune2.4}
Blatt, M., Burchardt, A., Dedner, A., Engwer, C., Fahlke, J., Flemisch, B.,
  Gersbacher, C., Gräser, C., Gruber, F., Grüninger, C., Kempf, D.,
  Klöfkorn, R., Malkmus, T., Müthing, S., Nolte, M., Piatkowski, M. and
  Sander, O.
\newblock The {D}istributed and {U}nified {N}umerics {E}nvironment, Version
  2.4.
\newblock \emph{Archive of Numerical Software}, 4(100), pp. 13--29, May 2016.
\newblock \urlprefix\url{http://dx.doi.org/10.11588/ans.2016.100.26526}.

\bibitem{BDS2017}
Broersen, D., Dahmen, W. and Stevenson, R.~P.
\newblock On the Stability of {DPG} Formulations of Transport Equations.
\newblock \emph{Mathematics of Computation}, 2017.
\newblock \urlprefix\url{http://dx.doi.org/10.1090/mcom/3242}.

\bibitem{Brunner2002}
Brunner, T.~A.
\newblock Forms of Approximate Radiation Transport.
\newblock Sandia Report SAND2002-1778, Sandia National Laboratories, July 2002.
\newblock \urlprefix\url{http://dx.doi.org/10.2172/800993}.

\bibitem{CDD2001}
Cohen, A., Dahmen, W. and DeVore, R.
\newblock Adaptive Wavelet Methods for Elliptic Operator Equations: Convergence
  Rates.
\newblock \emph{Mathematics of Computation}, 70(233), pp. 27--75, 2001.
\newblock \urlprefix\url{http://dx.doi.org/10.1090/S0025-5718-00-01252-7}.

\bibitem{CDD2002}
Cohen, A., Dahmen, W. and DeVore, R.
\newblock Adaptive Wavelet Methods {II}---Beyond the Elliptic Case.
\newblock \emph{Foundations of Computational Mathematics}, 2(3), pp. 203--202,
  August 2002.
\newblock \urlprefix\url{http://dx.doi.org/10.1007/s102080010027}.

\bibitem{DHS2006}
Dahmen, W., Harbrecht, H. and Schneider, R.
\newblock Compression Techniques for Boundary Integral
  Equations---Asymptotically Optimal Complexity Estimates.
\newblock \emph{SIAM Journal on Numerical Analysis}, 43(6), pp. 2251--2271,
  2006.
\newblock \urlprefix\url{http://dx.doi.org/10.1137/S0036142903428852}.

\bibitem{DHSW2012}
Dahmen, W., Huang, C., Schwab, C. and Welper, G.
\newblock Adaptive {P}etrov--{G}alerkin Methods for First Order Transport
  Equations.
\newblock \emph{SIAM Journal on Numerical Analysis}, 50(5), pp. 2420--2445,
  2012.
\newblock \urlprefix\url{http://dx.doi.org/10.1137/110823158}.

\bibitem{DS2019}
Dahmen, W. and Stevenson, R.~P.
\newblock Adaptive Strategies for Transport Equations.
\newblock \emph{Comput. Meth. in Appl. Math.}, 19(3), pp. 431--464, 2019.
\newblock \urlprefix\url{http://dx.doi.org/10.1515/cmam-2018-0230}.

\bibitem{DL1993b}
Dautray, R. and Lions, J.-L.
\newblock \emph{Evolution Problems {II}}, volume~6 of \emph{Mathematical
  Analysis and Numerical Methods for Science and Technology}.
\newblock Springer, 1993.
\newblock \urlprefix\url{http://dx.doi.org/10.1007/978-3-642-58004-8}.

\bibitem{ES2012}
Egger, H. and Schlottbom, M.
\newblock A Mixed Variational Framework for the Radiative Transfer Equation.
\newblock \emph{Math. Mod. Meth. Appl.}, 22(03), 2012.
\newblock \urlprefix\url{http://dx.doi.org/10.1142/S021820251150014X}.

\bibitem{ES2014}
Egger, H. and Schlottbom, M.
\newblock An $L^p$ Theory for Stationary Radiative Transfer.
\newblock \emph{Applicable Analysis}, 93(6), pp. 1283--1296, April 2014.
\newblock \urlprefix\url{http://dx.doi.org/10.1080/00036811.2013.826798}.

\bibitem{GS2011b}
Grella, K. and Schwab, C.
\newblock Sparse Discrete Ordinates Method in Radiative Transfer.
\newblock \emph{Comp. Meth. in Applied Math.}, 11(3), pp. 305--326, September
  2011.
\newblock \urlprefix\url{http://dx.doi.org/10.2478/cmam-2011-0017}.

\bibitem{Gruber2018}
Gruber, F.
\newblock \emph{Adaptive Source Term Iteration: A Stable Formulation for
  Radiative Transfer}.
\newblock Ph.D. thesis, RWTH Aachen University, 2018.
\newblock \urlprefix\url{http://dx.doi.org/10.18154/RWTH-2018-230893}.

\bibitem{GKM2017}
Gruber, F., Klewinghaus, A. and Mula, O.
\newblock The {DUNE-DPG} Library for Solving {PDE}s with {Discontinuous
  Petrov--Galerkin} Finite Elements.
\newblock \emph{Archive of Numerical Software}, 5(1), pp. 111--128, 6~March
  2017.
\newblock \urlprefix\url{http://dx.doi.org/10.11588/ans.2017.1.27719}.

\bibitem{GK2010}
Guermond, J.-L. and Kanschat, G.
\newblock Asymptotic Analysis of Upwind Discontinuous {Galerkin} Approximation
  of the Radiative Transport Equation in the Diffusive Limit.
\newblock \emph{SIAM Journal on Numerical Analysis}, 48(1), pp. 53--78, 2010.
\newblock \urlprefix\url{http://dx.doi.org/10.1137/090746938}.

\bibitem{HG1941}
Henyey, L.~G. and Greenstein, J.~L.
\newblock Diffuse Radiation in the Galaxy.
\newblock \emph{The Astrophysical Journal}, 93, pp. 70--83, 1941.

\bibitem{JP1983}
Johnson, C. and Pitkäranta, J.
\newblock Convergence of a Fully Discrete Scheme for Two-Dimensional Neutron
  Transport.
\newblock \emph{SIAM Journal on Numerical Analysis}, 20(5), pp. 951--966,
  October 1983.
\newblock \urlprefix\url{http://dx.doi.org/10.1137/0720065}.

\bibitem{Kanschat2009}
Kanschat, G.
\newblock Solution of Radiative Transfer Problems with Finite Elements.
\newblock In Kanschat, Meinköhn, Rannacher and Wehrse (editors),
  \emph{Numerical Methods in Multidimensional Radiative Transfer}, pp. 49--98.
  Springer, 2009.
\newblock \urlprefix\url{http://dx.doi.org/10.1007/978-3-540-85369-5_5}.

\bibitem{MY2008}
Modest, M.~F. and Yang, J.
\newblock Elliptic {PDE} Formulation and Boundary Conditions of the Spherical
  Harmonics Method of Arbitrary Order for General Three-dimensional Geometries.
\newblock \emph{Journal of Quantitative Spectroscopy and Radiative Transfer},
  109(9), pp. 1641--1666, 2008.
\newblock \urlprefix\url{http://dx.doi.org/10.1016/j.jqsrt.2007.12.018}.

\bibitem{Kharroubi1997}
Mokhtar-Kharroubi, M.
\newblock \emph{Mathematical Topics in Neutron Transport Theory: New Aspects},
  volume~46 of \emph{Series on Advances in Mathematics for Applied Sciences}.
\newblock World Scientific, Singapore, 1997.
\newblock \urlprefix\url{http://dx.doi.org/10.1142/3288}.

\bibitem{RGK2012}
Ragusa, J.~C., Guermond, J.-L. and Kanschat, G.
\newblock A Robust {$S_N$-DG}-approximation for Radiation Transport in
  Optically Thick and Diffusive Regimes.
\newblock \emph{Journal of Computational Physics}, 231(4), pp. 1947--1962,
  2012.
\newblock \urlprefix\url{http://dx.doi.org/10.1016/j.jcp.2011.11.017}.

\end{thebibliography}
\end{small}
\end{document}